\documentclass[aos]{imsart}
\usepackage{amsmath,amssymb,bbm,amsthm,graphicx,subcaption,booktabs}
\usepackage[numbers]{natbib}
\RequirePackage[colorlinks,citecolor=blue,urlcolor=blue]{hyperref}

\startlocaldefs

\newcommand{\Ind}{1\!\mathrm{l}}
\theoremstyle{plain}
\newtheorem{theorem}{Theorem}[section]
\newtheorem{lemma}[theorem]{Lemma}

\newtheorem{proposition}[theorem]{Proposition}
\newtheorem{remark}[theorem]{Remark}

\theoremstyle{definition}
\newtheorem{definition}{Definition}[section]
\newtheorem{assumption}{Assumption}

\numberwithin{equation}{section}
\endlocaldefs

\arxiv{1411.6716}

\begin{document}
\begin{frontmatter}
\title{Supremum Norm Posterior Contraction and Credible Sets for Nonparametric Multivariate Regression}
\runtitle{Sup-norm Posterior Contraction and Credible Sets}

\begin{aug}
\author{\fnms{William Weimin} \snm{Yoo}\corref{}\ead[label=e1]{yooweimin0203@gmail.com}}
\address{W. W. Yoo\\Universit\'{e} Paris Dauphine\\CEREMADE\\Place du Mar\'{e}chal de Lattrie de Tassigny,\\75016 Paris, France\\\printead{e1}}
\and
\author{\fnms{Subhashis} \snm{Ghosal}\ead[label=e2]{sghosal@ncsu.edu}}
\address{S. Ghosal\\Department of Statistics\\North Carolina State University\\4276 SAS Hall, 2311 Stinson Drive\\Raleigh, North Carolina 27695-8203\\USA\\\printead{e2}}
\affiliation{Universit\'{e} Paris Dauphine and North Carolina State University}
\runauthor{W. W. Yoo and S. Ghosal}
\end{aug}

\begin{abstract}
In the setting of nonparametric multivariate regression with unknown error variance $\sigma^2$, we study asymptotic properties of a Bayesian method for  estimating a regression function $f$ and its mixed partial derivatives. We use a random series of tensor product of B-splines with normal basis coefficients as a prior for $f$, and $\sigma$ is either estimated using the empirical Bayes approach or is endowed with a suitable prior in a hierarchical Bayes approach. We establish pointwise, $L_2$ and $L_\infty$-posterior contraction rates for $f$ and its mixed partial derivatives, and show that they coincide with the minimax rates. Our results cover even the anisotropic situation, where the true regression function may have different smoothness in different directions. Using the convergence bounds, we show that pointwise, $L_2$ and $L_\infty$-credible sets for $f$ and its mixed partial derivatives have guaranteed frequentist coverage with optimal size. New results on tensor products of B-splines are also obtained in the course.
\end{abstract}


\end{frontmatter}

\section{Introduction}
Consider the nonparametric regression model
\begin{align}\label{eq:mainprob}
Y_i=f(\boldsymbol{X}_i)+\varepsilon_i,\quad i=1,\dotsc,n,
\end{align}
where $Y_i$ is a response variable, $\boldsymbol{X}_i$ is a $d$-dimensional covariate, and $\varepsilon_1,\ldots,\varepsilon_n$ are independent and identically distributed (i.i.d.) as $\mathrm{N}(0,\sigma^2)$ with unknown $0<\sigma<\infty$. The covariates are deterministic or are sampled from some fixed distribution independent of $\varepsilon_i$. In both cases, each $\boldsymbol{X}_i$ takes values in some rectangular region in $\mathbb{R}^d$, which is assumed to be $[0,1]^d$ without loss of generality. We follow the Bayesian approach by representing $f$ by a finite linear combination of tensor products of B-splines and endowing the coefficients with a multivariate normal prior. We consider both the empirical and the hierarchical Bayes approach for the variance $\sigma^2$. For the latter approach, a conjugate inverse-gamma prior is particularly convenient.

We study frequentist behavior of the posterior distributions and the resulting credible sets for $f$ and its mixed partial derivatives, in terms of pointwise, $L_2$ and $L_\infty$ (supremum) distances. We assume that the true regression function $f_0$ belongs to an anisotropic H\"{o}lder space (see Definition \ref{definition:aninorm} below), and the errors under the true distribution are sub-Gaussian.

Posterior contraction rates for regression functions in the $L_2$-norm are well studied, but results for the stronger $L_\infty$-norm are limited. Gin\'{e} and Nickl \citep{nickl2011} studied contraction rates in $L_r$-metric, $1\leq r\leq\infty$, and obtained optimal rate using conjugacy for the Gaussian white noise model and a rate for density estimation based on a random wavelet series and Dirichlet process mixture using a testing approach. In the same context, Castillo \citep{castillosupnorm} introduced techniques based on semiparametric Bernstein-von Misses (BvM) theorems to obtain optimal $L_\infty$-contraction rates. Hoffman et al. \citep{adaptsupnorm} derived adaptive optimal $L_\infty$-contraction rate for the white noise model and also for density estimation. Scricciolo \citep{scricciolo2014} applied the techniques of \citep{nickl2011} to obtain $L_\infty$-rates using Gaussian kernel mixtures prior for analytic true densities.

De Jonge and van Zanten \citep{dejonge2012} used finite random series based on tensor products of B-splines to construct a prior for nonparametric regression and derived adaptive $L_2$-contraction rate for the regression function in the isotropic case. A BvM theorem for the posterior of $\sigma$ is treated in \citep{dejonge2013}. Shen and Ghosal \cite{ShenandGhosal:RandomSeries, ShenandGhosal:DensityRegression} used tensor products of B-splines respectively for Bayesian multivariate density estimation and high dimensional density regression in the anisotropic case.

Nonparametric confidence bands for an unknown function were considered by  \citep{smirnov1950,bickel1973} and more recently by \citep{bootstrap2003,nickl2010,gaussbootstrap2013}. A Bayesian approaches the problem by constructing a credible set with a prescribed posterior probability. It is then natural to ask if the credible set has adequate frequentist coverage for large sample sizes. For parametric problems, a BvM theorem concludes that Bayesian and frequentist measures of uncertainly are nearly the same in large samples. However, for the infinite dimensional normal mean model (equivalently the Gaussian white noise model), \citep{cox1993,freedman1999} observed that for many true parameters in $\ell_2$, credible regions can have inadequate coverage. Leahu \citep{leahu2011} showed that if the prior variances are chosen very big so that the support of the prior extends beyond $\ell_2$, then coverage can be obtained. Knapik et al. \citep{inverseprob,inverseprob2013} showed that for sequences with specific smoothness, by deliberately undersmoothing the prior, coverage of credible sets may be guaranteed. Sniekers and van der Vaart \citep{credible} obtained similar results for nonparametric regression using a scaled Brownian motion prior.

Castillo and Nickl \citep{castillo2013} showed that for the Gaussian white noise model a BvM theorem can hold in weaker topologies for some natural priors, and the resulting credible sets appropriately modified will have asymptotically the correct coverage and optimal size. A similar result for the stronger $L_\infty$-norm using this weak notion of BvM theorem is considered in \citep{nickl2014}. Adaptive $L_2$-credible regions with adequate frequentist coverage are constructed using the empirical Bayes approach in \citep{botond2015} for the Gaussian white noise model and in \citep{smooth} for the nonparametric regression model using smoothing splines. In the setting of the Gaussian white noise model, Ray \citep{kolyan2015} constructed adaptive $L_2$-credible sets using a weak BvM theorem, and also adaptive $L_\infty$-credible band using a spike and slab prior.

In this paper we consider multivariate nonparametric regression with unknown variance parameter and study posterior contraction rates and coverage of credible sets in the pointwise, $L_2$- and $L_\infty$-senses, for the regression function $f$ as well as its mixed partial derivatives. Study of posterior contraction rate in $L_\infty$-norm is important for its natural interpretation and implications for other problems such as the convergence of the mode of a function. A $L_\infty$-credible band is easier to visualize than a $L_2$-credible set.
We assume that the smoothness of the function is given but allow anisotropy, so the smoothness level may vary with the direction. Anisotropic function has applications in estimating time-dependent spectral density of a locally stationary time series (see \citep{aniappli}), and variable selection (see \citep{hoffmann}).

A prior on the regression function is constructed using a finite random series of tensor products of B-splines with normally distributed coefficients. Posterior conjugacy leads to explicit expression for the posterior distribution which is convenient for computation as well as theoretical analysis.
Although wavelets are also widely used to construct random series priors, B-splines have the added advantage in that mixed partial derivatives of $f$ are expressible in terms of lower degree B-splines. This allows posterior analysis for mixed partial derivatives of $f$, a topic that is largely unaddressed in the literature, except implicitly as inverse problems in the Gaussian white noise model.

The paper is organized as follows. The next section introduces notations and assumptions. Section \ref{sec:tprior} describes the prior and the resulting posterior distribution. Section \ref{sec:tsupnormf} contains main results on pointwise and $L_\infty$-contraction rates of $f$ and its mixed partial derivatives. Section \ref{sec:cred} presents results on coverage of the corresponding credible sets. Section \ref{sec:sim} contains a simulation study of the proposed method. Proofs are in Section \ref{sec:proof}. New results on tensor products of B-splines are presented in the Appendix.

\section{Assumptions and preliminaries}
\label{sec:notation}

We describe notations and assumptions used in this paper. Given two numerical sequences $a_n$ and $b_n$, $a_n=O(b_n)$ or $a_n\lesssim b_n$ means $a_n/b_n$ is bounded, while $a_n=o(b_n)$ or $a_n\ll b_n$ means $a_n/b_n\rightarrow0$. Also, $a_n\asymp b_n$ means $a_n=O(b_n)$ and $b_n=O(a_n)$. For stochastic sequence $Z_n$, $Z_n=O_{\mathrm{P}}(a_n)$ means $\mathrm{P}(|Z_n|\leq Ca_n)\rightarrow1$ for some constant $C>0$. Let $\mathbb{N}=\{1,2,\dotsc\}$ and $\mathbb{N}_0=\mathbb{N}\cup\{0\}$.

Define $\|\boldsymbol{x}\|_p=(\sum_{k=1}^d |x_k|^p)^{1/p}$, $1\le p<\infty$, $\|\boldsymbol{x}\|_\infty=\max_{1\leq k\leq d}|x_k|$, and write $\|\boldsymbol{x}\|$ for $\|\boldsymbol{x}\|_2$, the Euclidean norm. We write $\boldsymbol{x}\leq\boldsymbol{y}$ if $x_k\leq y_k, k=1,\dotsc,d$. For an $m\times m$ matrix $\boldsymbol{A}= (\!( a_{ij})\!)$, let $\lambda_\mathrm{min}(\boldsymbol{A})$ and $\lambda_\mathrm{max}(\boldsymbol{A})$ be the smallest and largest eigenvalues, and the $(r,s)$ matrix norm of $\boldsymbol{A}$ as $\|\boldsymbol{A}\|_{(r,s)}=\sup\{\|\boldsymbol{Ax}\|_s:\|\boldsymbol{x}\|_r\leq1\}$. In particular, $\|\boldsymbol{A}\|_{(2,2)}=|\lambda_\text{max}(\boldsymbol{A})|$ and $\|\boldsymbol{A}\|_{(\infty,\infty)}=\max_{1\leq i\leq m}\sum_{j=1}^m|a_{ij}|$. These norms are related by $|a_{ij}|\leq\|\boldsymbol{A}\|_{(2,2)}\leq\|\boldsymbol{A}\|_{(\infty,\infty)}$ for $1\leq i,j\leq m$. With another matrix $\boldsymbol{B}$ of the same size, $\boldsymbol{A}\leq\boldsymbol{B}$ means $\boldsymbol{B}-\boldsymbol{A}$ is non-negative definite. We denote by $\boldsymbol{I}_m$ the $m\times m$ identity matrix and by $\boldsymbol{1}_d$ the $d\times1$ vector of ones.

For $f:U\rightarrow\mathbb{R}$ on some bounded set $U\subseteq\mathbb{R}^d$ with interior points, let $\|f\|_p$ be the $L_p$-norm, and $\|f\|_\infty=\sup_{x\in U}|f(x)|$. For $\boldsymbol{r}=(r_1,\dotsc,r_d)^T\in\mathbb{N}_0^d$, let $D^{\boldsymbol{r}}$ be the partial derivative operator $\partial^{|\boldsymbol{r}|}/\partial x_1^{r_1}\dotsm\partial x_d^{r_d}$, where $|\boldsymbol{r}|=\sum_{k=1}^dr_k$. If $\boldsymbol{r}=\boldsymbol{0}$, we interpret $D^{\boldsymbol{0}}f\equiv f$. We say $\boldsymbol{Z}\sim\mathrm{N}_J(\boldsymbol{\xi},\boldsymbol{\Omega})$ if $\boldsymbol{Z}$ has a $J$-dimensional normal distribution with mean vector $\boldsymbol{\xi}$ and covariance matrix $\boldsymbol{\Omega}$. For a random function $\{Z(t),t\in U\}$, write $Z\sim\mathrm{GP}(\xi,\Omega)$ if $Z$ is a Gaussian process with $\mathrm{E}Z(t)=\xi(t)$ and $\mathrm{Cov}(Z(s),Z(t))=\Omega(s,t)$.

\begin{definition}\label{definition:aninorm}
The anisotropic H\"{o}lder space $\mathcal{H}^{\boldsymbol{\alpha}}([0,1]^d)$ of order $\boldsymbol{\alpha}=(\alpha_1,\dotsc,\alpha_d)^T$ consists of functions $f:[0,1]^d\rightarrow\mathbb{R}$ such that $\|f\|_{\boldsymbol{\alpha},\infty}<\infty$, where $\|\cdot\|_{\boldsymbol{\alpha},\infty}$ is the anisotropic H\"{o}lder norm
\begin{align}\label{eq:aninorm}
\max\left\{\|D^{\boldsymbol{r}}f\|_\infty+\sum_{k=1}^d\left\|D^{(\alpha_k-r_k)\boldsymbol{e}_k}D^{\boldsymbol{r}}f\right\|_\infty:\; \boldsymbol{r}\in\mathbb{N}_0^d,\,  \sum_{k=1}^dr_k/\alpha_k<1\right\}
\end{align}
and $\boldsymbol{e}_k \in\mathbb{R}^d$ has $1$ in the $k$th position and zero elsewhere.
\end{definition}

Let $\alpha^{*}$ be the harmonic mean of $(\alpha_1,\dotsc,\alpha_d)^T$, i.e., $\alpha^{*-1}=d^{-1}\sum_{k=1}^d\alpha_k^{-1}$. For $\boldsymbol{x}=(x_1,\dotsc,x_d)^T$, we define $\boldsymbol{b}_{\boldsymbol{J},\boldsymbol{q}}(\boldsymbol{x})=(B_{j_1,q_1}(x_1)\dotsm B_{j_d,q_d}(x_d),1\leq j_k\leq J_k, k=1,\dotsc,d$) to be a collection of $J=\prod_{k=1}^dJ_k$ tensor-product of B-splines, where $B_{j_k,q_k}(x_k)$ is the $k$th component B-spline of fixed order $q_k\geq\alpha_k$, with knot sequence $0=t_{k,0}<t_{k,1}<\dotsb<t_{k,N_k}<t_{k,N_{k+1}}=1$, and let $J_k=q_k+N_k$ and $\boldsymbol{J}=(J_1,\ldots,J_d)^T$. In the prior construction the knots depend on $n$ and $N_k$ increases to infinity with $n$ subject to  $\prod_{k=1}^dJ_k\leq n$. At each $k=1,\dotsc,d$, define $\delta_{k,l}=t_{k,l}-t_{k,l-1}$ to be the one-step knot increment, and let $\Delta_k=\max_{1\leq l\leq N_k}\delta_{k,l}$ be the mesh size. We assume that the knot sequence for each direction is quasi-uniform [Definition 6.4 of \citep{lschumaker}], that is $\Delta_k/\min_{1\leq l\leq N_k}\delta_{k,l}\leq C$, for some $C>0$. This assumption is satisfied for the uniform and nested uniform partitions as special cases (Examples 6.6 and 6.7 of \citep{lschumaker}) and we can choose a subset of knots from any given knot sequence to form a quasi-uniform sequence with $C=3$ [Lemma 6.17 of \citep{lschumaker}].

If the design points $\boldsymbol{X}_i=(X_{i1},\dotsc,X_{id})^T$ for $i=1,\dotsc,n$, are fixed, assume that there exists a cumulative distribution function $G(\boldsymbol{x})$, with positive and continuous density on $[0,1]^d$ such that
\begin{equation}\label{assump:tbspline2}
\sup_{\boldsymbol{x}\in[0,1]^d}|G_n(\boldsymbol{x})-G(\boldsymbol{x})|=o\left(\prod_{k=1}^dN_k^{-1}\right),
\end{equation}
where $G_n(\boldsymbol{x})=n^{-1}\sum_{i=1}^n
 \Ind_{\prod_{k=1}^d[0,X_{ik}]}(\boldsymbol{x})$ is the empirical distribution of
  $\{\boldsymbol{X}_i,i=1,\ldots,n\}$, with $\Ind_U(\cdot)$ the indicator function on $U$.

\begin{remark}\rm
For example, let $n=m^d$ for some $m\in\mathbb{N}$, the discrete uniform design $\boldsymbol{X}_i\in\{(j-1)/(m-1):j=1,\dotsc,m\}^d$ with $i=1,\dotsc,n$, satisfies \eqref{assump:tbspline2} with $G$ being the uniform distribution on $[0,1]^d$ and $N_k\lesssim n^{\alpha^{*}/\{\alpha_k(2\alpha^{*}+d)\}}$ for $k=1,\dotsc,d$.
\end{remark}

For random design points, assume $\boldsymbol{X}_i\stackrel{\mathrm{i.i.d.}}{\sim}G$ with a continuous density on $[0,1]^d$, then \eqref{assump:tbspline2} holds with probability tending to one if $N_k\lesssim n^{\alpha^{*}/\{\alpha_k(2\alpha^{*}+d)\}}$ for $k=1,\dotsc,d$, and $\alpha^{*}>d/2$ by Donsker's theorem. In this paper, we shall prove results on posterior contraction rates and credible sets based on fixed design points; the random case can be treated by conditioning on $\boldsymbol{X}_i,i=1,\dotsc,n$.

Let $\boldsymbol{B}=(\boldsymbol{b}_{\boldsymbol{J},\boldsymbol{q}}(\boldsymbol{X}_1),\dotsc,\boldsymbol{b}_{\boldsymbol{J},\boldsymbol{q}}(\boldsymbol{X}_n))^T$. Each entry of $\boldsymbol{B}^T\boldsymbol{B}$ is indexed by $d$-dimensional multi-indices, i.e., for $\boldsymbol{u}=(u_1,\dotsc,u_d)^T$ and $\boldsymbol{v}=(v_1,\dotsc,v_d)^T$ with $1\leq u_k,v_k\leq J_k, k=1,\dotsc,d$, the $(\boldsymbol{u},\boldsymbol{v})$th entry is  $(\boldsymbol{B}^T\boldsymbol{B})_{\boldsymbol{u},\boldsymbol{v}}=\sum_{i=1}^n\prod_{k=1}^dB_{u_k,q_k}(X_{ik})B_{v_k,q_k}(X_{ik})$. The following generalization of matrix banding property will be useful.

\begin{definition}\label{def:band}
Let $\boldsymbol{A}=(\!(a_{\boldsymbol{u},\boldsymbol{v}})\!)$ be a matrix with rows and columns indexed by $d$-dimensional multi-indices $\boldsymbol{1}_d\leq\boldsymbol{u},\boldsymbol{v}\leq\boldsymbol{J}$ respectively, where arrangement of the elements are arbitrary. We say that $\boldsymbol{A}$ is $\boldsymbol{h}=(h_1,\dotsc,h_d)^T$ banded if $a_{\boldsymbol{u},\boldsymbol{v}}=0$ whenever $|u_k-v_k|>h_k$ for some $k=1,\dotsc,d$.
\end{definition}

Given $\boldsymbol{X}_i=(X_{i1},\dotsc,X_{id})^T$ for $i=1,\dotsc,n$, such that $X_{ik}\in[t_{k,l-1},t_{k,l}]$, only $q_k$ adjacent basis functions $(B_{l,q_k}(X_{ik}),\dotsc,B_{l+q_k-1,q_k}(X_{ik}))^T$ will be nonzero for $k=1,\dotsc,d$. Hence if $|u_m-v_m|>q_m$ for some $m=1,\dotsc,d$, then $B_{u_m,q_m}(X_{im})B_{v_m,q_m}(X_{im})=0$, and we conclude $(\boldsymbol{B}^T\boldsymbol{B})_{\boldsymbol{u},\boldsymbol{v}}=0$. It then follows that $\boldsymbol{B}^T\boldsymbol{B}$ is $\boldsymbol{q}=(q_1,\dotsc,q_d)^T$-banded.

Since approximation results for anisotropic functions by linear combinations of tensor-products of B-splines assume integer smoothness (see Chapter 12, Section 3 of \citep{lschumaker}), we assume that $\boldsymbol{\alpha}\in\mathbb{N}^d$. For the isotropic case, the norm in \eqref{eq:aninorm} can be generalized (see Section 2.7.1 of \citep{empirical}) and the approximation rate is obtained for all smoothness levels (Theorem 22 of Chapter XII in \citep{deBoor}). This allows generalization of posterior contraction results for arbitrary smoothness levels. We now describe the assumption on $f_0$ used in this paper.

\begin{assumption}\label{assump:tf0}
Under the true distribution $P_0$, $Y_i=f_0(\boldsymbol{X}_i)+\varepsilon_i$, such that $\varepsilon_i$ are i.i.d. sub-Gaussian with mean $0$ and variance $\sigma_0^2$ for $i=1,\dotsc,n$. Also, $f_0\in\mathcal{H}^{\boldsymbol{\alpha}}([0,1]^d)$ with order $\boldsymbol{\alpha}=(\alpha_1,\dotsc,\alpha_d)^T\in\mathbb{N}^d$. If the design points are deterministic, we assume that \eqref{assump:tbspline2}  holds. If the design points are random, we assume that $\alpha^*>d/2$.
\end{assumption}

Let $\mathrm{E}_0(\cdot)$ and $\mathrm{Var}_0(\cdot)$ be the expectation and variance operators taken with respect to $P_0$. We write $\boldsymbol{Y}=(Y_1,\dotsc,Y_n)^T$, $\boldsymbol{X}=(\boldsymbol{X}_1^T,\dotsc,\boldsymbol{X}_n^T)^T$, $\boldsymbol{F}_0=(f_0(\boldsymbol{X}_1),\dotsc,f_0(\boldsymbol{X}_n))^T$ and $\boldsymbol{\varepsilon}=(\varepsilon_1,\dotsc,\varepsilon_n)^T$.

\section{Prior and posterior conjugacy}\label{sec:tprior}

We induce a prior on $f$ by representing it as a tensor-product B-splines series, i.e., $f(\boldsymbol{x})=\boldsymbol{b}_{\boldsymbol{J},\boldsymbol{q}}(\boldsymbol{x})^T\boldsymbol{\theta}$, where $\boldsymbol{\theta}=\{\theta_{j_1,\dotsc,j_d}:1\leq j_k\leq J_k,k=1,\dotsc,d\}$ are the basis coefficients. Then its $\boldsymbol{r}=(r_1,\dotsc,r_d)^T$ mixed partial derivative is
\begin{equation*}
D^{\boldsymbol{r}}f(\boldsymbol{x})=\sum_{j_1=1}^{J_1}\dotsi\sum_{j_d=1}^{J_d}\theta_{j_1,\dotsc,j_d}\prod_{k=1}^d\frac{\partial^{r_k}}{\partial x_k^{r_k}}B_{j_k,q_k}(x_k).
\end{equation*}
Define an operator $\mathfrak{D}_{j_k}^{r_k}$ acting on $\theta_{j_1,\dotsc,j_d}$ such that $\mathfrak{D}_{j_k}^{0}\theta_{j_1,\dotsc,j_d}=\theta_{j_1,\dotsc,j_d}$, and for $r_k\geq1$,
\begin{equation}\label{eq:firstdiff}
\mathfrak{D}_{j_k}^{r_k}\theta_{j_1,\dotsc,j_d}=\frac{\mathfrak{D}_{j_k}^{r_k-1}\theta_{j_1,\dotsc,j_{k-1},j_k+1,j_{k+1},\dotsc,j_d}-\mathfrak{D}_{j_k}^{r_k-1}\theta_{j_1,\dotsc,j_{k-1},j_k,j_{k+1},\dotsc,j_d}}{(t_{k,j_k}-t_{k,j_k-q_k+1})/(q_k-r_k)}.
\end{equation}
Furthermore, let $\mathfrak{D}^{\boldsymbol{r}}\theta_{j_1,\dotsc,j_d}=\mathfrak{D}^{r_1}_{j_1}\dotsm\mathfrak{D}^{r_d}_{j_d}\theta_{j_1,\dotsc,j_d}$ be the application of $\mathfrak{D}^{r_k}_{j_k}$ to $\theta_{j_1,\dotsc,j_d}$ for all direction $k=1,\dotsc,d$. Using equations (15) and (16) of Chapter X from \citep{deBoor}, $D^{\boldsymbol{r}}f(\boldsymbol{x})$ can be written as
\begin{equation}\label{eq:derivspline}
\sum_{j_1=1}^{J_1-r_1}\dotsi\sum_{j_d=1}^{J_d-r_d}\mathfrak{D}^{\boldsymbol{r}}\theta_{j_1,\dotsc,j_d}\prod_{k=1}^dB_{j_k,q_k-r_k}(x_k)
=\boldsymbol{b}_{\boldsymbol{J},\boldsymbol{q}-\boldsymbol{r}}(\boldsymbol{x})^T\boldsymbol{W}_{\boldsymbol{r}}\boldsymbol{\theta},
\end{equation}
where $\boldsymbol{W}_{\boldsymbol{r}}$ is a $\prod_{k=1}^d(J_k-r_k)\times\prod_{k=1}^dJ_k$ matrix, with entries given by \eqref{eq:tendpoints}--\eqref{eq:twijmiddle} in Lemma \ref{lem:Wr}. These entries are coefficients associated with applying the weighted finite differencing operator of \eqref{eq:firstdiff} iteratively on $\boldsymbol{\theta}$ in all directions.

We represent the model in \eqref{eq:mainprob} by $\boldsymbol{Y}|(\boldsymbol{X},\boldsymbol{\theta},\sigma)\sim\mathrm{N}_n(\boldsymbol{B\theta},\sigma^2\boldsymbol{I}_n)$. In this paper, we treat $\boldsymbol{J}=(J_1,\dotsc,J_d)^T$ as deterministic and allow it to depend on $n, d$ and $\boldsymbol{\alpha}$. On the basis coefficients, we assign $\boldsymbol{\theta}|\sigma\sim\mathrm{N}_J(\boldsymbol{\eta},\sigma^2\boldsymbol{\Omega})$, where $\|\boldsymbol{\eta}\|_\infty$ is uniformly bounded. The entries of $\boldsymbol{\Omega}$ do not depend on $n$, and are indexed using $d$-dimensional multi-indices described above. We further assume that $\boldsymbol{\Omega}^{-1}$ is a $\boldsymbol{m}=(m_1,\dotsc,m_d)^T$ banded matrix with fixed $\boldsymbol{m}$. Note that $\boldsymbol{\Omega}$ depends on $n$ only through its dimension $J\times J$. Furthermore, as $n\rightarrow\infty$, we assume that there exists constants $0<c_1\leq c_2<\infty$ such that
\begin{equation}\label{assump:tprior}
c_1\boldsymbol{I}_J\leq\boldsymbol{\Omega}\leq c_2\boldsymbol{I}_J.
\end{equation}
It follows that $D^{\boldsymbol{r}}f|(\boldsymbol{Y},\sigma)\sim\mathrm{GP}(\boldsymbol{A}_{\boldsymbol{r}}\boldsymbol{Y}+\boldsymbol{c}_{\boldsymbol{r}}\boldsymbol{\eta},\sigma^2\Sigma_{\boldsymbol{r}})$, where $\boldsymbol{A}_{\boldsymbol{r}}$,  $\boldsymbol{c}_{\boldsymbol{r}}$ and the covariance kernel are defined for  $\boldsymbol{x},\boldsymbol{y}\in[0,1]^d$ by
\begin{align}
\boldsymbol{A}_{\boldsymbol{r}}(\boldsymbol{x})&=\boldsymbol{b}_{\boldsymbol{J},\boldsymbol{q}-\boldsymbol{r}}(\boldsymbol{x})^T\boldsymbol{W}_{\boldsymbol{r}}\left(\boldsymbol{B}^T\boldsymbol{B}+\boldsymbol{\Omega}^{-1}\right)^{-1}\boldsymbol{B}^T,\label{eq:tAr}\\
\boldsymbol{c}_{\boldsymbol{r}}(\boldsymbol{x})&=\boldsymbol{b}_{\boldsymbol{J},\boldsymbol{q}-\boldsymbol{r}}(\boldsymbol{x})^T\boldsymbol{W}_{\boldsymbol{r}}\left(\boldsymbol{B}^T\boldsymbol{B}+\boldsymbol{\Omega}^{-1}\right)^{-1}\boldsymbol{\Omega}^{-1},\label{eq:tcr}\\
\Sigma_{\boldsymbol{r}}(\boldsymbol{x},\boldsymbol{y})&=\boldsymbol{b}_{\boldsymbol{J},\boldsymbol{q}-\boldsymbol{r}}(\boldsymbol{x})^T\boldsymbol{W}_{\boldsymbol{r}}
\left(\boldsymbol{B}^T\boldsymbol{B}+\boldsymbol{\Omega}^{-1}\right)^{-1}\boldsymbol{W}_{\boldsymbol{r}}^T\boldsymbol{b}_{\boldsymbol{J},\boldsymbol{q}-\boldsymbol{r}}
(\boldsymbol{y}).\label{eq:tsigmar}
\end{align}
Since the posterior mean of $D^{\boldsymbol{r}}f$ is an affine transformation of $\boldsymbol{Y}$, Assumption \ref{assump:tf0} implies that  $\boldsymbol{A}_{\boldsymbol{r}}\boldsymbol{Y}+\boldsymbol{c}_{\boldsymbol{r}}\boldsymbol{\eta}$ is a sub-Gaussian process under $P_0$. If $\boldsymbol{r}=\boldsymbol{0}$,  defining $\boldsymbol{W}_{\boldsymbol{0}}=\boldsymbol{I}_J$, we obtain the conditional posterior distribution of $f$.

To deal with $\sigma$, observe that $\boldsymbol{Y}|\sigma\sim\mathrm{N}_n[\boldsymbol{B\eta},\sigma^2(\boldsymbol{B\Omega B}^T+\boldsymbol{I}_n)]$. Maximizing the corresponding log-likelihood with respect to $\sigma$ leads to
\begin{align}\label{eq:tsigma2hat}
\widehat{\sigma}_n^2&=n^{-1}{(\boldsymbol{Y}-\boldsymbol{B\eta})^T(\boldsymbol{B\Omega B}^T+\boldsymbol{I}_n)^{-1}(\boldsymbol{Y}-\boldsymbol{B\eta})}.
\end{align}
Empirical Bayes then entails substituting the maximum likelihood estimator $\widehat{\sigma}_n$ for $\sigma$ in the conditional posterior of $D^{\boldsymbol{r}}f$, i.e.,
\begin{align}
\Pi(D^{\boldsymbol{r}}f|\boldsymbol{Y},\sigma)|_{\sigma=\widehat{\sigma}_n}=\Pi_{\widehat{\sigma}_n}(D^{\boldsymbol{r}}f|\boldsymbol{Y})\sim \mathrm{GP}(\boldsymbol{A}_{\boldsymbol{r}}\boldsymbol{Y}+\boldsymbol{c}_{\boldsymbol{r}}\boldsymbol{\eta},\widehat{\sigma}_n^2\Sigma_{\boldsymbol{r}}).
\end{align}
In a hierarchical Bayes approach, we further endow $\sigma$ with a continuous and positive prior density. A conjugate inverse-gamma (IG) prior $\sigma^2\sim\mathrm{IG}(\beta_1/2,\beta_2/2)$, with hyperparameters $\beta_1>4$ and $\beta_2>0$ is particularly convenient for both computation and theoretical analysis since  by direct calculations, the posterior of $\sigma^2$ is
\begin{align}\label{eq:tgamma}
\sigma^2|\boldsymbol{Y}\sim\mathrm{IG}((\beta_1+n)/{2},(\beta_2+n\widehat{\sigma}_n^2)/{2}).
\end{align}

Under the quasi-uniformity of the knots and \eqref{assump:tbspline2}, Lemma \ref{lem:matrix} concludes that there exist constants $0<C_1\leq C_2<\infty$ such that
\begin{equation}\label{eq:tdiagonal}
C_1n\left(\prod_{k=1}^dJ_k^{-1}\right)\boldsymbol{I}_J\leq\boldsymbol{B}^T\boldsymbol{B}\leq C_2n\left(\prod_{k=1}^dJ_k^{-1}\right)\boldsymbol{I}_J.
\end{equation}
In particular, $\|\boldsymbol{B}^T\boldsymbol{B}\|_{(2,2)}\asymp n\prod_{k=1}^dJ_k^{-1}$. Combining the above with \eqref{assump:tprior},
\begin{align}
\left(C_1n\prod_{k=1}^dJ_k^{-1}+{c_2}^{-1}\right)&
\leq\lambda_{\mathrm{min}}\left(\boldsymbol{B}^T\boldsymbol{B}+\boldsymbol{\Omega}^{-1}\right)\nonumber\\
&\leq\lambda_{\mathrm{max}}\left(\boldsymbol{B}^T\boldsymbol{B}+\boldsymbol{\Omega}^{-1}\right)\leq
\left(C_2n\prod_{k=1}^dJ_k^{-1}+{c_1}^{-1}\right).\label{eq:teigen}
\end{align}

\section{Posterior contraction rates}\label{sec:tsupnormf}

To establish posterior contraction rates for $f$ and its mixed partial derivatives with unknown $\sigma$, a key step is showing that the empirical Bayes estimator for $\sigma$ in the empirical Bayes approach or the posterior distribution of $\sigma$ in the hierarchical Bayes approach, are consistent, uniformly for the true regression function $f_0$ satisfying $\|f_0\|_{\boldsymbol{\alpha},\infty}\leq R$ for any given $R>0$.

\begin{proposition}\label{th:tsigma2con}
Let $J_k\asymp n^{\alpha^{*}/\{\alpha_k(2\alpha^{*}+d)\}}$, $k=1,\dotsc,d$. Then for any $R>0$,
the following assertions holds uniformly for the true regression $f_0$ satisfying
$\|f_0\|_{\boldsymbol{\alpha},\infty}\leq R$:
\begin{enumerate}
\item [(a)] the empirical Bayes estimator $\widehat{\sigma}_n$ converges to the true $\sigma_0$ at the rate
$\max(n^{-1/2}, n^{-2\alpha^*/(2\alpha^*+d)})$;
\item [(b)] if the inverse gamma prior $\mathrm{IG}(\beta_1/2,\beta_2/2)$ is used on $\sigma^2$, then the posterior for $\sigma$ contracts at $\sigma_0$ at the same rate;
\item [(c)] if the true distribution of the regression errors $\varepsilon_1,\ldots,\varepsilon_n$ is Gaussian, then for any prior on $\sigma$ with  positive and continuous density, the posterior distribution of $\sigma$ is consistent.
\end{enumerate}
\end{proposition}

For the rest of the paper, we shall treat $f$ and its mixed partial derivatives in a unified framework by viewing $f$ as $D^{\boldsymbol{0}}f$. Then the results on posterior contraction and credible sets (Section \ref{sec:cred}) for $f$ can be recovered by setting $\boldsymbol{r}=\boldsymbol{0}$. Since an explicit expression for the conditional posterior of $D^{\boldsymbol{r}}f$ given $\sigma$ is available due to the normal-normal conjugacy, we derive contraction rates by directly bounding posterior probabilities of deviations from the truth uniformly for $\sigma$ in a shrinking neighborhood of $\sigma_0$, which suffices in view of the consistency of the empirical Bayes estimator or that of the posterior distribution of $\sigma$. A decomposition of the posterior mean square error into posterior variance, variance and squared bias of the posterior mean is used for pointwise contraction, and uniformized using maximal inequalities to establish contraction with respect to the supremum distance. Contraction rates below are uniform in $\|f_0\|_{\boldsymbol{\alpha},\infty}\leq R$. We write $\epsilon_{n,\boldsymbol{r}}=n^{-\alpha^{*}\{1-\sum_{k=1}^d(r_k/\alpha_k)\}/(2\alpha^{*}+d)}$ and $\epsilon_{n,\boldsymbol{r},\infty}=(\log{n}/n)^{\alpha^{*}\{1-\sum_{k=1}^d(r_k/\alpha_k)\}/(2\alpha^{*}+d)}$. Observe that for $\epsilon_{n,\boldsymbol{r}}$ and $\epsilon_{n,\boldsymbol{r},\infty}$ to approach $0$ as $n\rightarrow\infty$, we will need $\sum_{k=1}^d(r_k/\alpha_k)<1$. For the hierarchical Bayes approach, we do not restrict to the inverse gamma prior for $\sigma^2$ but throughout assume that its posterior is consistent uniformly for $\|f_0\|_{\boldsymbol{\alpha},\infty}\leq R$ for any $R>0$.

\begin{theorem}[Pointwise contraction]\label{prop:tsupnormfprime}
If $J_k\asymp n^{\alpha^{*}/\{\alpha_k(2\alpha^{*}+d)\}}$ for $k=1,\dotsc,d$, then for any $\boldsymbol{x}\in[0,1]^d$ and $M_n\rightarrow\infty$,
\begin{align*}
\text{Empirical Bayes:}&\quad\mathrm{E}_0\Pi_{\widehat{\sigma}_n}(|D^{\boldsymbol{r}}f(\boldsymbol{x})-D^{\boldsymbol{r}}f_0(\boldsymbol{x})|>M_n\epsilon_{n,\boldsymbol{r}}|\boldsymbol{Y})\rightarrow0.\\
\text{Hierarchical Bayes:}&\quad\mathrm{E}_0\Pi(|D^{\boldsymbol{r}}f(\boldsymbol{x})-D^{\boldsymbol{r}}f_0(\boldsymbol{x})|>M_n\epsilon_{n,\boldsymbol{r}}|\boldsymbol{Y})\rightarrow0.
\end{align*}
\end{theorem}

\begin{remark}\rm
\label{L2 contraction}
The above rate of contraction holds for the $L_2$-distance as well under the same set of assumptions for both  empirical and hierarchical Bayes approaches. This follows since the posterior expectation of the squared $L_2$-norm can be bounded by the integral of the corresponding uniform estimates of the pointwise case obtained in the proof of Theorem~\ref{prop:tsupnormfprime}.
\end{remark}

\begin{theorem}[$L_\infty$-contraction]\label{th:tsupnormfprime}
If $J_k\asymp(n/\log{n})^{\alpha^{*}/\{\alpha_k(2\alpha^{*}+d)\}}$ for $k=1,\dotsc,d$, then for any $M_n\rightarrow\infty$,
\begin{align*}
\text{Empirical Bayes:}&\quad\mathrm{E}_0\Pi_{\widehat{\sigma}_n}(\|D^{\boldsymbol{r}}f-D^{\boldsymbol{r}}f_0\|_\infty>M_n\epsilon_{n,\boldsymbol{r},\infty}|\boldsymbol{Y})\rightarrow0.\\
\text{Hierarchical Bayes:}&\quad\mathrm{E}_0\Pi(\|D^{\boldsymbol{r}}f-D^{\boldsymbol{r}}f_0\|_\infty>M_n\epsilon_{n,\boldsymbol{r},\infty}|\boldsymbol{Y})\rightarrow0.
\end{align*}
\end{theorem}

Note that an extra logarithmic factor appears in the $L_\infty$-rate in agreement with the corresponding minimax rate for the problem (see \citep{stone1980,stone1982}). A similar result for the white noise model using a prior based on wavelet basis expansion for the signal function is given by Theorem 1 of \citep{nickl2011} for known variance. It is interesting to note that given any notion of posterior contraction and smoothness index, the same optimal $J_k$, $k=1,\dotsc,d$, applies to $f$ and its mixed partial derivatives, so the Bayes procedure automatically adapts to the order of the derivative to be estimated.

\section{Credible sets for $f$ and its mixed partial derivatives}\label{sec:cred}

We begin by constructing pointwise credible set for $D^{\boldsymbol{r}} f(\boldsymbol{x})$ at $\boldsymbol{x}\in[0,1]^d$, where $\boldsymbol{r}\in\mathbb{N}_0^d$ satisfies $\sum_{k=1}^d (r_k/\alpha_k)<1$. Let $\gamma_n\in[0,1]$ be a sequence such that $\gamma_n\rightarrow0$ as $n\rightarrow\infty$. Define $z_{\delta}$ to be the $(1-\delta)$-quantile of a standard normal. Since $\Pi(D^{\boldsymbol{r}}f(\boldsymbol{x})|\boldsymbol{Y},\sigma)\sim
\mathrm{N}(\boldsymbol{A}_{\boldsymbol{r}}(\boldsymbol{x})\boldsymbol{Y}+\boldsymbol{c}_{\boldsymbol{r}}(\boldsymbol{x})\boldsymbol{\eta},
\sigma^2\Sigma_{\boldsymbol{r}}(\boldsymbol{x},\boldsymbol{x}))$, we can construct a $(1-\gamma_n)$-pointwise credible interval for $D^{\boldsymbol{r}} f(\boldsymbol{x})$ from the relation
\begin{align*}
\Pi(g:|g(\boldsymbol{x})-\boldsymbol{A}_{\boldsymbol{r}}(\boldsymbol{x})\boldsymbol{Y}-\boldsymbol{c}_{\boldsymbol{r}}(\boldsymbol{x})\boldsymbol{\eta}|\leq z_{\gamma_n/2}\sigma\sqrt{\Sigma_{\boldsymbol{r}}(\boldsymbol{x},\boldsymbol{x})}|\boldsymbol{Y},\sigma)=1-\gamma_n.
\end{align*}
However, as $\sigma$ is unknown, we use empirical Bayes by substituting $\sigma$ by  $\widehat{\sigma}_n$ derived in \eqref{eq:tsigma2hat}, leading to the following empirical credible set:
$$
\widehat{\mathcal{C}}_{n, \boldsymbol{r}, \gamma_n}(\boldsymbol{x})=\{g:|g(\boldsymbol{x})-\boldsymbol{A}_{\boldsymbol{r}}(\boldsymbol{x})\boldsymbol{Y}-
\boldsymbol{c}_{\boldsymbol{r}}(\boldsymbol{x})\boldsymbol{\eta}|\leq z_{\gamma_n/2}\widehat{\sigma}_n\sqrt{\Sigma_{\boldsymbol{r}}(\boldsymbol{x},\boldsymbol{x})}\}.
$$
For the hierarchical Bayes approach, the resulting credible region is given by $\mathcal{C}_{n, \boldsymbol{r}, \gamma_n}(\boldsymbol{x}) =\{g:|g(\boldsymbol{x})-\boldsymbol{A}_{\boldsymbol{r}}(\boldsymbol{x})\boldsymbol{Y}-\boldsymbol{c}_{\boldsymbol{r}}(\boldsymbol{x})\boldsymbol{\eta}|\leq R_{n, \boldsymbol{r}, \gamma_n}(\boldsymbol{x})\}$, where $R_{n, \boldsymbol{r}, \gamma_n}(\boldsymbol{x})$ is the $(1-\gamma_n)$-quantile of the marginal posterior distribution of $|D^{\boldsymbol{r}}f(\boldsymbol{x})-\boldsymbol{A}_{\boldsymbol{r}}(\boldsymbol{x})\boldsymbol{Y}-\boldsymbol{c}_{\boldsymbol{r}}(\boldsymbol{x})\boldsymbol{\eta}|$ after  integrating out $\sigma$ with respect to its posterior distribution. If the conjugate inverse-gamma prior is used on $\sigma^2$, then the cut-off may be expressed explicitly in terms of quantiles of a generalized t-distribution. In general, the cut-off value $R_{n, \boldsymbol{r}, \gamma_n}(\boldsymbol{x})$ may be found by posterior sampling: generate $\sigma$ from its marginal posterior distribution and $D^{\boldsymbol{r}}f|(\boldsymbol{Y},\sigma)\sim\mathrm{GP}(\boldsymbol{A}_{\boldsymbol{r}}\boldsymbol{Y}+\boldsymbol{c}_{\boldsymbol{r}}\boldsymbol{\eta},\sigma^2\Sigma_{\boldsymbol{r}})$.

\begin{theorem}[Pointwise credible intervals]\label{th:tcrpoint}
If $J_k\asymp n^{\alpha^{*}/\{\alpha_k(2\alpha^{*}+d)\}}$, $k=1,\dotsc,d$, then for $\gamma_n\to 0$, the coverage of $\widehat{\mathcal{C}}_{n, \boldsymbol{r}, \gamma_n}(\boldsymbol{x})$ tends to $1$ and its radius is $O_{P_0}(\epsilon_{n,\boldsymbol{r}}\sqrt{\log{(1/\gamma_n)}})$ at  $\boldsymbol{x}\in[0,1]^d$ uniformly on $\|f_0\|_{\boldsymbol{\alpha},\infty}\leq R$.

If the posterior distribution of $\sigma$ is consistent, then the same conclusion holds for the hierarchical Bayes credible set $\mathcal{C}_{n, \boldsymbol{r}, \gamma_n}(\boldsymbol{x})$.
\end{theorem}

\begin{remark}\rm
\label{L2credible}
We can also define a $(1-\gamma_n)$-credible set in the $L_2$-norm for $D^{\boldsymbol{r}}f$ given $\boldsymbol{Y}$ and $\sigma$ as the set of all functions which differ from $\boldsymbol{A}_{\boldsymbol{r}}(\boldsymbol{x})\boldsymbol{Y}+\boldsymbol{c}_{\boldsymbol{r}}(\boldsymbol{x})\boldsymbol{\eta}$ in the $L_2$-norm by $\sigma h_{n, \boldsymbol{r}, 2, \gamma_n}$, where $h_{n, \boldsymbol{r}, 2, \gamma_n}$ is the $1-\gamma_n$ quantile of the $L_2$-norm of $\mathrm{GP}(0,\Sigma_{\boldsymbol{r}})$. Then the empirical Bayes credible set is obtained by substituting $\sigma$ by $\widehat{\sigma}_n$. The hierarchical Bayes credible set is obtained by replacing  $\sigma h_{n, \boldsymbol{r},2, \gamma_n}$ by the $1-\gamma_n$ quantile of  $\| D^{\boldsymbol{r}}f-\boldsymbol{A}_{\boldsymbol{r}}\boldsymbol{Y}-\boldsymbol{c}_{\boldsymbol{r}}\boldsymbol{\eta}\|_2$. Both credible regions have asymptotic coverage $1$ under the assumptions in Theorem~\ref{th:tcrpoint}.
\end{remark}

The $(1-\gamma_n)$-empirical Bayes $L_\infty$-credible set for $D^{\boldsymbol{r}}f$, can be expressed as  $\{g:\|g-\boldsymbol{A}_{\boldsymbol{r}}\boldsymbol{Y}-\boldsymbol{c}_{\boldsymbol{r}}\boldsymbol{\eta}\|_\infty\leq
\rho_n\widehat{\sigma}_nh_{n, \boldsymbol{r}, \infty, \gamma_n}\}$, where  $h_{n, \boldsymbol{r}, \infty, \gamma_n}$ is the $(1-\gamma_n)$-quantile of the $L_\infty$-norm of $\mathrm{GP}(0,\Sigma_{\boldsymbol{r}})$. It turns out that in order to obtain adequate frequentist coverage, this natural credible ball needs to be slightly inflated by a factor $\rho_n$, leading to the inflated empirical Bayes credible region
\begin{align}\label{eq:tradius}
\widehat{\mathcal{C}}_{n, \boldsymbol{r}, \infty, \gamma_n}^{\rho_n}=\{g:\|g-\boldsymbol{A}_{\boldsymbol{r}}\boldsymbol{Y}-\boldsymbol{c}_{\boldsymbol{r}}\boldsymbol{\eta}\|_\infty\leq
\rho_n\widehat{\sigma}_nh_{n, \boldsymbol{r}, \infty, \gamma_n}\}.
\end{align}
On the other hand, unlike in the pointwise or the $L_2$-credible regions, we need not make $\gamma_n\to 0$, but can allow any fixed $\gamma<1/2$.
In the hierarchical Bayes approach, we consider the analogous credible ball ${\mathcal{C}}_{n, \boldsymbol{r}, \infty, \gamma}^{\rho_n}=\{g:\|g-\boldsymbol{A}_{\boldsymbol{r}}\boldsymbol{Y}-\boldsymbol{c}_{\boldsymbol{r}}\boldsymbol{\eta}\|_\infty\leq
\rho_n R_{n, \boldsymbol{r}, \infty, \gamma}\}$, where $R_{n, \boldsymbol{r}, \infty, \gamma}$ stands for the $(1-\gamma)$-quantile of the marginal posterior distribution of $\|D^{\boldsymbol{r}}f-\boldsymbol{A}_{\boldsymbol{r}}\boldsymbol{Y}-\boldsymbol{c}_{\boldsymbol{r}}\boldsymbol{\eta}\|_\infty$ integrating out $\sigma$ with respect to its posterior distribution.

\begin{theorem}[$L_\infty$-credible region]\label{th:tcrinftyrho}
If $J_k\asymp(n/\log{n})^{\alpha^{*}/\{\alpha_k(2\alpha+d)\}}$ for $k=1,\dotsc,d$, then for any $\rho_n\to \infty$ and $\gamma<1/2$, the coverage of $\widehat{\mathcal{C}}_{n, \boldsymbol{r}, \infty, \gamma}^{\rho_n}$ tends to $1$ and its radius is $O_{P_0}(\epsilon_{n,\boldsymbol{r},\infty}\rho_n)$ uniformly in $\|f_0\|_{\boldsymbol{\alpha},\infty}\leq R$.
Moreover, if the true distribution of the regression errors is Gaussian, then we can let $\rho_n=\rho$ for some sufficiently large constant $\rho>0$.

If the posterior distribution of $\sigma$ is consistent, then the same conclusion holds for the hierarchical Bayes $L_\infty$-credible ball ${\mathcal{C}}_{n, \boldsymbol{r}, \infty, \gamma}^{\rho_n}$.
\end{theorem}

\begin{remark}\rm
To control the size of $\widehat{\mathcal{C}}_{n, \boldsymbol{r}, \infty, \gamma}^{\rho_n}$ and ensure guaranteed frequentist coverage, we can take $\rho_n$ to be a factor slowly tending to infinity, or a sufficiently large constant for the Gaussian situation. A similar correction factor was also used by \citep{botond2015} in the context of adaptive $L_2$-credible region.
\end{remark}

\section{Simulation}\label{sec:sim}
We compare finite sample performance of pointwise credible intervals and $L_\infty$-credible bands for $f$ in one dimension (i.e., $d=1$, $r=0$)  with confidence intervals and $L_\infty$-confidence bands proposed by Theorem~4.1 of \citep{localspline}. Following \citep{inverseprob}, we consider the true function
$f_0(x)=\sqrt{2}\sum_{i=1}^\infty i^{-3/2}\sin i\cos\{{(i-1/2)\pi x}\}$,
$x\in[0,1]$, which has smoothness $\alpha=1$. We observed the signal $f_0$ with i.i.d. $\mathrm{N}(0,0.1)$ errors at covariate values at $X_i=(i-1)/(n-1)$ for $i=1,\dotsc,n$. We use cubic B-splines (i.e., $q=4$) with uniform knot sequence, where we added 4 duplicate knots at $0$ and $1$. For the prior parameters, we set $\boldsymbol{\eta}=\boldsymbol{0}$ and $\boldsymbol{\Omega}=\boldsymbol{I}_J$. We construct $(1-\gamma_n)$-empirical credible intervals for $\gamma_n=5/n$ with $\widehat{\sigma}_n$ computed using \eqref{eq:tsigma2hat}. The corresponding confidence regions are constructed using Theorem 4.1 of \citep{localspline} based on the least squares estimator $\widehat{f}(x)=\boldsymbol{b}_{J,q}(x)^T\widehat{\boldsymbol{\theta}}$ for $\widehat{\boldsymbol{\theta}}=(\boldsymbol{B}^T\boldsymbol{B})^{-1}\boldsymbol{B}^T\boldsymbol{Y}$, and $\widetilde{\sigma}_n^2=(\boldsymbol{Y}-\boldsymbol{B\widehat{\theta}})^T
(\boldsymbol{Y}-\boldsymbol{B\widehat{\theta}})/(n-J)$. In the Bayesian context when the smoothing parameter $J$ is to be determined from the data, it is natural to use its posterior mode. However for a fair comparison, we used leave-one-out cross validation to determine $J$ for both methods and also observed that the posterior mode essentially chose the same values. We conduct our experiment across sample sizes $n=100,300,500,700,1000,2000$. For pointwise credible and confidence intervals, we report the empirical coverage based on $1000$ Monte Carlo runs for each $n$. All simulations were carried out in \texttt{R} using the \texttt{bs} function from the \texttt{splines} package.

The coverage probabilities of pointwise credible and confidence intervals are shown in Figure \ref{fig:pcov}.
\begin{figure}[h!]
\centering
\begin{subfigure}[b]{0.25\textwidth}
\includegraphics[width=\textwidth]{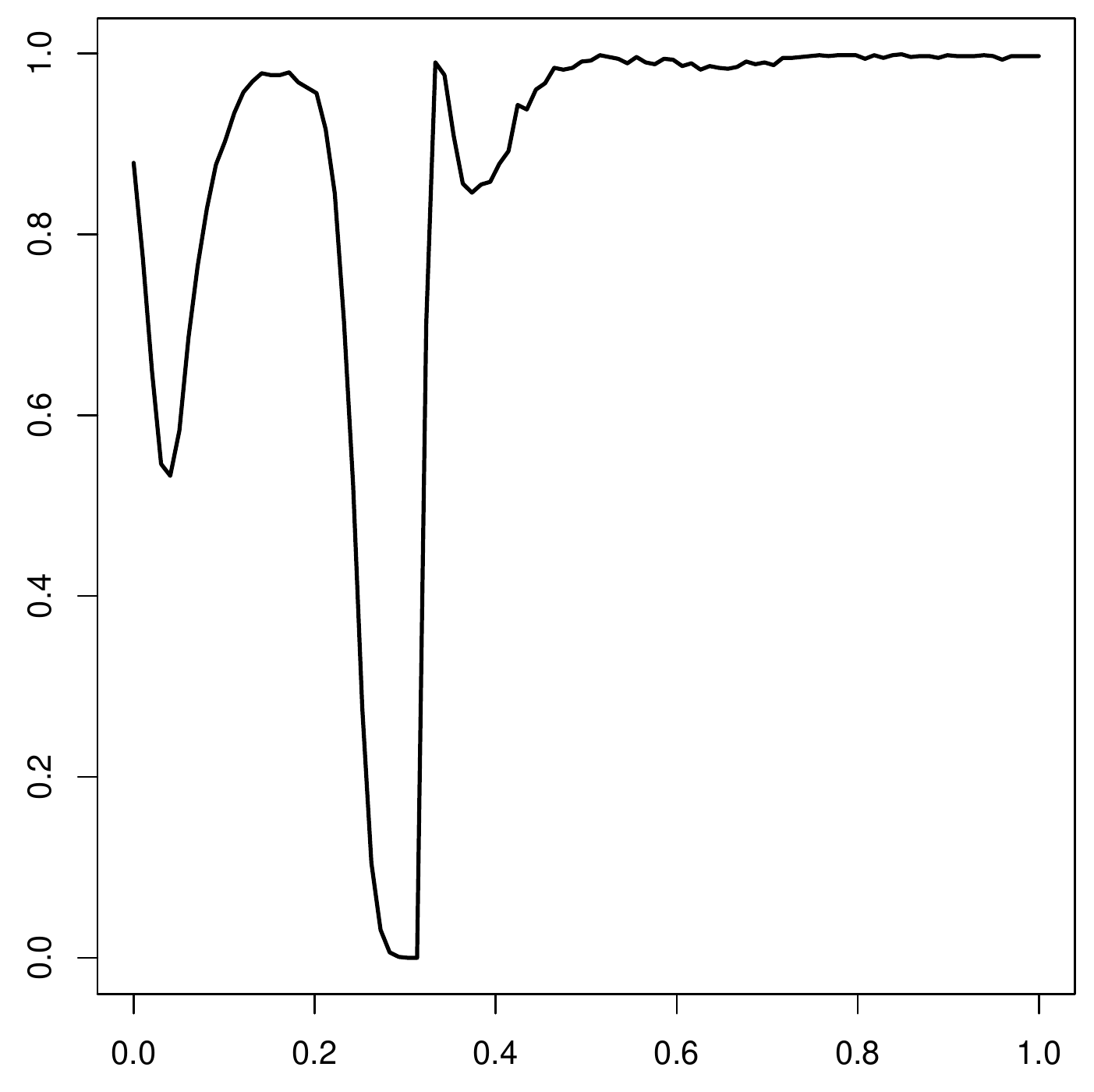}
\caption{$n=100$\\  \centering credibility $=0.95$}
\end{subfigure}
\begin{subfigure}[b]{0.25\textwidth}
\includegraphics[width=\textwidth]{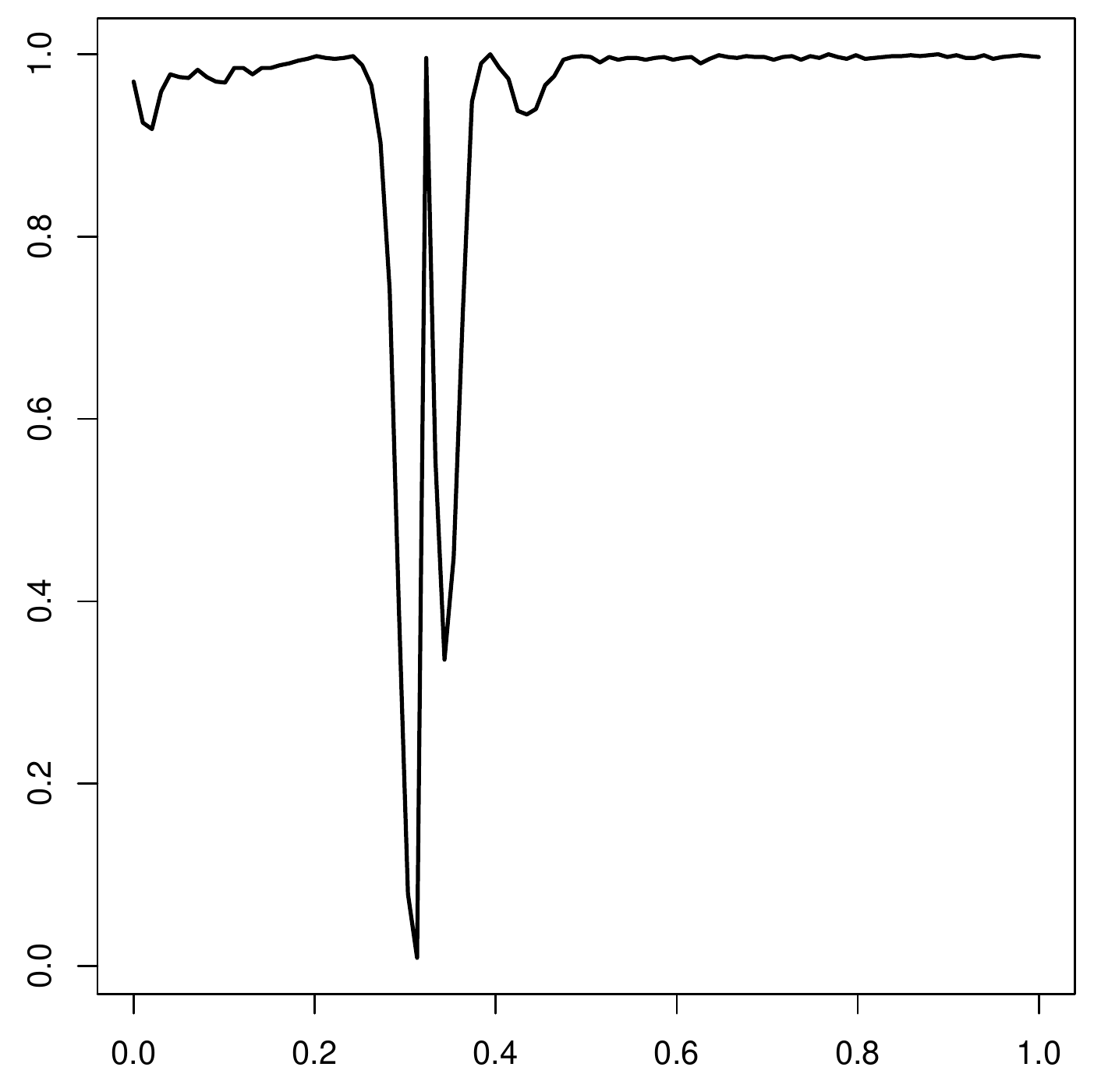}
\caption{$n=500$\\  \centering credibility $=0.99$}
\end{subfigure}
\begin{subfigure}[b]{0.25\textwidth}
\includegraphics[width=\textwidth]{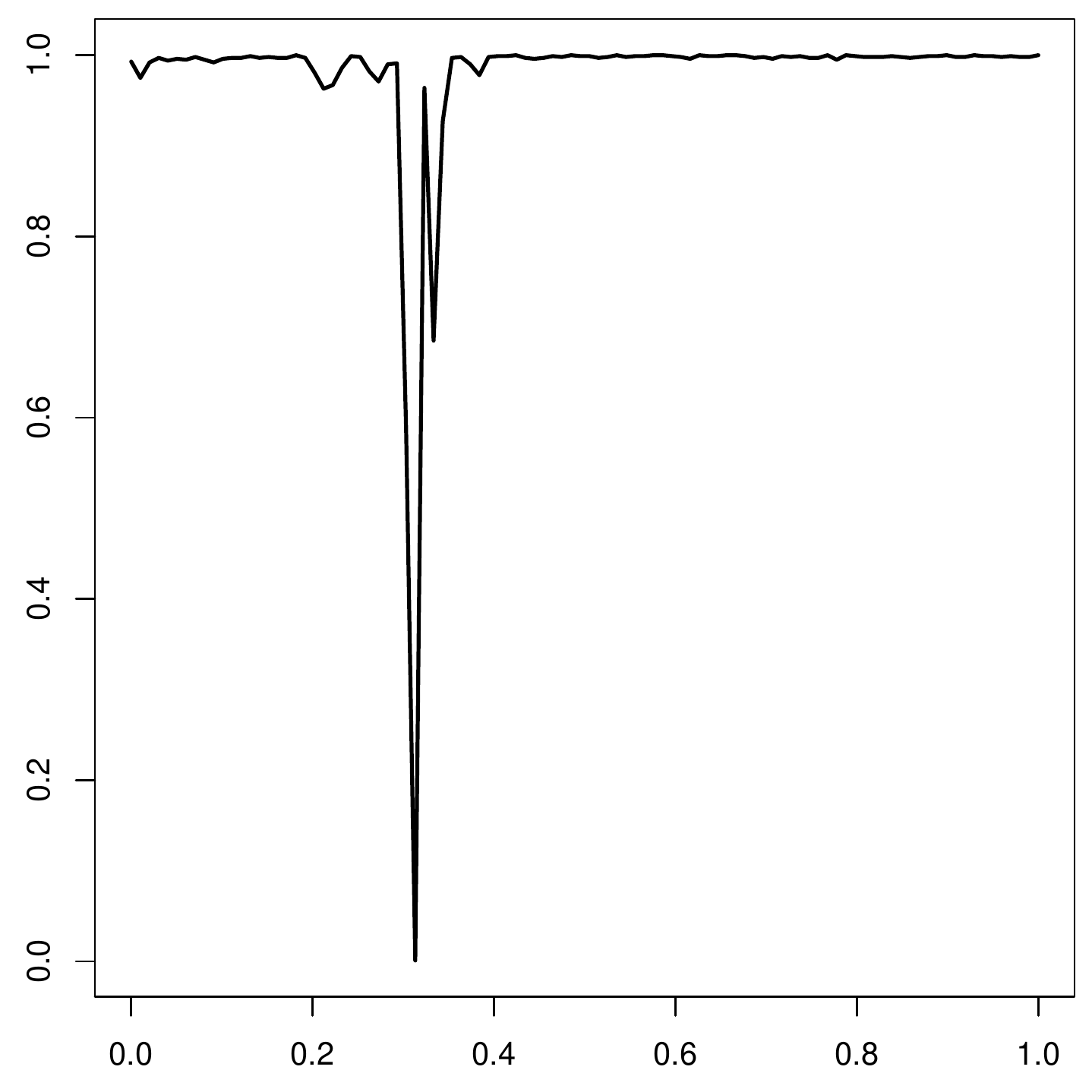}
\caption{$n=2000$\\ \centering credibility $=0.9975$}
\end{subfigure}

\begin{subfigure}[b]{0.25\textwidth}
\includegraphics[width=\textwidth]{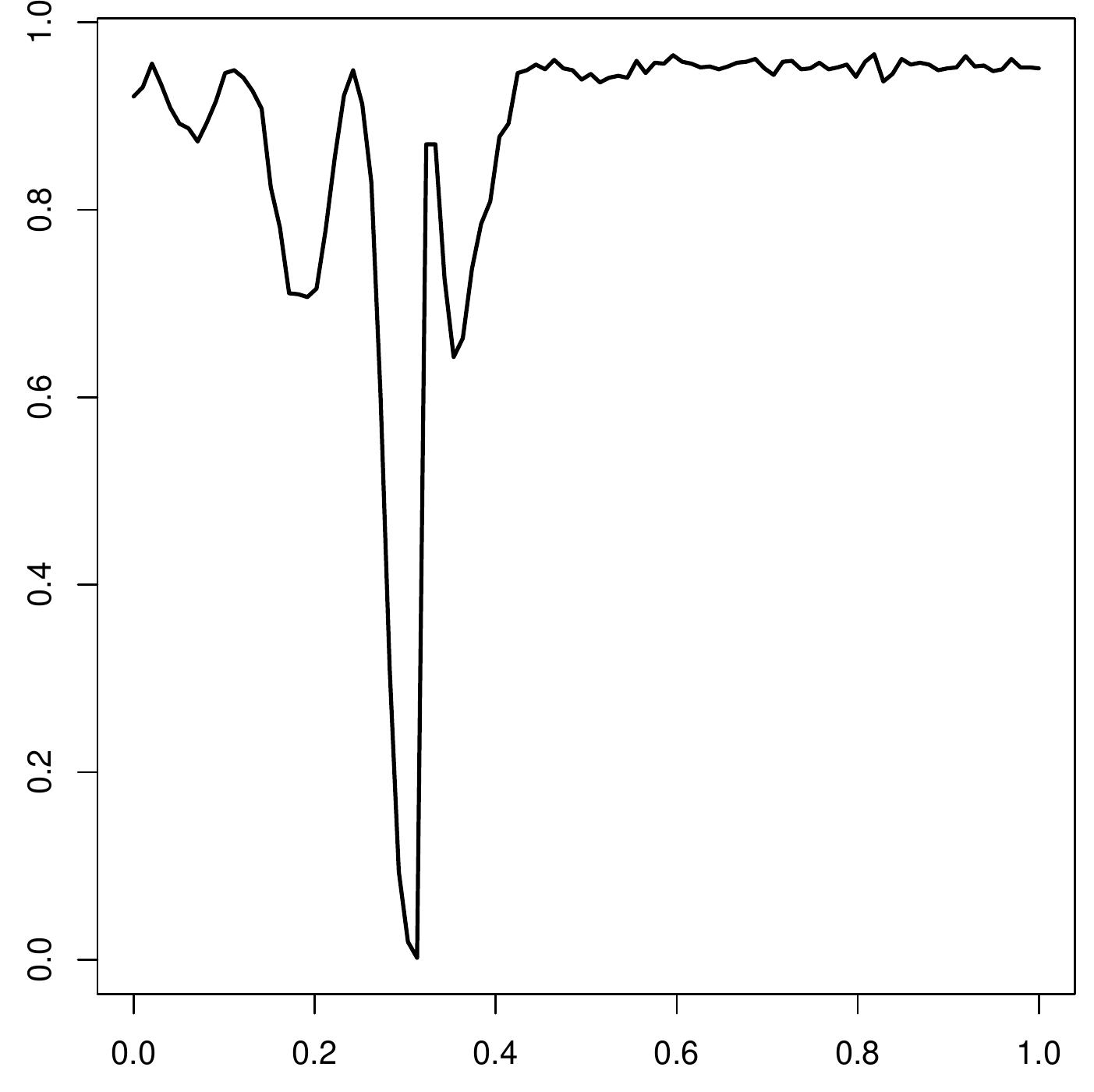}
\caption{$n=100$\\ \centering confidence $=0.95$}
\end{subfigure}
\begin{subfigure}[b]{0.25\textwidth}
\includegraphics[width=\textwidth]{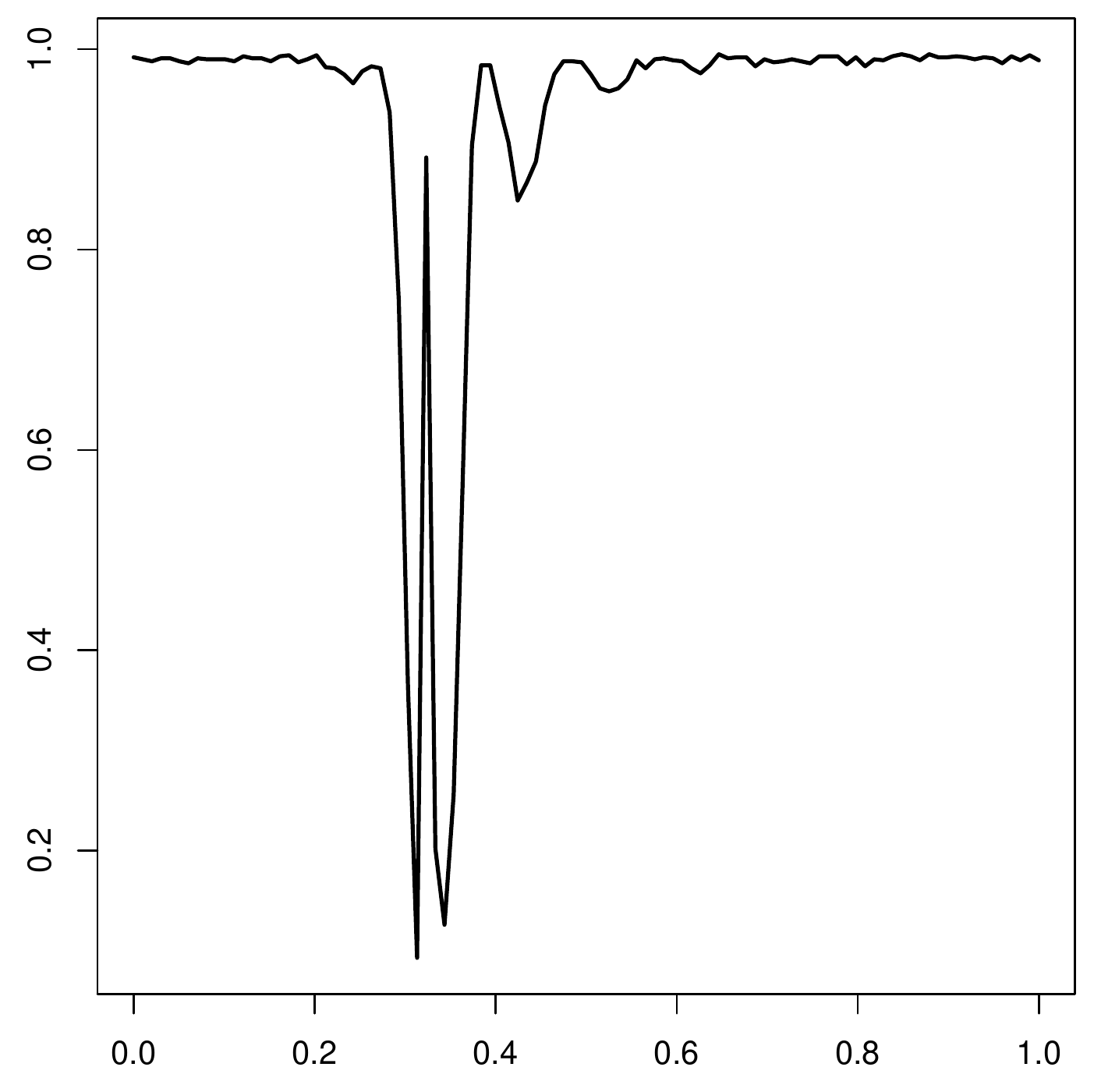}
\caption{$n=500$\\ \centering confidence $=0.99$}
\end{subfigure}
\begin{subfigure}[b]{0.25\textwidth}
\includegraphics[width=\textwidth]{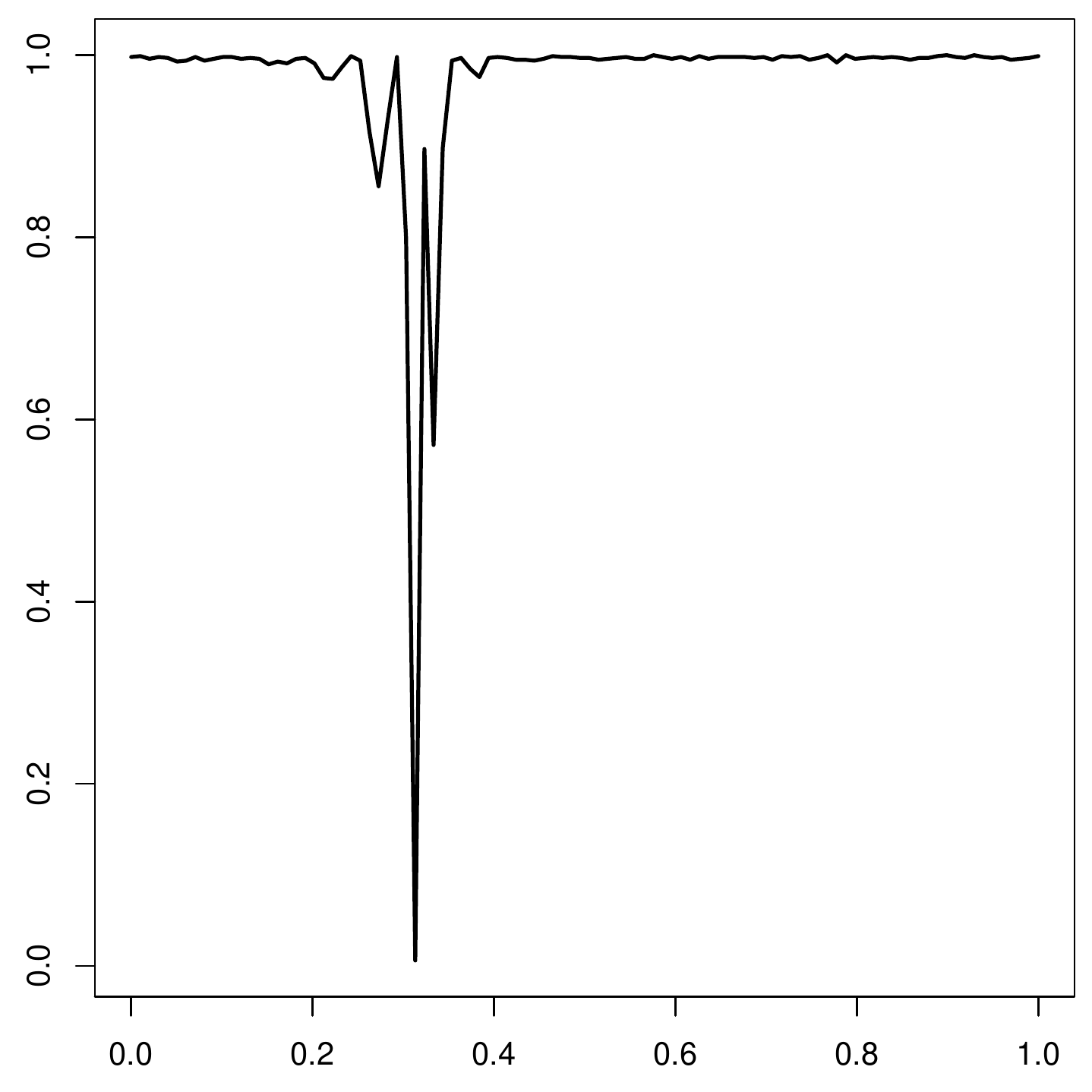}
\caption{$n=2000$\\ \centering confidence $=0.9975$}
\end{subfigure}
\caption{Pointwise coverage probabilities for credible and confidence intervals. The $y$-axis is the coverage probabilities and the $x$-axis is the covariate $x$.}
\label{fig:pcov}
\end{figure}
One distinguishing feature is the downward spike at around the bump of $f_0$ at $x=0.3$ in Figure \ref{fig:covband}, and the plots narrow down to this point as $n$ increases. Moreover, the pointwise coverage is $0$ at this point for both Bayesian and frequentist methods in all sample sizes considered. This phenomenon occurs perhaps due to the fact that the true function at $x=0.3$ has a sharp bend but the function is much smoother elsewhere, so based on a limited sample the cross-validation method oversmooths by choosing a smaller $J$ than ideal. Both methods yield almost the same pointwise coverage for large sample sizes, and are equivalent in quantifying uncertainty of estimating $f_0$. To cover the function at all points, we consider the simultaneous (modified) credible band at the level $1-\gamma=0.95$, given by
$(\boldsymbol{A}_0(x)\boldsymbol{Y}+\boldsymbol{c}_0(x)\boldsymbol{\eta})\pm\rho\widehat{\sigma}_nh_{n, 0, \infty, \gamma}$.

The second assertion of Theorem \ref{th:tcrinftyrho} allows us to use a fixed $\rho$ because our true errors are normally distributed which we choose as $\rho=0.5$. To construct $(1-\gamma)$-asymptotic confidence band, we use Theorem 4.2 of \citep{localspline}.
\begin{table}[htbp]
  \centering
  \caption{$95\%$ simultaneous credible and confidence bands.}
    \begin{tabular}{ccccccc}
    \toprule
    $n$   & \textbf{100} & \textbf{300} & \textbf{500} & \textbf{700} & \textbf{1000} & \textbf{2000} \\
    \midrule
    Credible band coverage  & 0.852     & 0.896     & 0.954     & 0.945     & 0.964     & 0.972 \\
    Confidence band coverage & 0.972 & 0.948 & 0.963 & 0.978 & 0.985 & 0.986 \\
    Credible band radius & 0.235     & 0.155     & 0.148     & 0.127     & 0.121     & 0.098 \\
    Confidence band mean radius & 0.27     & 0.165     & 0.147     & 0.132     & 0.129     & 0.101 \\
    Confidence band max radius & 0.64     & 0.436     & 0.409     & 0.374     & 0.372     & 0.3 \\
    \bottomrule
    \end{tabular}
  \label{tab:covband}
\end{table}
\begin{figure}[h!]
\centering
\begin{subfigure}[b]{0.25\textwidth}
\includegraphics[width=\textwidth]{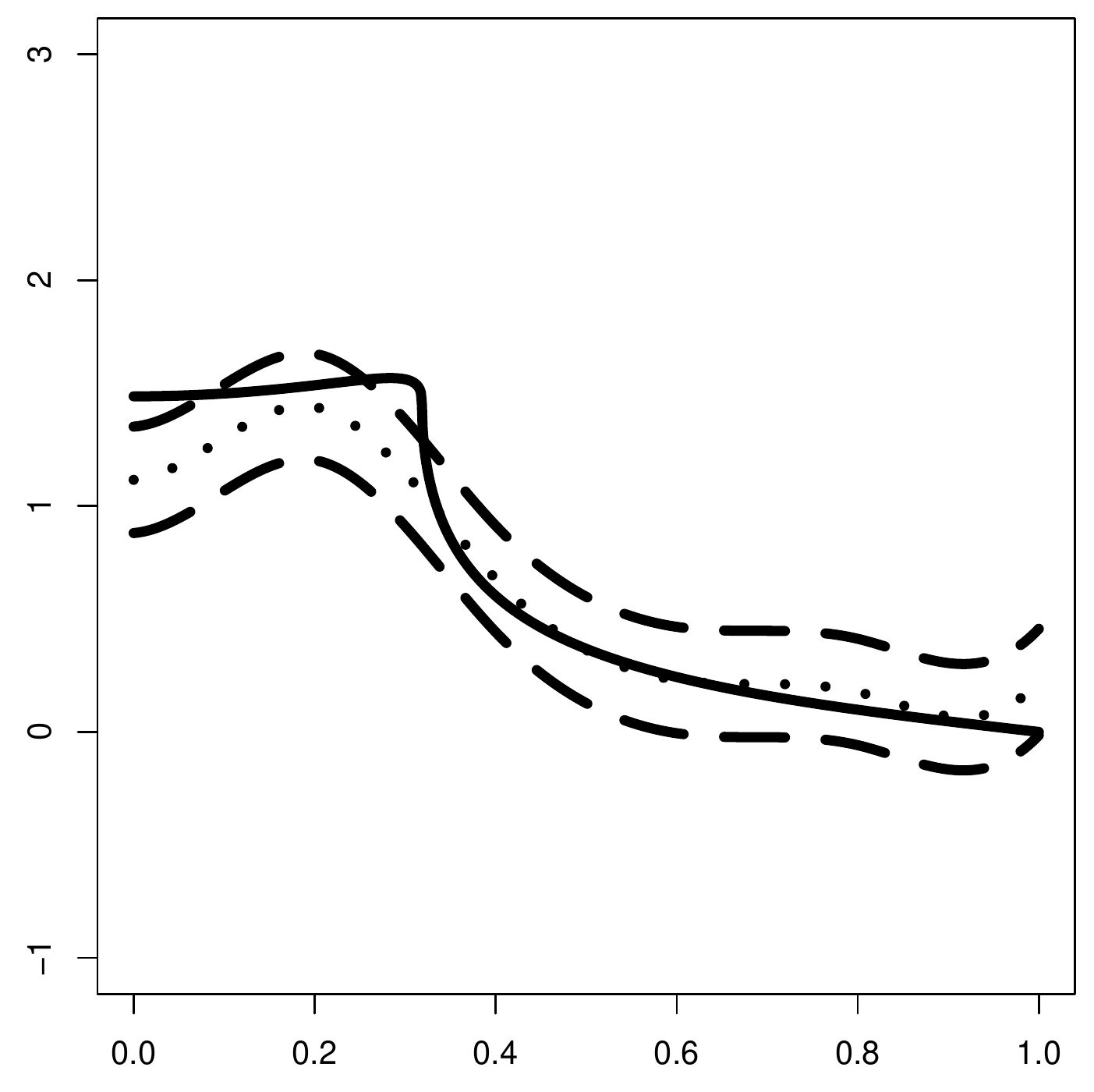}
\caption{Bayes: $n=100$}
\end{subfigure}
\begin{subfigure}[b]{0.25\textwidth}
\includegraphics[width=\textwidth]{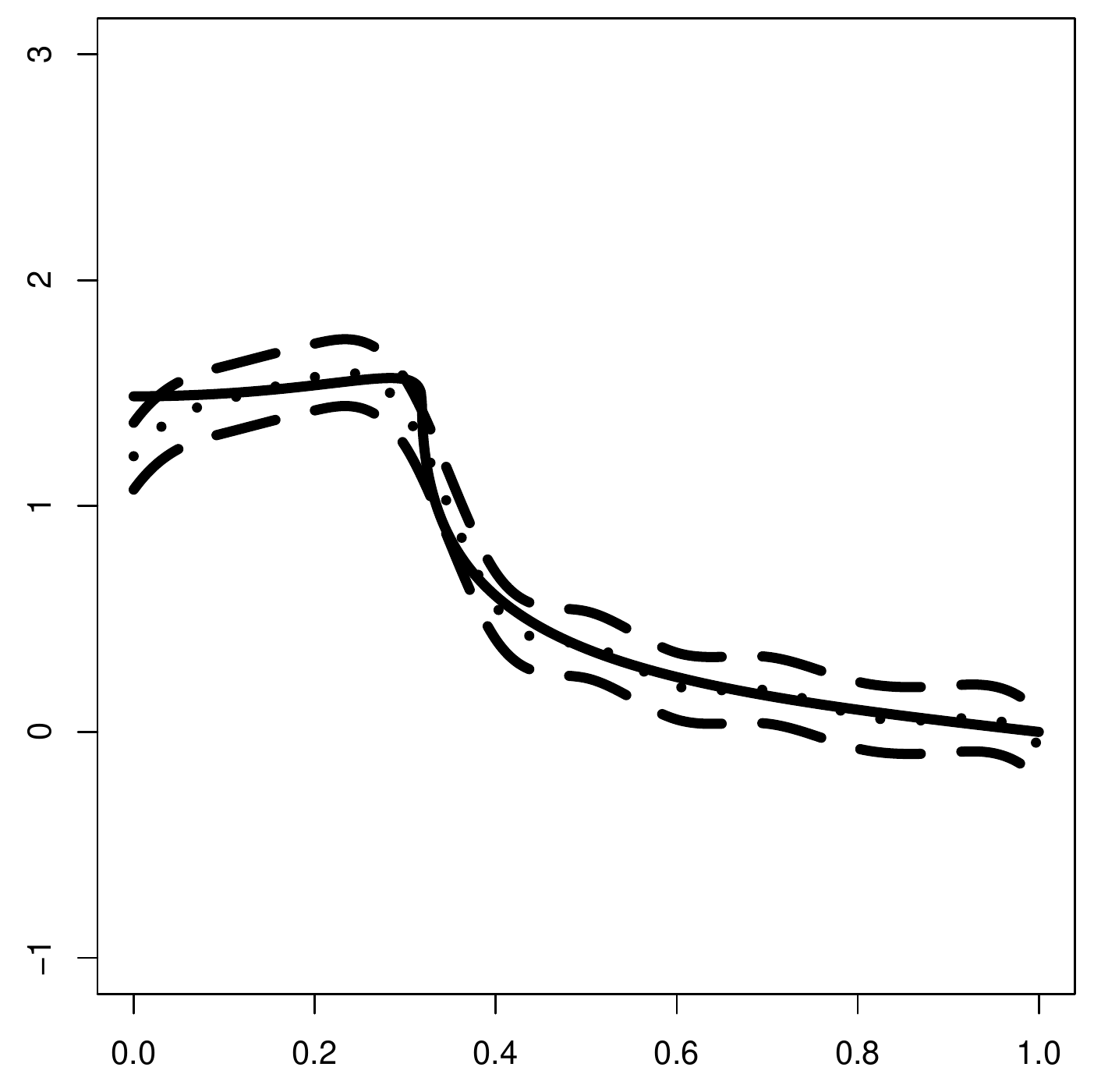}
\caption{Bayes: $n=500$}
\end{subfigure}
\begin{subfigure}[b]{0.25\textwidth}
\includegraphics[width=\textwidth]{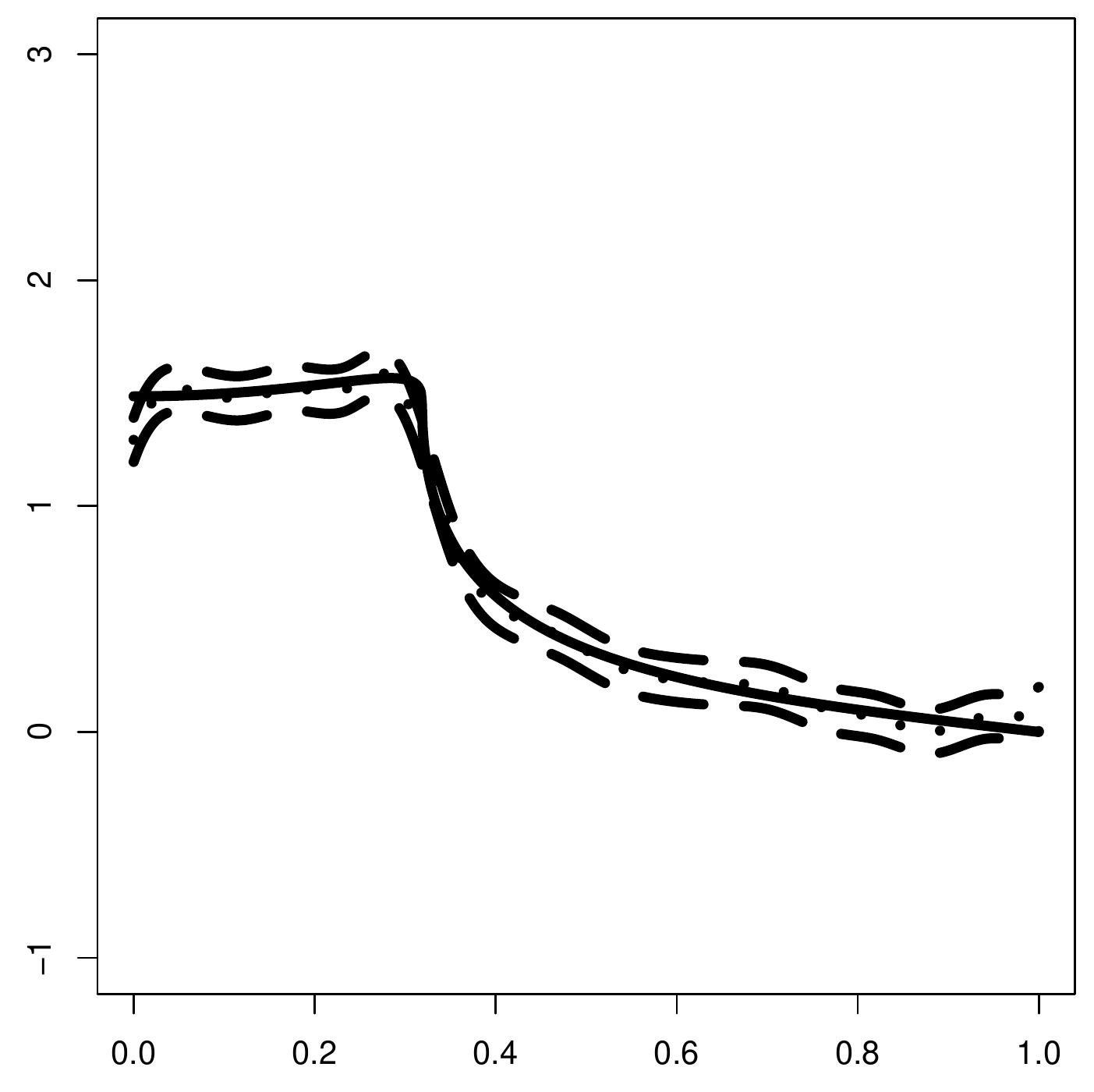}
\caption{Bayes: $n=2000$}
\end{subfigure}
\begin{subfigure}[b]{0.25\textwidth}
\includegraphics[width=\textwidth]{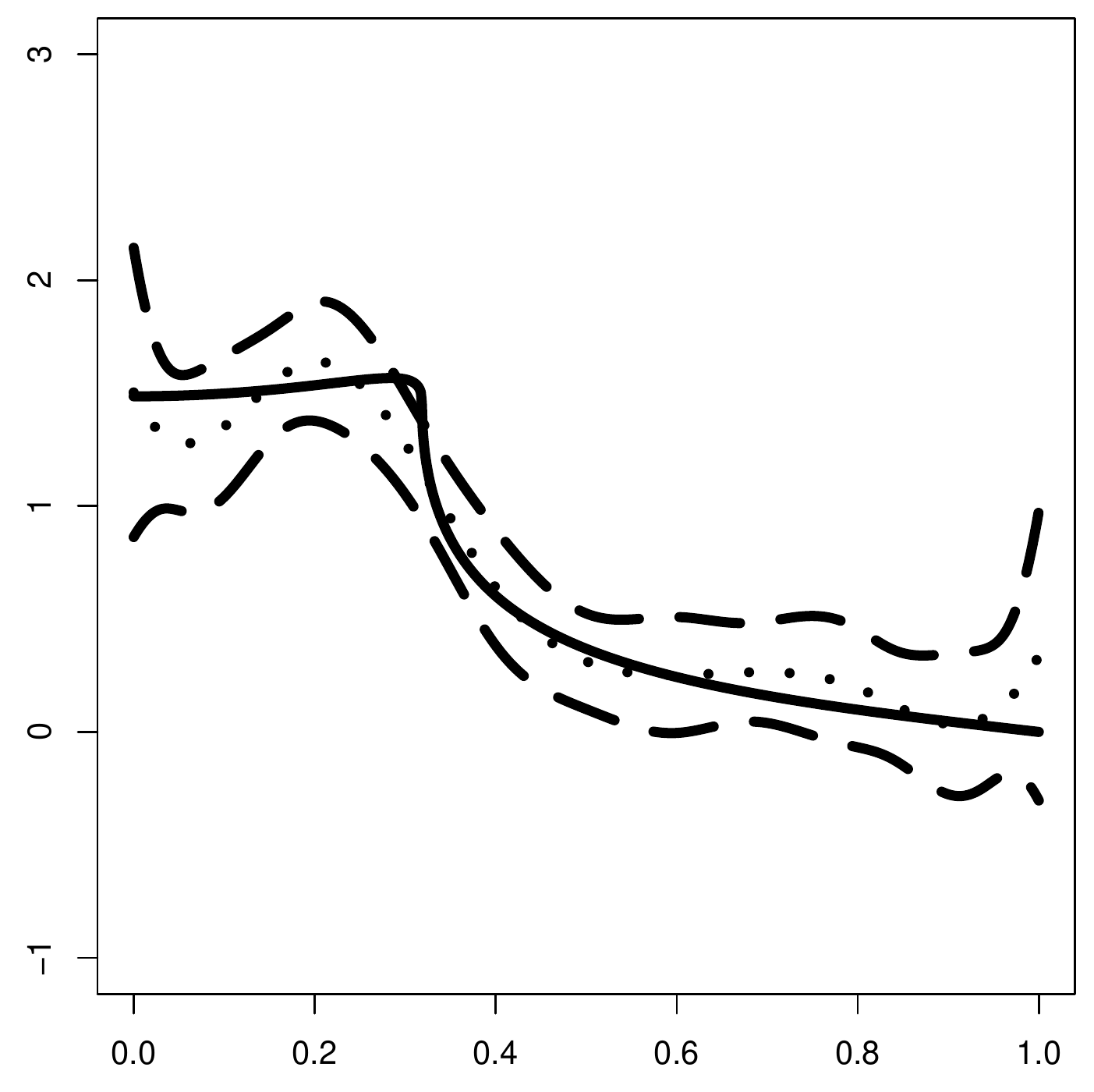}
\caption{Frequentist: \\ \centering$n=100$}
\end{subfigure}
\begin{subfigure}[b]{0.25\textwidth}
\includegraphics[width=\textwidth]{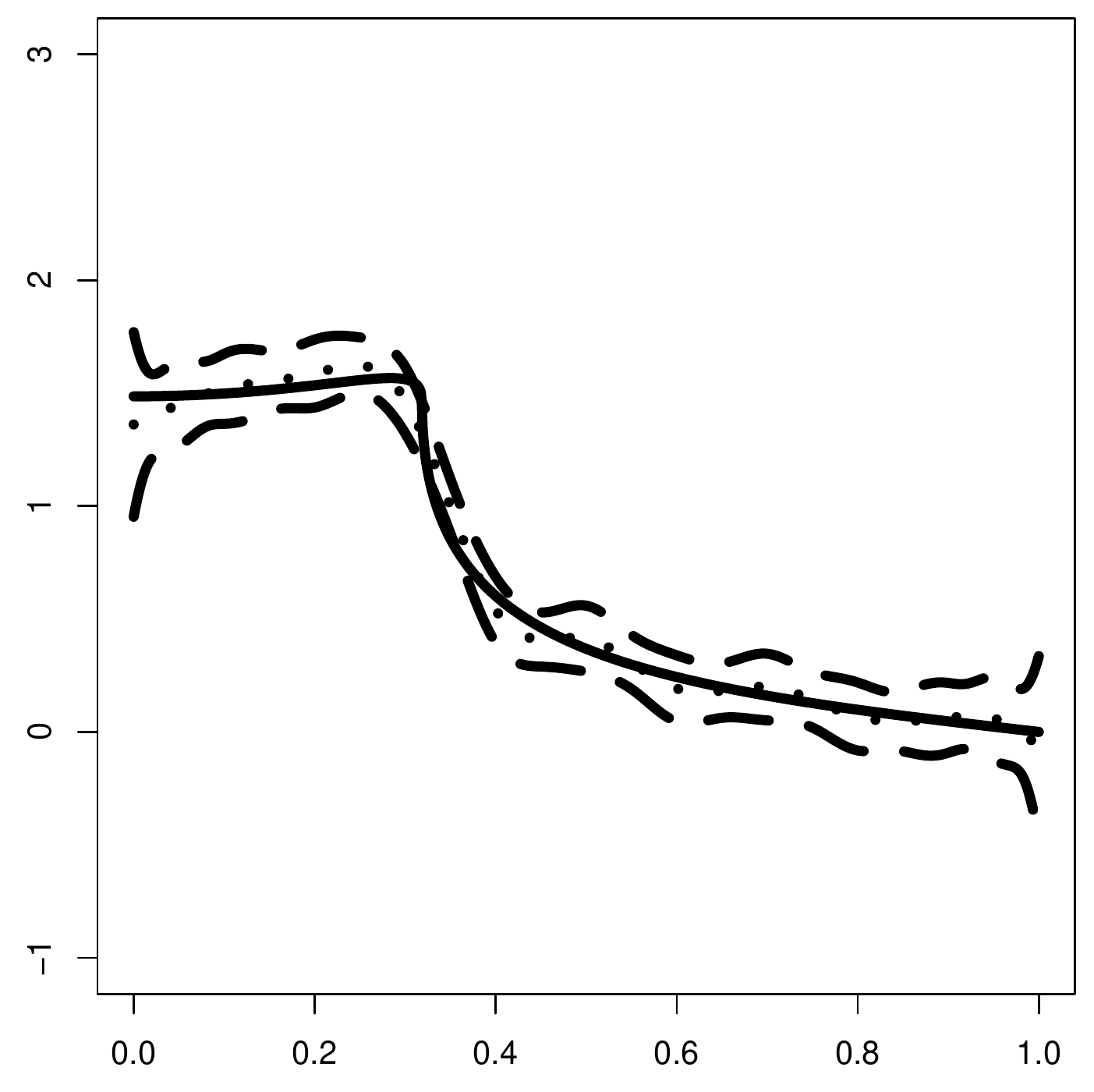}
\caption{Frequentist: \\ \centering $n=500$}
\end{subfigure}
\begin{subfigure}[b]{0.25\textwidth}
\includegraphics[width=\textwidth]{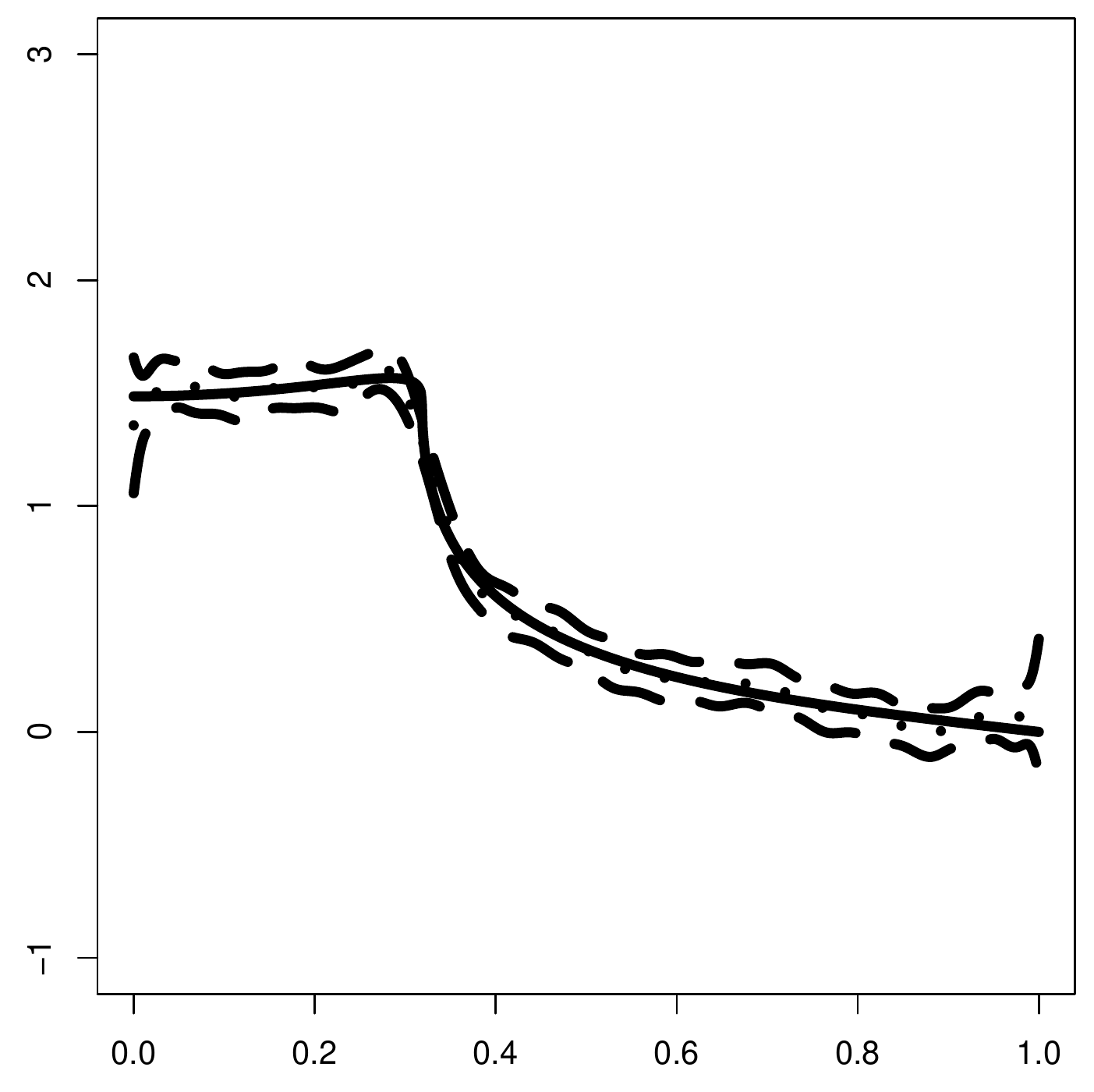}
\caption{Frequentist: \\ \centering $n=2000$}
\end{subfigure}
\caption{Dots: posterior mean (top) and $\widehat{f}$ (bottom), Solid: true function, Dashes: $95\%$ $L_\infty$-credible (top) and confidence (bottom) bands.}
\label{fig:covband}
\end{figure}
Table \ref{tab:covband} shows the coverage of $95\%$ simultaneous credible and confidence bands. At $n=100$, the apparent higher coverage of the confidence bands is due to the positive bias of $\widetilde{\sigma}_n^2$ for small $n$. From $n=300$ onward, the coverage of both credible and confidence bands steadily increase with $n$. The corresponding graphical representations of these bands are shown in Figure \ref{fig:covband}. The top panel corresponds to the proposed Bayesian method, where the dotted line stands for the posterior mean and dashed lines for the  $95\%$ credible band. The bottom panel corresponds to the frequentist method of \citep{localspline}, where the dotted line standing for the least squares estimator $\widehat{f}$ and the dashed lines for the $95\%$ $L_\infty$-confidence bands. In both panels, the solid line is the true function $f_0$. Observe that the credible bands have fixed length, while the confidence bands have varying lengths. This is because the procedure of \citep{localspline} is based on the supremum of the scaled absolute differences. Therefore for the latter we present both average and maximum radius. The frequentist method has larger width at the endpoints due to the fact that there are fewer observations, and this results in larger maximum radius.

\section{Proofs}\label{sec:proof}
We shall repeatedly use the following fact about approximation power of tensor product B-splines given by (12.37) of \citep{lschumaker}.

For any $R>0$, if $\|f_0\|_{\boldsymbol{\alpha},\infty}\leq R$,  there exists a $\boldsymbol{\theta}_\infty\in\mathbb{R}^J$ such that for constant $C>0$ depending only on $\boldsymbol{\alpha}$, $\boldsymbol{q}$ and $d$, we have
\begin{equation}\label{eq:tapprox}
\|\boldsymbol{b}_{\boldsymbol{J},\boldsymbol{q}}(\cdot)^T\boldsymbol{\theta}_\infty-f_0\|_\infty\leq C\sum_{k=1}^dJ_k^{-\alpha_k}\left\|\frac{\partial^{\alpha_k}}{\partial x_k^{\alpha_k}}f_0\right\|_\infty\lesssim\sum_{k=1}^dJ_k^{-\alpha_k}.
\end{equation}
Since $\|\boldsymbol{b}_{\boldsymbol{J},\boldsymbol{q}}(\cdot)^T\boldsymbol{\theta}_\infty\|_\infty
\leq\|f_0\|_{\boldsymbol{\alpha},\infty}+C\sum_{k=1}^dJ_k^{-\alpha_k}\|f_0\|_{\boldsymbol{\alpha},\infty}\lesssim R+d$,
\begin{align}\label{eq:tthetainfty}
\sup_{\|f_0\|_{\boldsymbol{\alpha},\infty}\leq R}\|\boldsymbol{\theta}_\infty\|_\infty\lesssim\sup_{\|f_0\|_{\boldsymbol{\alpha},\infty}\leq R}\|\boldsymbol{b}_{\boldsymbol{J},\boldsymbol{q}}(\cdot)^T\boldsymbol{\theta}_\infty\|_\infty=O(1),
\end{align}
by (12.25) of \citep{lschumaker}. An extension of the approximation result for derivatives is given by the following lemma.

\begin{lemma}\label{lem:derivbound}
There exists $C>0$ depending only on $\boldsymbol{\alpha}$, $\boldsymbol{q}$ and $d$ such that for $f_0\in \mathcal{H}^{\boldsymbol{\alpha}}([0,1]^d)$,
\begin{equation*}
\|\boldsymbol{b}_{\boldsymbol{J},\boldsymbol{q}-\boldsymbol{r}}(\cdot)^T\boldsymbol{W}_{\boldsymbol{r}}\boldsymbol{\theta}_\infty-D^{\boldsymbol{r}}f_0\|_\infty\leq C\left(\sum_{k=1}^dJ_k^{-(\alpha_k-r_k)}\left\|D^{(\alpha_k-r_k)\boldsymbol{e}_k}D^{\boldsymbol{r}}f_0\right\|_\infty\right).
\end{equation*}
\end{lemma}
\begin{proof}
Let $I_{j_1,\dotsc,j_d}=\prod_{k=1}^d[t_{k,j_k-q_k},t_{k,j_k}]$. Define a bounded linear operator
$Qf(\boldsymbol{x})=\sum_{j_1=1}^{J_1}\dotsi\sum_{j_d=1}^{J_d}(\lambda_{j_1,\dotsc,j_d}f)\prod_{k=1}^dB_{j_k,q_k}(x_k)$
on $\mathcal{H}^{\boldsymbol{\alpha}}([0,1]^d)$,
where $\lambda_{j_1,\dotsc,j_d}=\prod_{k=1}^d\lambda_{j_k}$ and $\lambda_{j_k}$ is the dual basis of $B_{j_k,q_k}(\cdot)$, i.e., $\lambda_{j_k}$ is a linear functional such that $\lambda_{i_k}B_{j_k,q_k}(\cdot)=\Ind_{\{i_k=j_k\}}(\cdot)$ for $k=1,\dotsc,d$ (see Section 4.6 of \citep{lschumaker}). Using Theorem 13.20 of \citep{lschumaker}, there exists a tensor-product Taylor's polynomial $p_{j_1,\dotsc,j_d}(\boldsymbol{x})$ such that
\begin{equation*}
\|D^{\boldsymbol{r}}(f_0-p_{j_1,\dotsc,j_d})|_{I_{j_1,\dotsc,j_d}}\|_\infty\leq C\sum_{k=1}^dJ_k^{-(\alpha_k-r_k)}\left\|D^{(\alpha_k-r_k)\boldsymbol{e}_k}D^{\boldsymbol{r}}f_0
\middle|_{I_{j_1,\dotsc,j_d}}\right\|_\infty,
\end{equation*}
where $f|_{I_{j_1,\dotsc,j_d}}$ is the restriction of $f$ onto $I_{j_1,\dotsc,j_d}$ and $C>0$ depends only on $\boldsymbol{\alpha}$, $\boldsymbol{q}$ and $d$. By equations (12.30) and (12.31) of Theorem 12.6 in \citep{lschumaker}, $\|(D^{\boldsymbol{r}}f_0-QD^{\boldsymbol{r}}f_0)|_{I_{j_1,\dotsc,j_d}}\|_\infty$ is bounded above by
\begin{align*}
&\|D^{\boldsymbol{r}}(f_0-p_{j_1,\dotsc,j_d})|_{I_{j_1,\dotsc,j_d}}\|_\infty+\|Q(D^{\boldsymbol{r}}f_0-D^{\boldsymbol{r}}p_{j_1,\dotsc,j_d})|_{I_{j_1,\dotsc,j_d}}\|_\infty\\
&\qquad\leq C\|D^{\boldsymbol{r}}(f_0-p_{j_1,\dotsc,j_d})|_{I_{j_1,\dotsc,j_d}}\|_\infty\\
&\qquad\leq C\sum_{k=1}^dJ_k^{-(\alpha_k-r_k)}\left\|D^{(\alpha_k-r_k)\boldsymbol{e}_k}D^{\boldsymbol{r}}f_0
\middle|_{I_{j_1,\dotsc,j_d}}\right\|_\infty.
\end{align*}
Since $QD^{\boldsymbol{r}}f_0=D^{\boldsymbol{r}}Qf_0$, identifying $(\boldsymbol{\theta}_\infty)_{j_1,\dotsc,j_d}$ from \eqref{eq:tapprox} with $\lambda_{j_1,\dotsc,j_d}f_0$ and applying equations (15) and (16) of Chapter X in \citep{deBoor}, we see that $QD^{\boldsymbol{r}}f_0= \boldsymbol{b}_{\boldsymbol{J},\boldsymbol{q}-\boldsymbol{r}}(\cdot)^T\boldsymbol{W}_{\boldsymbol{r}}\boldsymbol{\theta}_\infty$. Now sum both sides over $1\leq j_k\leq J_k, k=1,\dotsc,d$.
\end{proof}

\begin{proof}[Proof of Proposition \ref{th:tsigma2con}]
Define $\boldsymbol{U}=(\boldsymbol{B\Omega B}^T+\boldsymbol{I}_n)^{-1}$ and $J=\prod_{k=1}^dJ_k$. By equation (33) of page 355 in \citep{expectquad},
\begin{align}
|\mathrm{E}_0(\widehat{\sigma}_n^2)-\sigma_0^2|&= |n^{-1}\sigma_0^2\mathrm{tr}(\boldsymbol{U})-\sigma_0^2|+n^{-1}(\boldsymbol{F}_0-\boldsymbol{B\eta})^T
\boldsymbol{U}(\boldsymbol{F}_0-\boldsymbol{B\eta})\nonumber\\
&\lesssim n^{-1}[\mathrm{tr}(\boldsymbol{I}_n-\boldsymbol{U})+(\boldsymbol{F}_0-\boldsymbol{B\theta}_\infty)^T
\boldsymbol{U}(\boldsymbol{F}_0-\boldsymbol{B\theta}_\infty)\nonumber\\
&\qquad+(\boldsymbol{B\theta}_\infty-\boldsymbol{B\eta})^T
\boldsymbol{U}(\boldsymbol{B\theta}_\infty-\boldsymbol{B\eta})]\label{eq:tsigma0bias},
\end{align}
where we used $(\boldsymbol{x}+\boldsymbol{y})^T\boldsymbol{D}(\boldsymbol{x}+\boldsymbol{y})\leq 2\boldsymbol{x}^T\boldsymbol{Dx}+2\boldsymbol{y}^T\boldsymbol{Dy}$ for any $\boldsymbol{D}\geq\boldsymbol{0}$. Let $\boldsymbol{P}_{\boldsymbol{B}}=\boldsymbol{B}(\boldsymbol{B}^T\boldsymbol{B})^{-1}\boldsymbol{B}^T$. Let $\boldsymbol{A}$ be an $m\times m$ matrix, $\boldsymbol{C}$ an $m\times r$ matrix, $\boldsymbol{T}$ an $r\times r$ matrix, and $\boldsymbol{W}$ an $r\times m$ matrix, with $\boldsymbol{A}$ and $\boldsymbol{T}$ invertible. Then by  the binomial inverse theorem [Theorem 18.2.8 of \citep{davidmatrix}]
\begin{align}\label{eq:bit}
(\boldsymbol{A}+\boldsymbol{CTW})^{-1}=\boldsymbol{A}^{-1}-\boldsymbol{A}^{-1}\boldsymbol{C}(\boldsymbol{T}^{-1}+\boldsymbol{WA}^{-1}\boldsymbol{C})^{-1}\boldsymbol{WA}^{-1}.
\end{align}
Therefore, two applications of \eqref{eq:bit} to $\boldsymbol{U}$ yield
\begin{align}\label{eq:tbiasinter}
(\boldsymbol{B\Omega B}^T+\boldsymbol{I}_n)^{-1}&=\boldsymbol{I}_n-\boldsymbol{B}
(\boldsymbol{B}^T\boldsymbol{B}+\boldsymbol{\Omega}^{-1})^{-1}\boldsymbol{B}^T=\boldsymbol{I}_n-\boldsymbol{P}_{\boldsymbol{B}}
+\boldsymbol{V},
\end{align}
where $\boldsymbol{V}=\boldsymbol{B}(\boldsymbol{B}^T\boldsymbol{B})^{-1}
[\boldsymbol{\Omega}+(\boldsymbol{B}^T\boldsymbol{B})^{-1}]^{-1}(\boldsymbol{B}^T\boldsymbol{B})^{-1}\boldsymbol{B}^T
\geq\boldsymbol{0}$. Hence the first term in \eqref{eq:tsigma0bias} is
\begin{align}
n^{-1}\mathrm{tr}(\boldsymbol{P}_{\boldsymbol{B}}-\boldsymbol{V})&\leq n^{-1}\mathrm{tr}(\boldsymbol{P}_{\boldsymbol{B}})
=J/n.\label{eq:tsigma0biasbound1}
\end{align}
Note $\boldsymbol{U}\leq\boldsymbol{I}_n$ since $\boldsymbol{B\Omega B}^T\geq\boldsymbol{0}$, and the second term in \eqref{eq:tsigma0bias} is bounded by
\begin{align}
{n}^{-1}\|\boldsymbol{U}\|_{(2,2)}\|\boldsymbol{F}_0-\boldsymbol{B\theta}_\infty\|^2
\leq\|\boldsymbol{F}_0-\boldsymbol{B\theta}_\infty\|_\infty^2
\lesssim\sum_{k=1}^dJ_k^{-2\alpha_k},\label{eq:tsigma0biasbound2}
\end{align}
in view of \eqref{eq:tapprox}. By \eqref{eq:tbiasinter} and $(\boldsymbol{I}-\boldsymbol{P}_{\boldsymbol{B}})\boldsymbol{B}=\boldsymbol{0}$, the last term in \eqref{eq:tsigma0bias} is $n^{-1}(\boldsymbol{\theta}_\infty-\boldsymbol{\eta})^T[\boldsymbol{\Omega}+(\boldsymbol{B}^T\boldsymbol{B})^{-1}]^{-1}
(\boldsymbol{\theta}_\infty-\boldsymbol{\eta})$, which is bounded above by
\begin{align}
{n}^{-1}\left(c_1+{C_2}^{-1}{J}/{n}\right)^{-1}J\|\boldsymbol{\theta}_\infty-\boldsymbol{\eta}\|_\infty^2
\lesssim {J}/{n},\label{eq:tsigma0biasbound3}
\end{align}
where we used \eqref{assump:tprior} and \eqref{eq:tdiagonal} to bound the maximum eigenvalue of $[\boldsymbol{\Omega}+(\boldsymbol{B}^T\boldsymbol{B})^{-1}]^{-1}$. By \eqref{eq:tthetainfty} and assumption on the prior, $\|\boldsymbol{\theta}_\infty-\boldsymbol{\eta}\|_\infty^2=O(1)$. Combining the bounds in \eqref{eq:tsigma0biasbound1}, \eqref{eq:tsigma0biasbound2} and \eqref{eq:tsigma0biasbound3} into \eqref{eq:tsigma0bias}, we obtain
$|\mathrm{E}_0(\widehat{\sigma}_n^2)-\sigma_0^2|\lesssim J/n+\sum_{k=1}^dJ_k^{-2\alpha_k}$.

Let $\boldsymbol{Y}=\boldsymbol{F}_0+\boldsymbol{\varepsilon}$ and write $n\widehat{\sigma}_n^2=(\boldsymbol{F}_0-\boldsymbol{B\eta})^T\boldsymbol{U}(\boldsymbol{F}_0-\boldsymbol{B\eta})+2(\boldsymbol{F}_0-\boldsymbol{B\eta})^T\boldsymbol{U\varepsilon}
+\boldsymbol{\varepsilon}^T\boldsymbol{U\varepsilon}$. Using the fact $\mathrm{Var}(T_1+T_2)\leq2\mathrm{Var}(T_1)+2\mathrm{Var}(T_2)$, it follows that $\mathrm{Var}_0(\widehat{\sigma}_n^2)$ is bounded up to a constant multiple by
\begin{align}\label{eq:tsigma0var}
&n^{-2}[(\boldsymbol{F}_0-\boldsymbol{B\theta}_\infty)^T\boldsymbol{U}^2(\boldsymbol{F}_0-\boldsymbol{B\theta}_\infty)\nonumber\\
&\qquad+(\boldsymbol{B\theta}_\infty-\boldsymbol{B\eta})^T\boldsymbol{U}^2(\boldsymbol{B\theta}_\infty-\boldsymbol{B\eta})
+\mathrm{Var}_0(\boldsymbol{\varepsilon}^T\boldsymbol{U}\boldsymbol{\varepsilon})].
\end{align}
In view of \eqref{eq:tapprox} and $\boldsymbol{U}\leq\boldsymbol{I}_n$, the first term above is bounded by
\begin{align}
{n}^{-2}\|\boldsymbol{U}\|_{(2,2)}^2\|\boldsymbol{F}_0-\boldsymbol{B\theta}_\infty\|^2
\leq {n}^{-1}\|\boldsymbol{F}_0-\boldsymbol{B\theta}_\infty\|_\infty^2\lesssim {n}^{-1}\sum_{k=1}^dJ_k^{-2\alpha_k}.\label{eq:tsigma0varbound1}
\end{align}
By the idempotency of $\boldsymbol{I}_n-\boldsymbol{P}_{\boldsymbol{B}}$ and $(\boldsymbol{I}_n-\boldsymbol{P}_{\boldsymbol{B}})\boldsymbol{B}=\boldsymbol{0}$, we have that $\boldsymbol{B}^T(\boldsymbol{I}_n-\boldsymbol{P}_{\boldsymbol{B}}
+\boldsymbol{V})^2\boldsymbol{B}$ is
\begin{align*}
\boldsymbol{B}^T\boldsymbol{V}^2\boldsymbol{B}&=[\boldsymbol{\Omega}+(\boldsymbol{B}^T\boldsymbol{B})^{-1}]^{-1}
(\boldsymbol{B}^T\boldsymbol{B})^{-1}[\boldsymbol{\Omega}+(\boldsymbol{B}^T\boldsymbol{B})^{-1}]^{-1}\\
&\leq[\boldsymbol{\Omega}+(\boldsymbol{B}^T\boldsymbol{B})^{-1}]^{-1}\leq\boldsymbol{B}^T\boldsymbol{B}.
\end{align*}
Therefore, in view of \eqref{eq:tbiasinter}, the second term in \eqref{eq:tsigma0var} is bounded by
\begin{align}
(\boldsymbol{\theta}_\infty-\boldsymbol{\eta})^T\boldsymbol{B}^T\boldsymbol{B}(\boldsymbol{\theta}_\infty-\boldsymbol{\eta})/n^2
\leq  J\|\boldsymbol{B}^T\boldsymbol{B}\|_{(2,2)}\|\boldsymbol{\theta}_\infty-\boldsymbol{\eta}\|_\infty^2/n
\lesssim n^{-1},\label{eq:tsigma0varbound2}
\end{align}
where we used \eqref{eq:tdiagonal} to bound $\|\boldsymbol{B}^T\boldsymbol{B}\|_{(2,2)}$, while $\|\boldsymbol{\theta}_\infty-\boldsymbol{\eta}\|_\infty^2$ is bounded using \eqref{eq:tthetainfty} and the assumption on the prior. By Lemma \ref{lem:tvarquad}, the last term in \eqref{eq:tsigma0var} is $O(n^{-1})$. Combining this with the bounds established in \eqref{eq:tsigma0varbound1} and \eqref{eq:tsigma0varbound2} into \eqref{eq:tsigma0var}, we obtain $\mathrm{Var}_0(\widehat{\sigma}_n^2)\lesssim n^{-1}$. If $J_k\asymp n^{\alpha^{*}/\{\alpha_k(2\alpha^{*}+d)\}}$ for $k=1,\dotsc,n$, the mean square error is
\begin{align}\label{eq:sigma2mse}
\mathrm{E}_0(\widehat{\sigma}_n^2-\sigma_0^2)^2\lesssim {n}^{-1}+ J^2 n^{-2}
+\sum_{k=1}^dJ_k^{-4\alpha_k}\lesssim {n}^{-1}+n^{-{4\alpha^{*}}/{(2\alpha^{*}+d)}},
\end{align}
which implies the first assertion.

For the assertion (b), observe that
\begin{eqnarray*}
\mathrm{E}(\sigma^2|\boldsymbol{Y})&=& {\beta_2}(\beta_1+n-2)^{-1}+{n}(\beta_1+n-2)^{-1}\widehat{\sigma}_n^2,\nonumber\\
\mathrm{Var}(\sigma^2|\boldsymbol{Y})&=& {4}(\beta_1+n-4)^{-1}({\beta_2}(\beta_1+n-2)^{-1}
+{n}(\beta_1+n-2)^{-1}\widehat{\sigma}_n^2)^2.
\end{eqnarray*}
Applying Markov's inequality, the posterior for $\sigma^2$ is seen to concentrate around $\widehat{\sigma}_n^2$ at the rate $n^{-1/2}$, so the assertion follows from (a).

Assertion (c) can be concluded from an anisotropic extension of the estimates obtained in the proof of Theorem~4.1 together with Theorem~A.1 of \citep{dejonge2013}. Indeed the posterior contracts at the rate $n^{-\alpha^*/(2\alpha^*+d)}$, and actually at the rate $n^{-1/2}$ for $\alpha^*>d/2$ by an anisotropic extension of their Theorem~4.1. Consistency can also be approached directly from the marginal model $p_{n,\sigma}$ for $\boldsymbol{Y}$ given $\sigma$, where $f$ is integrated out, by a Schwartz-type  posterior consistency argument using the test $|\widehat{\sigma}_n-\sigma_0|>\epsilon$, which is consistent at the true density
$p_{0,n}$ by part (a). The only departure from Schwartz's argument is that in the present case it is convenient to directly establish that for any $c>0$,
$e^{cn} \int ({p_{n,\sigma}}/{p_{0,n}}) d\Pi(\sigma)\to \infty$ in probability under $p_{0,n}$ using the consistency of $\widehat{\sigma}_n$ at $\sigma_0$.
\end{proof}

We write $\mathcal{U}_n$ for a shrinking neighborhood of $\sigma_0$ such that
with probability tending to one, $\widehat{\sigma}_n\in \mathcal{U}_n$ and  $\Pi(\sigma\in \mathcal{U}_n|\boldsymbol{Y})\to 1$.
We write $D^{\boldsymbol{r}}\widetilde{f}$ for $\mathrm{E}(D^{\boldsymbol{r}}f|\boldsymbol{Y})= \boldsymbol{A}_{\boldsymbol{r}}\boldsymbol{Y}+\boldsymbol{c}_{\boldsymbol{r}}\boldsymbol{\eta}$. Recall that $\epsilon_{n,\boldsymbol{r}}=n^{-\alpha^{*}\{1-\sum_{k=1}^d(r_k/\alpha_k)\}/(2\alpha^{*}+d)}$ and $\epsilon_{n,\boldsymbol{r},\infty}=(\log{n}/n)^{\alpha^{*}\{1-\sum_{k=1}^d(r_k/\alpha_k)\}/(2\alpha^{*}+d)}$.

\begin{proof}[Proof of Theorem \ref{prop:tsupnormfprime}] Recall that at $\boldsymbol{x}\in[0,1]^d$, $(D^{\boldsymbol{r}}f(\boldsymbol{x})|\boldsymbol{Y},\sigma)\sim\mathrm{N}(D^{\boldsymbol{r}}\widetilde{f}(\boldsymbol{x}),\sigma^2\Sigma_{\boldsymbol{r}}(\boldsymbol{x},\boldsymbol{x}))$, with $\Sigma_{\boldsymbol{r}}(\boldsymbol{x},\boldsymbol{x})$ given in \eqref{eq:tsigmar}.
Under $P_0$, $D^{\boldsymbol{r}}\widetilde{f}(\boldsymbol{x})$ is a sub-Gaussian variable with mean $\boldsymbol{A}_{\boldsymbol{r}}(\boldsymbol{x})\boldsymbol{F}_0+\boldsymbol{c}_{\boldsymbol{r}}(\boldsymbol{x})\boldsymbol{\eta}$ and variance  $\sigma_0^2\Psi_{\boldsymbol{r}}(\boldsymbol{x},\boldsymbol{x})$, where $\Psi_{\boldsymbol{r}}(\boldsymbol{x},\boldsymbol{y})$ is
\begin{align*}
\boldsymbol{b}_{\boldsymbol{J},\boldsymbol{q}-\boldsymbol{r}}(\boldsymbol{x})^T\boldsymbol{W}_{\boldsymbol{r}}\left(\boldsymbol{B}^T\boldsymbol{B}
+\boldsymbol{\Omega}^{-1}\right)^{-1}\boldsymbol{B}^T\boldsymbol{B}\left(\boldsymbol{B}^T\boldsymbol{B}
+\boldsymbol{\Omega}^{-1}\right)^{-1}\boldsymbol{W}_{\boldsymbol{r}}^T\boldsymbol{b}_{\boldsymbol{J},\boldsymbol{q}-\boldsymbol{r}}(\boldsymbol{y}).
\end{align*}
Note that the posterior variance $\sigma^2\Sigma_{\boldsymbol{r}}(\boldsymbol{x},\boldsymbol{x})$ of $D^{\boldsymbol{r}} f$ does not depend on $\boldsymbol{Y}$, while  $D^{\boldsymbol{r}}\widetilde{f}(\boldsymbol{x})$ does not depend on $\sigma$. Therefore uniformly on $\|f_0\|_{\boldsymbol{\alpha},\infty}\leq R$,
\begin{align}\label{eq:tMarkov}
&\mathrm{E}_0\sup_{\sigma\in\mathcal{U}_n}\mathrm{E}([D^{\boldsymbol{r}}f(\boldsymbol{x})-D^{\boldsymbol{r}}f_0(\boldsymbol{x})]^2|\boldsymbol{Y},\sigma)\nonumber\\
&\qquad =\sup_{\sigma\in\mathcal{U}_n}\mathrm{E}([D^{\boldsymbol{r}}f(\boldsymbol{x})-D^{\boldsymbol{r}}\widetilde{f}(\boldsymbol{x})]^2|\sigma)
+\mathrm{E}_0[D^{\boldsymbol{r}}\widetilde{f}(\boldsymbol{x}) -D^{\boldsymbol{r}}f_0(\boldsymbol{x})]^2\nonumber\\
&\qquad=\sup_{\sigma\in\mathcal{U}_n}\sigma^2\Sigma_{\boldsymbol{r}}(\boldsymbol{x},\boldsymbol{x})
+\sigma_0^2\Psi_{\boldsymbol{r}}(\boldsymbol{x},\boldsymbol{x})+[\mathrm{E}_0D^{\boldsymbol{r}}\widetilde{f}(\boldsymbol{x})-D^{\boldsymbol{r}}f_0(\boldsymbol{x})]^2.
\end{align}
To bound $\sigma^2\Sigma_{\boldsymbol{r}}(\boldsymbol{x},\boldsymbol{x})$, first observe that $\|\boldsymbol{b}_{\boldsymbol{J},\boldsymbol{q}-\boldsymbol{r}}(\boldsymbol{x})\|^2$ is bounded by
\begin{equation}\label{eq:tbnorm2}
\prod_{k=1}^d\max_{1\leq j_k\leq J_k}B_{j_k,q_k-r_k}(x_k)\sum_{j_1=1}^{J_1-r_1}\dotsi\sum_{j_d=1}^{J_d-r_d}B_{j_k,q_k-r_k}(x_k)\leq1.
\end{equation}
In view of \eqref{eq:firstdiff}, each row of $\boldsymbol{W}_{\boldsymbol{r}}$ has $\prod_{k=1}^d(r_k+1)$ nonzero entries and each column has at most $\prod_{k=1}^d(r_k+1)$ nonzero entries. Then by Lemmas \ref{lem:Wr} and \ref{lem:tasymp}, each of these nonzero entries is of the order $\prod_{k=1}^d\Delta_k^{-r_k}\asymp\prod_{k=1}^dJ_k^{r_k}$. Hence, both $\|\boldsymbol{W}_{\boldsymbol{r}}\|_{(\infty,\infty)}$ and $\|\boldsymbol{W}^T_{\boldsymbol{r}}\|_{(\infty,\infty)}$ are $O(\prod_{k=1}^dJ_k^{r_k})$. Thus,
\begin{align}
\|\boldsymbol{W}_{\boldsymbol{r}}^T\boldsymbol{W}_{\boldsymbol{r}}\|_{(2,2)}&\leq\|\boldsymbol{W}_{\boldsymbol{r}}^T
\boldsymbol{W}_{\boldsymbol{r}}\|_{(\infty,\infty)}\lesssim\prod_{k=1}^dJ_k^{2r_k}.\label{eq:tww}
\end{align}
By the Cauchy-Schwarz inequality, \eqref{eq:tbnorm2}, \eqref{eq:tww} and \eqref{eq:teigen},  $\sigma^2\Sigma_{\boldsymbol{r}}(\boldsymbol{x},\boldsymbol{x})$ over $\sigma\in\mathcal{U}_n$ is uniformly bounded by
\begin{align}
&(\sigma_0^2+o(1))\|\boldsymbol{b}_{\boldsymbol{J},\boldsymbol{q}-\boldsymbol{r}}(\boldsymbol{x})\|^2
\|\boldsymbol{W}_{\boldsymbol{r}}^T\boldsymbol{W}_{\boldsymbol{r}}\|_{(2,2)}\
\left\|\left(\boldsymbol{B}^T\boldsymbol{B}+\boldsymbol{\Omega}^{-1}\right)^{-1}\right\|_{(2,2)}\nonumber\\
&\qquad\lesssim\left(C_1n\prod_{k=1}^dJ_k^{-1}+{c_2}^{-1}\right)^{-1}\left(\prod_{k=1}^dJ_k^{2r_k}\right)
\lesssim {n}^{-1}\prod_{k=1}^dJ_k^{2r_k+1},\label{eq:tvarbound}
\end{align}
Using \eqref{eq:teigen}, \eqref{eq:tdiagonal}, \eqref{eq:tbnorm2} and \eqref{eq:tww},
the variance $\sigma_0^2\Psi_{\boldsymbol{r}}(\boldsymbol{x},\boldsymbol{x})$ of $D^{\boldsymbol{r}}f(\boldsymbol{x})$  is bounded by
\begin{align}
&\sigma_0^2\left\|\left(\boldsymbol{B}^T\boldsymbol{B}+\boldsymbol{\Omega}^{-1}\right)^{-1}\right\|^2_{(2,2)}\|\boldsymbol{B}^T\boldsymbol{B}\|_{(2,2)}
\|\boldsymbol{b}_{\boldsymbol{J},\boldsymbol{q}-\boldsymbol{r}}(\boldsymbol{x})\|^2\|\boldsymbol{W}_{\boldsymbol{r}}^T\boldsymbol{W}_{\boldsymbol{r}}\|_{(2,2)}\nonumber\\
&\qquad\lesssim \left({n}^{-1}\prod_{k=1}^dJ_k\right)^2\left(n\prod_{k=1}^dJ_k^{-1}\right)\left(\prod_{k=1}^dJ_k^{2r_k}\right)
\lesssim {n}^{-1}\prod_{k=1}^dJ_k^{2r_k+1}.\label{eq:tvarbiasbound}
\end{align}
The last term in \eqref{eq:tMarkov} is bounded as
\begin{align*}
&|\mathrm{E}_0D^{\boldsymbol{r}}\widetilde{f}(\boldsymbol{x})-D^{\boldsymbol{r}}f_0(\boldsymbol{x})|\\
&\leq\left|\boldsymbol{b}_{\boldsymbol{J},\boldsymbol{q}-\boldsymbol{r}}(\boldsymbol{x})^T\boldsymbol{W}_{\boldsymbol{r}}\left(\boldsymbol{B}^T\boldsymbol{B}+\boldsymbol{\Omega}^{-1}\right)^{-1}\left(\boldsymbol{B}^T\boldsymbol{F}_0+\boldsymbol{\Omega}^{-1}\boldsymbol{\eta}\right)-\boldsymbol{b}_{\boldsymbol{J},\boldsymbol{q}-\boldsymbol{r}}(\boldsymbol{x})^T\boldsymbol{W}_{\boldsymbol{r}}\boldsymbol{\theta}_\infty\right|\\
&\qquad+|D^{\boldsymbol{r}}f_0(\boldsymbol{x})-\boldsymbol{b}_{\boldsymbol{J},\boldsymbol{q}-\boldsymbol{r}}(\boldsymbol{x})^T\boldsymbol{W}_{\boldsymbol{r}}\boldsymbol{\theta}_\infty|.
\end{align*}
By bounding the second term using Lemma \ref{lem:derivbound} and using  $\|\boldsymbol{b}_{\boldsymbol{J},\boldsymbol{q}-\boldsymbol{r}}(\boldsymbol{x})\|_1=1$, the right hand side above, up to $O(\sum_{k=1}^dJ_k^{-(\alpha_k-r_k)})$, is bounded by
\begin{align*}
&\left|\boldsymbol{b}_{\boldsymbol{J},\boldsymbol{q}-\boldsymbol{r}}(\boldsymbol{x})^T\boldsymbol{W}_{\boldsymbol{r}}\left(\boldsymbol{B}^T\boldsymbol{B}+\boldsymbol{\Omega}^{-1}\right)^{-1}\left[\boldsymbol{B}^T(\boldsymbol{F}_0-\boldsymbol{B}\boldsymbol{\theta}_\infty)+\boldsymbol{\Omega}^{-1}(\boldsymbol{\eta}-\boldsymbol{\theta}_\infty)\right]\right|\\
&\leq\left\|\left(\boldsymbol{B}^T\boldsymbol{B}+\boldsymbol{\Omega}^{-1}\right)^{-1}\right\|_{(\infty,\infty)}\|\boldsymbol{W}_{\boldsymbol{r}}\|_{(\infty,\infty)}\Big\{\|\boldsymbol{B}^T(\boldsymbol{F}_0-\boldsymbol{B}\boldsymbol{\theta}_\infty)\|_\infty\\
&\qquad+\left\|\boldsymbol{\Omega}^{-1}\right\|_{(\infty,\infty)}(\|\boldsymbol{\theta}_\infty\|_\infty+\|\boldsymbol{\eta}\|_\infty)\Big\}.
\end{align*}
Since $\boldsymbol{\Omega}^{-1}$ is $\boldsymbol{m}$-banded with fixed $\boldsymbol{m}$ and has uniformly bounded entries,  $\|\boldsymbol{\Omega}^{-1}\|_{(\infty,\infty)}=O(1)$. As  $\boldsymbol{B}^T\boldsymbol{B}$ is $\boldsymbol{q}$-banded, Lemma \ref{lem:tband} and \eqref{eq:teigen} imply that $\|(\boldsymbol{B}^T\boldsymbol{B}+\boldsymbol{\Omega}^{-1})^{-1}\|_{(\infty,\infty)}\lesssim n^{-1}\prod_{k=1}^dJ_k$. By \eqref{eq:tww}, we have $\|\boldsymbol{W}_{\boldsymbol{r}}\|_{(\infty,\infty)}\lesssim\prod_{k=1}^dJ_k^{r_k}$. Also, $\|\boldsymbol{\theta}_\infty\|_\infty$ and $\|\boldsymbol{\eta}\|_\infty$ are both $O(1)$ by \eqref{eq:tthetainfty} and the assumption on the prior. Using the non-negativity of B-splines, Lemma \ref{lem:tsumB} and \eqref{eq:tapprox}, uniformly on  $\|f_0\|_{\boldsymbol{\alpha},\infty}\leq R$, we bound $\|\boldsymbol{B}^T(\boldsymbol{F}_0-\boldsymbol{B}\boldsymbol{\theta}_\infty)\|_\infty$ by
\begin{align*}
&\max_{1\leq j_k\leq J_k,k=1,\dotsc,d}\sum_{i=1}^n\prod_{k=1}^dB_{j_k,q_k}(X_{ik})|f_0(\boldsymbol{X}_i)-\boldsymbol{b}_{\boldsymbol{J},\boldsymbol{q}}(\boldsymbol{X}_i)^T\boldsymbol{\theta}_\infty|\\
&\quad\lesssim\sum_{k=1}^dJ_k^{-\alpha_k}\max_{1\leq j_k\leq J_k,k=1,\dotsc,d}\sum_{i=1}^n\prod_{k=1}^dB_{j_k,q_k}(X_{ik})\lesssim n\sum_{k=1}^dJ_k^{-\alpha_k}\left(\prod_{k=1}^dJ_k^{-1}\right).
\end{align*}
Therefore, combining the bounds obtained and squaring the bias of $D^{\boldsymbol{r}}\widetilde{f}$, we have for any $\boldsymbol{x}\in[0,1]^d$ uniformly on $\|f_0\|_{\boldsymbol{\alpha},\infty}\leq R$,
\begin{align}\label{eq:tbiasbound}
|\mathrm{E}_0D^{\boldsymbol{r}}\widetilde{f}(\boldsymbol{x})-D^{\boldsymbol{r}}f_0(\boldsymbol{x})|^2
\lesssim\prod_{k=1}^dJ_k^{2r_k}\left(n^{-2}\prod_{k=1}^dJ_k^2+\sum_{k=1}^dJ_k^{-2\alpha_k}\right).
\end{align}
Let $P_{n,\boldsymbol{r}}(\boldsymbol{x})=\mathrm{E}_0\sup_{\sigma\in\mathcal{U}_n}\mathrm{E}([D^{\boldsymbol{r}}f(\boldsymbol{x})-D^{\boldsymbol{r}}f_0(\boldsymbol{x})]^2|\boldsymbol{Y},\sigma)$. Combining \eqref{eq:tvarbound}, \eqref{eq:tvarbiasbound} and \eqref{eq:tbiasbound} into \eqref{eq:tMarkov},
\begin{align}
\sup_{\|f_0\|_{\boldsymbol{\alpha},\infty}\leq R}P_{n,\boldsymbol{r}}(\boldsymbol{x})&\lesssim\frac{1}{n}\prod_{k=1}^dJ_k^{2r_k+1}+\prod_{k=1}^dJ_k^{2r_k}\left(\frac{1}{n^2}\prod_{k=1}^dJ_k^2+\sum_{k=1}^dJ_k^{-2\alpha_k}\right)\nonumber\\
&\lesssim\prod_{k=1}^dJ_k^{2r_k}\left(\frac{1}{n}\prod_{k=1}^dJ_k+\sum_{k=1}^dJ_k^{-2\alpha_k}\right),\label{eq:tbound}
\end{align}
since $\prod_{k=1}^dJ_k\leq n$ by the assumption. To balance the orders of the two terms on the right, let $J_k=J^{1/\alpha_k}$ for $k=1,\dotsc,d$. Then the right hand side of \eqref{eq:tbound} reduces to
$O(J^{\sum_{k=1}^d(2r_k+1)/\alpha_k}/n)+O(J^{2(\sum_{k=1}^dr_k/\alpha_k-1)}).$
They will have the same order if $J\asymp n^{\alpha^{*}/(2\alpha^{*}+d)}$, and $J_k=J^{1/\alpha_k}\asymp n^{\alpha^{*}/\{\alpha_k(2\alpha^{*}+d)\}}$ for $k=1,\dotsc,d$. Hence, $P_{n,\boldsymbol{r}}(\boldsymbol{x})=O(\epsilon_{n,\boldsymbol{r}}^2)$ uniformly on $\|f_0\|_{\boldsymbol{\alpha},\infty}\leq R$, implying the first assertion.

For the hierarchical Bayes procedure, the assertion similarly follows from   $\mathrm{E}_0\Pi(|D^{\boldsymbol{r}}f(\boldsymbol{x})-D^{\boldsymbol{r}}f_0(\boldsymbol{x})|>M_n\epsilon_{n,\boldsymbol{r}}|\boldsymbol{Y})\le
{M_n^{-2}\epsilon_{n,\boldsymbol{r}}^{-2}}{P_{n,\boldsymbol{r}}(\boldsymbol{x})}+\mathrm{E}_0\Pi(\sigma\notin \mathcal{U}_n|\boldsymbol{Y})$.
\end{proof}

\begin{proof}[Proof of Theorem \ref{th:tsupnormfprime}]
Recall that $(D^{\boldsymbol{r}}f|\boldsymbol{Y},\sigma)\sim\mathrm{GP}(D^{\boldsymbol{r}}\widetilde{f},\sigma^2\Sigma_{\boldsymbol{r}})$. Let $Z_{n,\boldsymbol{r}}\sim\mathrm{GP}(0,\Sigma_{\boldsymbol{r}})$. Under the true distribution $P_0$, $D^{\boldsymbol{r}}\widetilde{f}$ is a sub-Gaussian process with mean function $\boldsymbol{A}_{\boldsymbol{r}}\boldsymbol{F}_0+\boldsymbol{c}_{\boldsymbol{r}}\boldsymbol{\eta}$ and covariance function $\sigma_0^2\Psi_{\boldsymbol{r}}$. Let $Q_{n,\boldsymbol{r}}$ be a sub-Gaussian process with mean function $0$ and covariance function $\sigma_0^2\Psi_{\boldsymbol{r}}$. Note that $Z_{n,\boldsymbol{r}}$ does not depend on $\boldsymbol{Y}$ and $f_0$, while $D^{\boldsymbol{r}}\widetilde{f}$ does not depend on $\sigma$. Then uniformly on  $\|f_0\|_{\boldsymbol{\alpha},\infty}\leq R$,
\begin{align}
&\mathrm{E}_0\sup_{\sigma^2\in\mathcal{U}_n}\mathrm{E}(\|D^{\boldsymbol{r}}f-D^{\boldsymbol{r}}f_0\|_\infty^2|\boldsymbol{Y},\sigma)\nonumber\\
&\qquad\lesssim\sup_{\sigma\in\mathcal{U}_n}\mathrm{E}(\|D^{\boldsymbol{r}}f-D^{\boldsymbol{r}}\widetilde{f}\|_\infty^2|\sigma)
+\mathrm{E}_0\|D^{\boldsymbol{r}}\widetilde{f}-D^{\boldsymbol{r}}f_0\|_\infty^2\nonumber\\
&\qquad\lesssim\sup_{\sigma\in\mathcal{U}_n}\sigma^2\mathrm{E}\|Z_{n,\boldsymbol{r}}\|_\infty^2+\mathrm{E}\|Q_{n,\boldsymbol{r}}\|_\infty^2+\|\boldsymbol{A}_{\boldsymbol{r}}\boldsymbol{F}_0+\boldsymbol{c}_{\boldsymbol{r}}\boldsymbol{\eta}-D^{\boldsymbol{r}}f_0\|_\infty^2.\label{eq:tsupinter}
\end{align}
Since $Q_{n,\boldsymbol{r}}=\boldsymbol{A}_{\boldsymbol{r}}\boldsymbol{\varepsilon}$, then by Assumption \ref{assump:tf0}, $Q_{n,\boldsymbol{r}}$ is sub-Gaussian with respect to the semi-metric $d(\boldsymbol{t},\boldsymbol{s})=\sqrt{\mathrm{Var}(Q_{n,\boldsymbol{r}}(\boldsymbol{t})-Q_{n,\boldsymbol{r}}(\boldsymbol{s}))}$. Note that $Z_{n,\boldsymbol{r}}$ and $Q_{n,\boldsymbol{r}}$ satisfy the condition for Lemma \ref{lem:supexpect} by Lemma \ref{lem:tincrement}. Applying Lemma \ref{lem:supexpect} with $p=2$, we have for any $0<\delta_n<1$, $\mathrm{E}\|Z_{n,\boldsymbol{r}}\|_\infty^2\lesssim\log{(1/\delta_n)}(n\delta_n^2+\frac{1}{n}\prod_{k=1}^dJ_k^{2r_k+1})$
in view of \eqref{eq:tvarbound}. Similarly, $\mathrm{E}\|Q_{n,\boldsymbol{r}}\|_\infty^2\lesssim\log{(1/\delta_n)}(n\delta_n^2+\frac{1}{n}\prod_{k=1}^dJ_k^{2r_k+1})$ by \eqref{eq:tvarbiasbound}. Setting $\delta_n^2\asymp n^{-2}\prod_{k=1}^dJ_k^{2r_k+1}$,
\begin{align}\label{eq:tvarsupbound}
\mathrm{E}\|Z_{n,\boldsymbol{r}}\|_\infty^2\lesssim\frac{\log{n}}{n}\prod_{k=1}^dJ_k^{2r_k+1},\quad\mathrm{E}\|Q_{n,\boldsymbol{r}}\|_\infty^2\lesssim\frac{\log{n}}{n}\prod_{k=1}^dJ_k^{2r_k+1}.
\end{align}
Since the bound for \eqref{eq:tbiasbound} is uniform for $\boldsymbol{x}\in[0,1]^d$ and $\|f_0\|_{\boldsymbol{\alpha},\infty}\leq R$,
\begin{align}\label{eq:tbiassupbound}
\|\boldsymbol{A}_{\boldsymbol{r}}\boldsymbol{F}_0+\boldsymbol{c}_{\boldsymbol{r}}\boldsymbol{\eta}
-D^{\boldsymbol{r}}f_0\|_\infty^2\lesssim\prod_{k=1}^dJ_k^{2r_k}\left(n^{-2}\prod_{k=1}^dJ_k^2+\sum_{k=1}^dJ_k^{-2\alpha_k}\right).
\end{align}
Combining \eqref{eq:tvarsupbound} and \eqref{eq:tbiassupbound} with  \eqref{eq:tsupinter}, uniformly on $\|f_0\|_{\boldsymbol{\alpha},\infty}\leq R$,
\begin{align*}
\mathrm{E}_0\sup_{\sigma\in\mathcal{U}_n}\mathrm{E}(\|D^{\boldsymbol{r}}f-D^{\boldsymbol{r}}f_0\|_\infty^2|\boldsymbol{Y},\sigma)
\lesssim\prod_{k=1}^dJ_k^{2r_k}\left(\frac{\log{n}}{n}\prod_{k=1}^dJ_k+\sum_{k=1}^dJ_k^{-2\alpha_k}\right).
\end{align*}
To balance the orders of the two terms on the right, let $J_k=J^{1/\alpha_k}$ for $k=1,\dotsc,d$. Then the bound above reduces to
\begin{align*}
O(J^{\sum_{k=1}^d(2r_k+1)/\alpha_k}n^{-1}\log{n})+O(J^{2(\sum_{k=1}^dr_k/\alpha_k-1)})=O(\epsilon_{n,\boldsymbol{r},\infty}^2),
\end{align*}
if $J\asymp(n/\log{n})^{\alpha^{*}/(2\alpha^{*}+d)}$ and  $J_k=J^{1/\alpha_k}\asymp (n/\log{n})^{\alpha^{*}/\{\alpha_k(2\alpha^{*}+d)\}}$ for $k=1,\dotsc,d$.
The rest of the proof can be completed as in Theorem \ref{prop:tsupnormfprime}.
\end{proof}

\begin{proof}[Proof of Theorem \ref{th:tcrpoint}]
Define $t_{n, \boldsymbol{r}, \gamma_n}(\boldsymbol{x})=\inf_{\sigma\in\mathcal{U}_n}z_{\gamma_n/2}\sigma\sqrt{\Sigma_{\boldsymbol{r}}(\boldsymbol{x},\boldsymbol{x})}$. To show $\widehat{\mathcal{C}}_{n, \boldsymbol{r}, \gamma_n}(\boldsymbol{x})$ has asymptotic coverage of $1$, it suffices to show that
\begin{align}\label{eq:tcoverageinf}
\sup_{\|f_0\|_{\boldsymbol{\alpha},\infty}\leq R}P_0\left(|D^{\boldsymbol{r}}f_0(\boldsymbol{x})-D^{\boldsymbol{r}}\widetilde{f}(\boldsymbol{x})|>t_{n, \boldsymbol{r}, \gamma_n}(\boldsymbol{x})\right)\rightarrow0.
\end{align}
Since $z_{\gamma_n/2}\rightarrow\infty$ and $\mathcal{U}_n$ shrinks to $\sigma_0$, we have $t_{n, \boldsymbol{r}, \gamma_n}(\boldsymbol{x})^2\gg\Sigma_{\boldsymbol{r}}(\boldsymbol{x},\boldsymbol{x})$. In view of \eqref{eq:teigen}, $\Sigma_{\boldsymbol{r}}(\boldsymbol{x},\boldsymbol{x})$ is bounded below by $$\lambda_{\mathrm{min}}\{(\boldsymbol{B}^T\boldsymbol{B}
+\boldsymbol{\Omega}^{-1})^{-1}\}\|\boldsymbol{W}_{\boldsymbol{r}}^T\boldsymbol{b}_{\boldsymbol{J},\boldsymbol{q}-\boldsymbol{r}}(\boldsymbol{x})\|^2
\gtrsim n^{-1}\prod_{k=1}^dJ_k\|\boldsymbol{W}_{\boldsymbol{r}}^T\boldsymbol{b}_{\boldsymbol{J},\boldsymbol{q}-\boldsymbol{r}}(\boldsymbol{x})\|^2.$$
For any $\boldsymbol{x}=(x_1,\dotsc,x_d)^T\in[0,1]^d$, let $i_{x_k}$ be a positive integer such that $x_k\in[t_{k,i_{x_k}-1},t_{k,i_{x_k}}]$. Then only $B_{i_{x_k},q_k-r_k}(x_k),\dotsc,B_{i_{x_k}+q_k-r_k-1,q_k-r_k}(x_k)$ are nonzero at each $k=1,\dotsc,d$. In view of \eqref{eq:firstdiff}, $(\partial^{r_k}/\partial x_k^{r_k})B_{j_k,q_k}(x_k)$ is a linear combination of $B_{j_k,q_k-r_k}(x_k),\dotsc,B_{j_k+r_k,q_k-r_k}(x_k)$ for any $1\leq j_k\leq J_k$ with $k=1,\dotsc,d$. Choose $j_k=i_{x_k}+q_k-r_k-1$ for $k=1,\dotsc,d$, then by \eqref{eq:tendpoints}, we have
\begin{align}\label{eq:Bprimelower}
\|\boldsymbol{W}_{\boldsymbol{r}}^T\boldsymbol{b}_{\boldsymbol{J},\boldsymbol{q}-\boldsymbol{r}}(\boldsymbol{x})\|^2
&=\sum_{j_1=1}^{J_1}\dotsi\sum_{j_d=1}^{J_d}\prod_{k=1}^d\left(\frac{\partial^{r_k}}{\partial x_k^{r_k}}B_{j_k,q_k}(x_k)\right)^2\nonumber\\
&\geq\prod_{k=1}^d\prod_{u=1}^{r_k}\frac{(q_k-u)^2B_{i_{x_k}+q_k-r_k-1,q_k-r_k}(x_k)^2}{(t_{k,i_{x_k}+q_k-r_k-1}-t_{k,i_{x_k}-r_k-1+u})^2}\nonumber\\
&\gtrsim\prod_{k=1}^d\frac{1}{\Delta_k^{2r_k}}\left(\frac{q_k-r_k}{q_k}\right)^2\gtrsim\prod_{k=1}^dJ_k^{2r_k},
\end{align}
since $t_{k,i_{x_k}+q_k-r_k-1}-t_{k,i_{x_k}-r_k-1+u}\leq(q_k-u+1)\Delta_k\asymp J_k^{-1}$ for $k=1,\dotsc,d$, by Lemma \ref{lem:tasymp}. Consequently, $t_{n, \boldsymbol{r}, \gamma_n}(\boldsymbol{x})^2\gg n^{-1}\prod_{k=1}^dJ_k^{2r_k+1}$.

In view of \eqref{eq:tbiasbound},
uniformly on $\|f_0\|_{\boldsymbol{\alpha},\infty}\leq R$,
\begin{equation}
\mathrm{E}_0[D^{\boldsymbol{r}}f_0(\boldsymbol{x})-D^{\boldsymbol{r}}\widetilde{f}(\boldsymbol{x})]^2
\lesssim n^{-1}\prod_{k=1}^dJ_k^{2r_k+1}+\sum_{k=1}^dJ_k^{-2(\alpha_k-r_k)}.
\label{posterior mean pointwise variation}
\end{equation}
Hence uniformly on $\|f_0\|_{\boldsymbol{\alpha},\infty}\leq R$, the
lack of coverage of $\widehat{\mathcal{C}}_{n, \boldsymbol{r}, \gamma_n}(\boldsymbol{x})$
\begin{align*}
P_0(|D^{\boldsymbol{r}}f_0(\boldsymbol{x})-D^{\boldsymbol{r}}\widetilde{f}(\boldsymbol{x})|
>t_{n, \boldsymbol{r}, \gamma_n}(\boldsymbol{x}))\lesssim\frac{\displaystyle n^{-1}\prod_{k=1}^dJ_k^{2r_k+1}+\sum_{k=1}^dJ_k^{-2(\alpha_k-r_k)}}{t_{n, \boldsymbol{r}, \gamma_n}(\boldsymbol{x})^2}.
\end{align*}
For the choice $J_k\asymp n^{\alpha^{*}/\{\alpha_k(2\alpha^{*}+d)\}}$, $k=1,\dotsc,d$, the bound tends to zero
uniformly on $\|f_0\|_{\boldsymbol{\alpha},\infty}\leq R$ and the diameter $\widehat{\sigma}_n z_{\gamma_n/2}\sqrt{\Sigma_{\boldsymbol{r}}(\boldsymbol{x},\boldsymbol{x})}$ of  $\widehat{\mathcal{C}}_{n, \boldsymbol{r}, \gamma_n}(\boldsymbol{x})$ is  $O_{P_0}(\epsilon_{n,\boldsymbol{r}}\sqrt{\log{(1/\gamma_n)}})$ because  $\Sigma_{\boldsymbol{r}}(\boldsymbol{x},\boldsymbol{x})\lesssim n^{-1}\prod_{k=1}^dJ_k^{2r_k+1}$ by \eqref{eq:tvarbound}, $\widehat{\sigma}_n$ converges to $\sigma_0$ and $z_{\gamma_n/2}=O(\sqrt{\log{(1/\gamma_n)}})$ by the estimate
$\mathrm{P}(Z>z)\leq z^{-1}\exp(-z^2/2)$ for $Z\sim\mathrm{N}(0,1)$.

To prove the corresponding assertion for the hierarchical Bayes credible interval, it suffices to show that
\begin{equation}
\label{eq:hirerachical credible 1}
n^{-1}\prod_{k=1}^d J_k^{2r_k+1} \ll R_{n, \boldsymbol{r}, \gamma_n}(\boldsymbol{x})^2 \ll n^{-1}\prod_{k=1}^d J_k^{2r_k+1} \log(1/\gamma_n)
\end{equation}
uniformly on $\|f_0\|_{\boldsymbol{\alpha},\infty}\leq R$. If $\mathcal{U}_n$ shrinks sufficiently slowly to $\sigma_0$, we can ensure that with probability tending to one, $\Pi(\sigma\in\mathcal{U}_n|\boldsymbol{Y})\ge 1-\gamma_n$. By the definition of $R_{n, \boldsymbol{r}, \gamma_n}(\boldsymbol{x})$, we have that
\begin{eqnarray*}
1-\gamma_n &=& \Pi(|D^{\boldsymbol{r}} f(\boldsymbol{x})- A_{\boldsymbol{r}} (\boldsymbol{x})\boldsymbol{Y}-c_{\boldsymbol{r}}(\boldsymbol{x})\boldsymbol{\eta}|\le R_{n, \boldsymbol{r}, \gamma_n}(\boldsymbol{x})|\boldsymbol{Y})\\
&\le & \sup_{\sigma\in\mathcal{U}_n} \Pi(|D^{\boldsymbol{r}} f(\boldsymbol{x})- A_{\boldsymbol{r}} (\boldsymbol{x})\boldsymbol{Y}-c_{\boldsymbol{r}}(\boldsymbol{x})\boldsymbol{\eta}|
\le  R_{n, \boldsymbol{r}, \gamma_n}(\boldsymbol{x})|\boldsymbol{Y},\sigma)+\gamma_n.
\end{eqnarray*}
Since given $\sigma$, $z_{\gamma_n}\sigma \sqrt{\boldsymbol{\Sigma}_{\boldsymbol{r}}(\boldsymbol{x},\boldsymbol{x})}$ is the $(1-2\gamma_n)$-posterior quantile of $|D^{\boldsymbol{r}} f(\boldsymbol{x})- A_{\boldsymbol{r}} (\boldsymbol{x})\boldsymbol{Y}-c_{\boldsymbol{r}}(\boldsymbol{x})\boldsymbol{\eta}|$, it follows that
on a set of probability tending to one, $R_{n, \boldsymbol{r}, \gamma_n}(\boldsymbol{x})\ge z_{\gamma_n}\inf \{\sigma: \sigma\in \mathcal{U}_n\}\sqrt{\boldsymbol{\Sigma}_{\boldsymbol{r}}(\boldsymbol{x},\boldsymbol{x})}$.
On the other hand from $$1-\gamma_n\ge \inf_{\sigma\in\mathcal{U}_n} \Pi(|D^{\boldsymbol{r}} f(\boldsymbol{x})- A_{\boldsymbol{r}} (\boldsymbol{x})\boldsymbol{Y}-c_{\boldsymbol{r}}(\boldsymbol{x})\boldsymbol{\eta}|\le R_{n, \boldsymbol{r}, \gamma_n}(\boldsymbol{x})|\boldsymbol{Y},\sigma)\Pi(\sigma\in \mathcal{U}_n|\boldsymbol{Y}),$$ we get $R_{n, \boldsymbol{r}, \gamma_n}(\boldsymbol{x})\le z_{\gamma_n/2(1-\gamma_n)}\sup \{\sigma: \sigma\in \mathcal{U}_n\}\sqrt{\boldsymbol{\Sigma}_{\boldsymbol{r}}(\boldsymbol{x},\boldsymbol{x})}$. This establishes \eqref{eq:hirerachical credible 1}.
\end{proof}

\begin{proof}[Proof of Theorem \ref{th:tcrinftyrho}]
For notational simplicity, we write $h_{n, \boldsymbol{r}, \infty, \gamma}$ as $h_{\infty,\gamma}$ and define $t_{\infty, \gamma}=\inf_{\sigma\in\mathcal{U}_n}\sigma h_{n, \boldsymbol{r}, \infty, \gamma}$. First we consider the empirical Bayes credible region. To show $\widehat{\mathcal{C}}_{n, \boldsymbol{r}, \infty, \gamma}^{\rho_n}$ has asymptotic coverage of $1$, it suffices to show that
\begin{align}
\sup_{\|f_0\|_{\boldsymbol{\alpha},\infty}\leq R}P_0(\|D^{\boldsymbol{r}}f_0-D^{\boldsymbol{r}}\widetilde{f}\|_\infty> \rho_nt_{\infty, \gamma})\rightarrow0.
\end{align}
Let $Z_{n,\boldsymbol{r}}\sim\mathrm{GP}(0,\Sigma_{\boldsymbol{r}})$. Let $M_Z$ be the median of $\|Z_{n,\boldsymbol{r}}\|_\infty$, i.e., $M_Z$ satisfying  $\mathrm{P}(\|Z_{n,\boldsymbol{r}}\|_\infty\leq M_Z)\geq1/2$ and $\mathrm{P}(\|Z_{n,\boldsymbol{r}}\|_\infty\geq M_Z)\geq1/2$. Let  $\sigma_{Z}^2=\sup_{\boldsymbol{x}\in[0,1]^d}\mathrm{Var}(Z_{n,\boldsymbol{r}}(\boldsymbol{x}))$ and note that by \eqref{eq:tvarbound}, $\sigma^2_{Z}\lesssim n^{-1}\prod_{k=1}^dJ_k^{2r_k+1}\to 0$ for $J_k\asymp(n/\log{n})^{\alpha^{*}/\{\alpha_k(2\alpha+d)\}}, k=1,\dotsc,d$. Using the facts that $\sigma_{Z}\leq 2M_Z$ and $|\mathrm{E}\|Z_{n,\boldsymbol{r}}\|_\infty-M_Z|\leq\sigma_{Z}(\pi/2)^{1/2}$ (see Pages 52 and 54 of \citep{talagrand}), we have $\mathrm{E}\|Z_{n,\boldsymbol{r}}\|_\infty\asymp M_Z$.

Because $\mathrm{P}(\|Z_{n,\boldsymbol{r}}\|_\infty>h_{\infty, \gamma})=\gamma$, and $\gamma<1/2$, we have $h_{\infty, \gamma}\geq M_Z\asymp\mathrm{E}\|Z_{n,\boldsymbol{r}}\|_\infty$. To lower bound  $\mathrm{E}\|Z_{n,\boldsymbol{r}}\|_\infty$, we introduce the notations $\mathcal{T}_k=\{t_{k,1},\dotsc,t_{k,N_k}\}$, $k=1,\dotsc,d$ and  $\mathcal{T}=\prod_{k=1}^d\mathcal{T}_k$. Define $\mathcal{I}=\{(i_1,\dotsc,i_d):1\leq i_k\leq N_k,k=1,\dotsc,d\}$ and arrange the elements of $\mathcal{I}$ lexicographically. Then, we can enumerate the $N=\prod_{k=1}^dN_k$ elements of $\mathcal{T}$ as $\{\boldsymbol{\tau}_{\boldsymbol{i}}:\boldsymbol{i}\in\mathcal{I}\}$, where $\boldsymbol{\tau}_{\boldsymbol{i}}=(t_{1,i_1},\dotsc,t_{d,i_d})$ with $(i_1,\dotsc,i_d)\in\mathcal{I}$. Define $u(x_1,\dotsc,x_d)=\prod_{k=1}^d(\partial^{r_k}/\partial x_k^{r_k})B_{j_k,q_k}(x_k)$. Applying the multivariate mean value theorem to $u(x_1,\dotsc,x_d)$ at  $\boldsymbol{\tau}_{\boldsymbol{i}}$ and $\boldsymbol{\tau}_{\boldsymbol{m}}$, we have for some point $\boldsymbol{\tau}^{*}=(t_1^{*},\dotsc,t_d^{*})=\lambda\boldsymbol{\tau}_{\boldsymbol{i}}+(1-\lambda)\boldsymbol{\tau}_{\boldsymbol{m}}$ with $\lambda\in[0,1]$,
\begin{align}
\sum_{j_1=1}^{J_1}\dotsi\sum_{j_d=1}^{J_d}|u(\boldsymbol{\tau}_{\boldsymbol{i}})-u(\boldsymbol{\tau}_{\boldsymbol{m}})|^2
&=\sum_{j_1=1}^{J_1}\dotsi\sum_{j_d=1}^{J_d}|\nabla u(\boldsymbol{\tau}^{*})^T(\boldsymbol{\tau}_{\boldsymbol{i}}-\boldsymbol{\tau}_{\boldsymbol{m}})|^2\nonumber\\
&=\sum_{j_1=1}^{J_1}\dotsi\sum_{j_d=1}^{J_d}\left|\sum_{\beta=1}^d\left(\frac{\partial u}{\partial x_\beta}\right)(t_{\beta,i_\beta}-t_{\beta,m_\beta})\right|^2.\label{eq:tmean}
\end{align}
Choosing $j_1=i_{x_1}+q_1-r_1-2$ and $j_k=i_{x_k}+q_k-r_k-1$ for $k=2,\dotsc,d$, it then follows that $(\partial^{r_1+1}/\partial x_1^{r_1+1})B_{j_1,q_1}(x_1)>0$, while $(\partial^{r_k}/\partial x_k^{r_k})B_{j_k,q_k}(x_k)>0$ and $(\partial^{r_k+1}/\partial x_k^{r_k+1})B_{j_k,q_k}(x_k)=0$ for $k=2,\dotsc,d$. We show only the first implication;  the other two can be argued similarly. For  $\boldsymbol{x}=(x_1,\dotsc,x_d)^T\in[0,1]^d$, let $i_{x_k}$ be a positive integer such that $x_k\in[t_{k,i_{x_k-1}},t_{k,i_{x_k}}]$ for $k=1,\dotsc,d$. Now by \eqref{eq:firstdiff},  $(\partial^{r_1+1}/\partial x_1^{r_1+1})B_{j_1,q_1}(x_1)$ is a linear combination of the set of functions $\{B_{j_1,q_1-r_1-1}(x_1),\dotsc,B_{j_1+r_1+1,q_1-r_1-1}(x_1)\}$  while only $\{B_{i_{x_1},q_1-r_1-1}(x_1),\dotsc,B_{i_{x_1}+q_1-r_1-2,q_1-r_1-1}(x_1)\}$ are nonzero by the support property of B-splines. For $j_1=i_{x_1}+q_1-r_1-2$, only the positive term corresponding to $B_{i_{x_1}+q_1-r_1-2,q_1-r_1-1}(x_1)$, with coefficients given by the second equation of \eqref{eq:tendpoints} below survives. Thus, only $\partial u/\partial x_1$ will be positive while $\partial u/\partial x_k=0$ for $k=2,\dotsc,d$. By repeated applications of \eqref{eq:Bprimelower}, the right hand side of \eqref{eq:tmean} is bounded below by
\begin{align*}
&\left(\frac{\partial^{r_1+1}}{\partial x_1^{r_1+1}}B_{j_1,q_1}(x_1)\right)^2\prod_{k=2}^d\left(\frac{\partial^{r_k}}{\partial x_k^{r_k}}B_{j_k,q_k}(x_k)\right)^2(t_{1,i_1}-t_{1,m_1})^2\nonumber\\
&\qquad\gtrsim J_1^{2r_1+2}\prod_{k=2}^dJ_k^{2r_k}\left(\min_{1\leq l\leq N_1}\delta_{1,l}^2\right)\gtrsim\prod_{k=1}^dJ_k^{2r_k},
\end{align*}
where $\delta_{1,l}=t_{1,l}-t_{1,l-1}$ for $1\leq l\leq N_1$, and the last inequality follows from the quasi-uniformity of knots and Lemma \ref{lem:tasymp}. Define $V_{\boldsymbol{i}}=Z_{n,\boldsymbol{r}}(\boldsymbol{\tau}_{\boldsymbol{i}})$ for $\boldsymbol{1}_d\leq \boldsymbol{i}\leq\boldsymbol{N}$ where $\boldsymbol{N}=(N_1,\dotsc,N_d)^T$. Note that $\|Z_{n,\boldsymbol{r}}\|_\infty\geq\max_{\boldsymbol{1}_d\leq\boldsymbol{i}\leq\boldsymbol{N}}V_{\boldsymbol{i}}$. Then by \eqref{eq:teigen}, for any $\boldsymbol{1}_d\leq\boldsymbol{i},\boldsymbol{m}\leq\boldsymbol{N}$,
\begin{align*}
\mathrm{E}(V_{\boldsymbol{i}}-V_{\boldsymbol{m}})^2
&\geq\lambda_{\mathrm{min}}\{(\boldsymbol{B}^T\boldsymbol{B}+\boldsymbol{\Omega}^{-1})^{-1}\}\|\boldsymbol{W}_{\boldsymbol{r}}^T
(\boldsymbol{b}_{\boldsymbol{J},\boldsymbol{q}-\boldsymbol{r}}(\boldsymbol{\tau}_{\boldsymbol{i}})-
\boldsymbol{b}_{\boldsymbol{J},\boldsymbol{q}-\boldsymbol{r}}(\boldsymbol{\tau}_{\boldsymbol{m}}))\|^2\nonumber\\
&\gtrsim\left(\frac{1}{n}\prod_{k=1}^dJ_k\right)\sum_{j_1=1}^{J_1}\dotsi\sum_{j_d=1}^{J_d}|u(\boldsymbol{\tau}_{\boldsymbol{i}})-u(\boldsymbol{\tau}_{\boldsymbol{m}})|^2
\geq \frac{c}{n}\prod_{k=1}^dJ_k^{2r_k+1},
\end{align*}
for a universal constant $c>0$. Define $U_{\boldsymbol{i}}=\sqrt{(2n/c)}\prod_{k=1}^dJ_k^{-(r_k+1/2)}V_{\boldsymbol{i}}$ and let $H_{\boldsymbol{i}}$ be i.i.d. $\mathrm{N}(0,1)$ with $\boldsymbol{1}_d\leq\boldsymbol{i}\leq\boldsymbol{N}$. By (3.14) of \citep{talagrand}, we have $\mathrm{E}(\max_{\boldsymbol{1}_d\leq\boldsymbol{i}\leq\boldsymbol{N}}H_{\boldsymbol{i}})\gtrsim\sqrt{\log{N}}$. Now, $\mathrm{E}(U_{\boldsymbol{i}}-U_{\boldsymbol{m}})^2\geq2=\mathrm{E}(H_{\boldsymbol{i}}-H_{\boldsymbol{m}})^2$ and hence by Slepian's Lemma (Corollary 3.14 of \citep{talagrand}),
\begin{align*}
\mathrm{E}\left(\max_{\boldsymbol{1}_d\leq\boldsymbol{i}\leq\boldsymbol{N}}U_{\boldsymbol{i}}\right)\geq\mathrm{E}\left(\max_{\boldsymbol{1}_d\leq \boldsymbol{i}\leq\boldsymbol{N}}H_{\boldsymbol{i}}\right)\gtrsim\sqrt{\log{N}},
\end{align*}
where $N=\prod_{k=1}^dN_k\sim\prod_{k=1}^dJ_k$. It then follows that $t_{\infty, \gamma}^2\gtrsim\sigma_0^2h_{\infty, \gamma}^2\gtrsim(\log{n}/n)\prod_{k=1}^dJ_k^{2r_k+1}$. Therefore using \eqref{eq:tvarsupbound} and \eqref{eq:tbiassupbound}, we have uniformly on $\|f_0\|_{\boldsymbol{\alpha},\infty}\leq R$,
\begin{align}
\mathrm{E}_0(\|D^{\boldsymbol{r}}f_0-D^{\boldsymbol{r}}\widetilde{f}\|_\infty^2)&\leq2\mathrm{E}(\|Q_{n,\boldsymbol{r}}\|_\infty^2)+
2\|\boldsymbol{A}_{\boldsymbol{r}}\boldsymbol{F}_0+\boldsymbol{c}_{\boldsymbol{r}}\boldsymbol{\eta}-D^{\boldsymbol{r}}f_0\|_\infty^2\nonumber\\
&\lesssim\frac{\log{n}}{n}\prod_{k=1}^dJ_k^{2r_k+1}+\sum_{k=1}^dJ_k^{-2(\alpha_k-r_k)}.\label{eq:tbiasderivsup}
\end{align}
Hence for the choice $J_k\asymp(n/\log{n})^{\alpha^{*}/\{\alpha_k(2\alpha^{*}+d)\}}$,  $k=1,\dotsc,d$,
$P_0(\|D^{\boldsymbol{r}}f_0-D^{\boldsymbol{r}}\widetilde{f}\|_\infty>\rho_nt_{\infty, \gamma})\to 0$
since $t_{\infty, \gamma}^2\gtrsim(\log{n}/n)\prod_{k=1}^dJ_k^{2r_k+1}$ and  $\rho_n\to\infty$.

If the true errors are i.i.d. $\mathrm{N}(0,\sigma_0^2)$, then $Q_{n,\boldsymbol{r}}\sim\mathrm{GP}(0,\sigma_0^2\Psi_{\boldsymbol{r}})$ under $P_0$. Define $\sigma_Q^2=\sup_{\boldsymbol{x}\in[0,1]^d}\mathrm{Var}(Q_{n,\boldsymbol{r}}(\boldsymbol{x}))$. We have for constants $C_1,C_2,C_3>0$,
\begin{align*}
t_{\infty, \gamma}\geq C_1\epsilon_{n,\boldsymbol{r},\infty},\;\|D^{\boldsymbol{r}}f_0-\mathrm{E}_0D^{\boldsymbol{r}}\widetilde{f}\|_\infty\leq C_2\epsilon_{n,\boldsymbol{r},\infty},\;\mathrm{E}\|Q_{n,\boldsymbol{r}}\|_\infty\leq C_3\epsilon_{n,\boldsymbol{r},\infty}.
\end{align*}
The first inequality was established above, while the second and third inequalities follow from \eqref{eq:tbiassupbound} and \eqref{eq:tvarsupbound}. Then by Proposition A.2.1 of \citep{empirical}, $P_0(\|Q_{n,\boldsymbol{r}}\|_\infty>2C_3\epsilon_{n,\boldsymbol{r},\infty})$ is bounded by
\begin{align*} P_0(\|Q_{n,\boldsymbol{r}}\|_\infty>\mathrm{E}\|Q_{n,\boldsymbol{r}}\|_\infty+C_3\epsilon_{n,\boldsymbol{r},\infty})\leq2\exp\{-C_3^2\epsilon_{n,\boldsymbol{r},\infty}^2/(2\sigma_Q^2)\}.
\end{align*}
In view of \eqref{eq:tvarbiasbound}, we have $\sigma_Q^2=O(\epsilon_{n,\boldsymbol{r}}^2)$. Since $\epsilon_{n,\boldsymbol{r}}\ll\epsilon_{n,\boldsymbol{r},\infty}$, this implies that the right hand side above tends to zero as $n\rightarrow\infty$. By the triangle inequality, we have $$P_0(\|D^{\boldsymbol{r}}f_0-D^{\boldsymbol{r}}\widetilde{f}\|_\infty>\rho t_{\infty, \gamma})\leq P_0(\|Q_{n,\boldsymbol{r}}\|_\infty>\rho t_{\infty, \gamma}-\|D^{\boldsymbol{r}}f_0-\mathrm{E}_0D^{\boldsymbol{r}}\widetilde{f}\|_\infty)$$
which tends to 0 if $\rho\geq(2C_3+C_2)/C_1$.

To estimate the diameter $\widehat{\sigma}_n \rho_{n}h_{\infty, \gamma}$,
the last inequality in Proposition A.2.1 of \citep{empirical} gives
$\gamma=\mathrm{P}(\|Z_{n,\boldsymbol{r}}\|_\infty>h_{\infty, \gamma})\leq2\exp\{-h_{\infty, \gamma}^2/(8\mathrm{E}\|Z_{n,\boldsymbol{r}}\|_\infty^2)\}.$
Therefore,  $h_{\infty, \gamma}\lesssim(\mathrm{E}\|Z_{n,\boldsymbol{r}}\|_\infty^2)^{1/2}\sqrt{-\log{\gamma}}$ and hence the assertion follows from \eqref{eq:tvarsupbound}.

To prove the assertions about hierarchical Bayes credible regions, we proceed as in the proof of Theorem~\ref{th:tcrpoint}. By definition
\begin{align*}
1-\gamma &= \Pi ( \|D^{\boldsymbol{r}} f-A_{\boldsymbol{r}}\boldsymbol{Y}-c_{\boldsymbol{r}} \boldsymbol{\eta}\|_\infty \le R_{n, \boldsymbol{r}, \infty, \gamma}|\boldsymbol{Y})\\
&\le \sup_{\sigma\in \mathcal{U}_n} \Pi ( \|D^{\boldsymbol{r}} f-A_{\boldsymbol{r}}\boldsymbol{Y}-c_{\boldsymbol{r}} \boldsymbol{\eta}\|_\infty \le R_{n, \boldsymbol{r}, \infty, \gamma}|\boldsymbol{Y},\sigma)+\Pi(\sigma \not\in \mathcal{U}_n|\boldsymbol{Y}).
\end{align*}
Choose $\gamma'$ strictly between $\gamma$ and $1/2$.
Making $\mathcal{U}_n$ to shrink sufficiently slowly to $\sigma_0$ so that $\Pi(\sigma\in\mathcal{U}_n|\boldsymbol{Y})\ge 1-\gamma'+\gamma$ with probability tending to one and using the facts that the conditional posterior distribution of $(D^{\boldsymbol{r}} f-A_{\boldsymbol{r}}\boldsymbol{Y}-c_{\boldsymbol{r}} \boldsymbol{\eta})/\sigma$ given $\sigma$ is equal to the distribution of the Gaussian process $Z_{n,\boldsymbol{r}}$, which is free of $\sigma$, and $\|Z_{n,\boldsymbol{r}}\|_\infty$ has $(1-\gamma')$ quantile $t_{\infty, \gamma'}$, we obtain
$$R_{n, \boldsymbol{r}, \infty, \gamma}\ge \inf\{\sigma:\, \sigma\in  \mathcal{U}_n\}t_{\infty, \gamma'}\asymp t_{\infty, \gamma'}\gtrsim \epsilon_{n, \boldsymbol{r}, \infty}.$$
Hence the modified hierarchical Bayes credible region $\mathcal{C}_{n, \boldsymbol{r}, \infty, \gamma} ^{\rho_n}$ has asymptotic coverage 1 for any $\rho_n\to \infty$, and for the Gaussian true error we can choose $\rho_n=\rho$ for a sufficiently large constant. To bound the diameter of  $\mathcal{C}_{n, \boldsymbol{r}, \infty, \gamma } ^{\rho_n}$
we use the relation
$$
1-\gamma \ge \inf_{\sigma\in \mathcal{U}_n} \Pi ( \|D^{\boldsymbol{r}} f-A_{\boldsymbol{r}}\boldsymbol{Y}-c_{\boldsymbol{r}} \boldsymbol{\eta}\|_\infty \le R_{n, \boldsymbol{r}, \infty, \gamma}|\boldsymbol{Y},\sigma)\Pi(\sigma\in \mathcal{U}_n|\boldsymbol{Y})
$$
to conclude that $R_{n, \boldsymbol{r}, \infty, \gamma}\le t_{\infty, \gamma/(1-\gamma'+\gamma)}$, which is of the order $\epsilon_{n,\boldsymbol{r},\infty}$ since $\gamma/(1-\gamma'+\gamma)<1/2$ by the choice of $\gamma'$.
\end{proof}

\begin{proof}[Proof of Remark~\ref{L2credible}]
We indicate how to show coverage of the empirical Bayes credible region; the necessary changes for the hierarchical version can be made as in the proofs of Theorems~\ref{th:tcrpoint} and \ref{th:tcrinftyrho}. The adequacy of the coverage will be shown if
$P_0\left(\|D^{\boldsymbol{r}}f_0-D^{\boldsymbol{r}}\widetilde{f}\|_2>t_{n, \boldsymbol{r}, 2,\gamma_n}\right)\rightarrow0$ uniformly on
$\|f_0\|_{\boldsymbol{\alpha},\infty}\leq R$, where
$t_{n, \boldsymbol{r}, 2,\gamma_n}=\inf_{\sigma^2\in\mathcal{U}_n}\sigma h_{n, \boldsymbol{r}, 2,\gamma_n}$. Let $Z_{n,\boldsymbol{r}}\sim\mathrm{GP}(0,\Sigma_{\boldsymbol{r}})$. Since $\|Z_{n,\boldsymbol{r}}\|_2\ge \int Z_{n,\boldsymbol{r}}$ which is normally distributed  with mean $0$ and variance $\int\int\Sigma_{\boldsymbol{r}}(\boldsymbol{x},\boldsymbol{y})d\boldsymbol{x}d\boldsymbol{y}$, it follows from \eqref{eq:tsigmar} that
$$h_{n, \boldsymbol{r},2, \gamma_n}^2\gg\int\int\Sigma_{\boldsymbol{r}}(\boldsymbol{x},\boldsymbol{y})
d\boldsymbol{x}d\boldsymbol{y}
\gtrsim\frac{1}{n}\prod_{k=1}^dJ_k\left\|\int\boldsymbol{W}_{\boldsymbol{r}}^T\boldsymbol{b}_{\boldsymbol{J},\boldsymbol{q}-\boldsymbol{r}}(\boldsymbol{x})d\boldsymbol{x}\right\|^2.$$
Extending the last two equations in the proof of Lemma 6.7 in \citep{localspline} to multivariate splines by arguments used in the proof of the last two theorems, it follows from the last display that  $t_{n, \boldsymbol{r},2, \gamma_n}^2\gtrsim\sigma_0^2h_{n,\boldsymbol{r}2,\gamma_n}^2\gg n^{-1}\prod_{k=1}^dJ_k^{2r_k+1}$. On the other hand,
$$\mathrm{E}_0\|f_0-\widetilde f\|_2^2=\int \mathrm{E}_0 |f_0(\boldsymbol{x})-\widetilde f(\boldsymbol{x})|^2 d\boldsymbol{x}\lesssim
\frac{1}{n}\prod_{k=1}^dJ_k^{2r_k+1}+\sum_{k=1}^dJ_k^{-2(\alpha_k-r_k)},
$$
by \eqref{posterior mean pointwise variation}.
Then uniformly on  $\|f_0\|_{\boldsymbol{\alpha},\infty}\leq R$, the coverage of $\widehat{C}_{n, \boldsymbol{r}, 2,\gamma_n}$ goes to one in probability by Markov's inequality. Since in view of (3.5) from \citep{talagrand},
$\gamma_n=P(\|Z_{n,\boldsymbol{r}}\|_2>h_{n,\boldsymbol{r},2,\gamma_n})\leq4\exp\{-h_{n,\boldsymbol{r},2,\gamma_n}^2/(8\mathrm{E}\|Z_{n,\boldsymbol{r}}\|_2^2)\},$
the size of the radius of the $L_2$-confidence region is estimated as $O_{P_0}(\epsilon_{n, \boldsymbol{r}}\sqrt{\log{(1/\gamma_n)}})$.
\end{proof}

\section{Appendix}
\begin{lemma}\label{lem:tasymp}
Under quasi-uniform knots, $\Delta_k\asymp N_k^{-1}\asymp J_k^{-1},k=1,\dotsc,d$.
\end{lemma}

\begin{proof}
The proof is straightforward because all $N_k$ spacings are of the same order and they sum to one.
\end{proof}

\begin{lemma}\label{lem:Wr}
Each non-zero entry of $\boldsymbol{W}_{\boldsymbol{r}}$ defined implicitly in \eqref{eq:derivspline} is uniformly $O(\prod_{k=1}^d\Delta_k^{-r_k})$.
\end{lemma}
\begin{proof}
Recall that the dimension of $\boldsymbol{W}_{\boldsymbol{r}}$ is $\prod_{k=1}^d(J_k-r_k)\times\prod_{k=1}^dJ_k$. In view of \eqref{eq:firstdiff}, each row of $\boldsymbol{W}_{\boldsymbol{r}}$ has only $\prod_{k=1}^d(r_k+1)$ nonzero entries and their arrangement is analogues to a banded matrix, namely the position of nonzero entries in the current row is a shift of one entry to the right of the nonzero entries' position in the previous row. Also, each column of $\boldsymbol{W}_{\boldsymbol{r}}$ has at most $\prod_{k=1}^d(r_k+1)$ nonzero entries. We index the rows and columns of $\boldsymbol{W}_{\boldsymbol{r}}$ using $d$-dimensional indices as in Definition \ref{def:band}.

Define $\tilde{r}=\prod_{k=1}^d(r_k+1)-2$, and let $\mathcal{G}=\{\boldsymbol{u}=(u_1,\dotsc,u_d):\boldsymbol{0}\leq\boldsymbol{u}\leq\boldsymbol{r},\boldsymbol{u}\neq\boldsymbol{r},\boldsymbol{u}\neq\boldsymbol{0}\}$.
By ordering the elements in $\mathcal{G}$ lexicographically, we can enumerate its elements by $\mathcal{G}=\{\boldsymbol{g}_1,\dotsc,\boldsymbol{g}_{\tilde{r}}\}$. Furthermore, define sets $\mathcal{I}=\{(i_1,\dotsc,i_d):1\leq i_k\leq J_k-r_k,k=1,\dotsc,d\}$ and $\mathcal{J}=\{(j_1,\dotsc,j_d):1\leq j_k\leq J_k,k=1,\dotsc,d\}$, where we order their elements lexicographically. Let $w^{(\boldsymbol{r})}_{\boldsymbol{i},\boldsymbol{j}}$ denote the $(\boldsymbol{i},\boldsymbol{j})$th element of $\boldsymbol{W}_{\boldsymbol{r}}$ such that $\boldsymbol{i}\in\mathcal{I}$ and $\boldsymbol{j}\in\mathcal{J}$. The expressions for the nonzero entries can be described as follows:
for each row $\boldsymbol{i}\in\mathcal{I}$, the first and last nonzero entries are given by
\begin{align}\label{eq:tendpoints}
w_{\boldsymbol{i},\boldsymbol{i}}^{(\boldsymbol{r})}&=(-1)^{\sum_{k=1}^dr_k}\prod_{k=1}^d\prod_{l=1}^{r_k}\frac{q_k-l}{t_{k,i_k}-t_{k,i_k-q_k+l}},\nonumber\\ w_{\boldsymbol{i},\boldsymbol{i}+\boldsymbol{r}}^{(\boldsymbol{r})}&=\prod_{k=1}^d\prod_{l=1}^{r_k}\frac{q_k-l}{t_{k,i_k+l-1}-t_{k,i_k-q_k+r_k}}.
\end{align}
If $\tilde{r}$ is odd, we partition $\mathcal{G}=\mathcal{G}_1\cup\{\boldsymbol{g}_{(\tilde{r}+1)/2}\}\cup\mathcal{G}_2$ where $\mathcal{G}_1=\{\boldsymbol{g}_1,\dotsc,\boldsymbol{g}_{(\tilde{r}-1)/2}\}$ and $\mathcal{G}_2=\{\boldsymbol{g}_{(\tilde{r}+3)/2},\dotsc,\boldsymbol{g}_{\tilde{r}}\}$. The intermediate nonzero entries $w_{\boldsymbol{i},\boldsymbol{i}+\boldsymbol{h}}^{(\boldsymbol{r})}$ for $\boldsymbol{h}=(h_1,\dotsc,h_d)^T\in\mathcal{G}_1$ are
\begin{align}\label{eq:twijodd1half}
(-1)^{\sum_{k=1}^d(r_k-h_k)}w_{\boldsymbol{i},\boldsymbol{i}}^{(\boldsymbol{r})}\left[1+\prod_{k=1}^d\sum^{\binom{r_k}{h_k}-1}_{t=1}\prod_{s=1}^t\frac{t_{k,i_k}-t_{k,i_k-q_k+s}}{t_{k,i_k+1}-t_{k,i_k+1-q_k+s}}\right],
\end{align}
while for $\boldsymbol{h}\in\mathcal{G}_2$, it is
\begin{align}\label{eq:twijodd2half}
(-1)^{\sum_{k=1}^d(r_k-h_k)}w_{\boldsymbol{i},\boldsymbol{i}+\boldsymbol{r}}^{(\boldsymbol{r})}\left[1+\prod_{k=1}^d\sum_{t=1}^{\binom{r_k}{h_k}-1}\prod_{s=1}^t\frac{t_{k,i_k+r_k-s}-t_{k,i_k+r_k-q_k}}{t_{k,i_k+r_k-1-s}-t_{k,i+r_k-1-q_k}}\right],
\end{align}
When $\boldsymbol{h}=\boldsymbol{g}_{(\tilde{r}+1)/2}$, we have
\begin{align}\label{eq:twijmiddle}
w_{\boldsymbol{i},\boldsymbol{i}+\boldsymbol{h}}^{(\boldsymbol{r})}&=w_{\boldsymbol{i},\boldsymbol{i}}^{(\boldsymbol{r})}
\prod_{k=1}^d\left(\frac{t_{k,i_k}-t_{k,i_k-q_k+1}}{t_{k,i_k+1}-t_{k,i_k+2-q_k}}\right)\nonumber\\
&\qquad\times\left[1+\prod_{k=1}^d
\sum_{t=1}^{\binom{r_k-1}{r_k/2-1}-1}\prod_{s=0}^{t-1}\frac{t_{k,i_k+s}-t_{k,i_k+2-q_k}}{t_{k,i_k+1+s}-t_{k,i_k+3-q_k}}\right]\nonumber\\
&\qquad+w_{\boldsymbol{i},\boldsymbol{i}+\boldsymbol{r}}^{(\boldsymbol{r})}\prod_{k=1}^d\left(\frac{t_{k,i_k+r_k-1}
-t_{k,i_k+r_k-q_k}}{t_{k,i_k+r_k-2}-t_{k,i_k+r_k-1-q_k}}\right)\nonumber\\
&\qquad\times\left[1+\prod_{k=1}^d\sum_{t=1}^{\binom{r_k-1}{r_k/2}-1}\prod_{s=0}^{t-1}\frac{t_{k,i_k+r_k-2}-t_{k,i_k+r_k-q_k-s}}{t_{k,i_k+r_k-3}-t_{k,i_k+r_k-1-q_k-s}}\right].
\end{align}
If $\tilde{r}$ is even, we partition $\mathcal{G}=\mathcal{G}_1\cup\mathcal{G}_2$ where $\mathcal{G}_1=\{\boldsymbol{g}_1,\dotsc,\boldsymbol{g}_{\tilde{r}/2}\}$ and $\mathcal{G}_2=\{\boldsymbol{g}_{\tilde{r}/2+1},\dotsc,\boldsymbol{g}_{\tilde{r}}\}$. Then the expression for $w_{\boldsymbol{i},\boldsymbol{i}+\boldsymbol{h}}^{(\boldsymbol{r})}$ is \eqref{eq:twijodd1half} for $\boldsymbol{h}\in\mathcal{G}_1$ and is \eqref{eq:twijodd2half} for $\boldsymbol{h}\in\mathcal{G}_2$. By the quasi-uniformity of the knots, the endpoints in \eqref{eq:tendpoints} are $O(\prod_{k=1}^d\Delta_k^{-r_k})$, while the fractions of knot differences appearing in \eqref{eq:twijodd1half}--\eqref{eq:twijmiddle} are $O(1)$.
\end{proof}

\begin{lemma}\label{lem:tsumB}
$\sum_{i=1}^n\prod_{k=1}^dB_{j_k,q_k}(X_{ik})^{p_k}\lesssim n\prod_{k=1}^dJ_k^{-1}$ for $1\leq j_k\leq J_k$ and $p_k\in\mathbb{N}$, $k=1,\dotsc,d$.
\end{lemma}

\begin{proof}
As $B_{j_k,q_k}(\cdot)\le 1$ and is positive only inside $(t_{k,j_k-q_k},t_{k,j_k})$,
\begin{align*}
\sum_{i=1}^n\prod_{k=1}^dB_{j_k,q_k}(X_{ik})^{p_k}&\leq n\int_{[0,1]^d}\prod_{k=1}^d\Ind_{(t_{k,j_k-q_k},t_{k,j_k}]}(\boldsymbol{x})dG_n(\boldsymbol{x}).
\end{align*}
By the quasi-uniformity of the knots, we have $t_{k,j_k}-t_{k,j_k-q_k}\leq q_k\Delta_k$ and $t_{k,j_k}-t_{k,j_k-q_k}\geq q_k\min_{1\leq l\leq N_k}\delta_{k,l}\geq q_k\Delta_k/C$. This implies that $t_{k,j_k}-t_{k,j_k-q_k}\asymp\Delta_k$ for $k=1,\dotsc,d$. Moreover, assumption \eqref{assump:tbspline2} and Lemma \ref{lem:tasymp} imply that the right hand side above is
\begin{align*}
nG_n\left[\prod_{k=1}^d(t_{k,j_k-q_k},t_{k,j_k})\right]&=nG\left[\prod_{k=1}^d(t_{k,j_k-q_k},t_{k,j_k})\right]+o\left(n\prod_{k=1}^dN_k^{-1}\right)\\
&\lesssim n\prod_{k=1}^d\Delta_k+o\left(n\prod_{k=1}^d\Delta_k\right)\lesssim n\prod_{k=1}^dJ_k^{-1}.\qedhere
\end{align*}
\end{proof}

\begin{lemma}\label{lem:tband}
Let $\boldsymbol{A}$ be a $J\times J$ symmetric and positive definite matrix with its rows and columns indexed by $d$-dimensional multi-indices, i.e., for $\boldsymbol{i}=(i_1,\dotsc,i_d)$ and $\boldsymbol{j}=(j_1,\dotsc,j_d)$, such that $1\leq i_k,j_k\leq J_k, k=1,\dotsc,d, J=\prod_{k=1}^dJ_k$, the $(\boldsymbol{i},\boldsymbol{j})$th element of $\boldsymbol{A}$ is $a_{\boldsymbol{i},\boldsymbol{j}}=\boldsymbol{A}\{(i_1,\dotsc,i_d),(j_1,\dotsc,j_d)\}$. Let $\boldsymbol{A}$ be $\boldsymbol{q}=(q_1,\dotsc,q_d)^T$ banded as in Definition \ref{def:band}. Furthermore, assume that the eigenvalues of $\boldsymbol{A}$ are contained in $[a\tau_m,b\tau_m]$ for fixed $0<a<b<\infty$ and some sequence $\tau_m$. Then $\|\boldsymbol{A}^{-1}\|_{(\infty,\infty)}=O(\tau_m^{-1})$.
\end{lemma}

\begin{proof}
We adapt the proof given in Proposition 2.2 of \citep{bandedmatrix} to the case of multi-dimensional banded matrix. We first note that if $\boldsymbol{A}$ is $\boldsymbol{q}$-banded and $\boldsymbol{B}$ is $\boldsymbol{w}$-banded as in Definition \ref{def:band}, then $\boldsymbol{AB}$ is $\boldsymbol{q}+\boldsymbol{w}$ banded. To see this, observe that $(\boldsymbol{AB})_{\boldsymbol{i},\boldsymbol{j}}=\sum_{l_1=1}^{J_1}\dotsi\sum_{l_d=1}^{J_d}a_{(i_1,\dotsc,i_d),(l_1,\dotsc,l_d)}b_{(j_1,\dotsc,j_d),(l_1,\dotsc,l_d)}\neq0$ only if at least one of the terms in the sum is nonzero. Thus $a_{(i_1,\dotsc,i_d),(l_1,\dotsc,l_d)}\neq0$ and $b_{(j_1,\dotsc,j_d),(l_1,\dotsc,l_d)}\neq0$ for some $(l_1,\dotsc,l_d)$. Hence $|i_k-l_k|\leq q_k$ and $|j_k-l_k|\leq w_k$ for $k=1,\dotsc,d$, and by the triangle inequality, $|i_k-j_k|\leq q_k+w_k$ for $k=1,\dotsc,d$. Therefore, $\boldsymbol{AB}$ is $\boldsymbol{q}+\boldsymbol{w}$ banded. Repeated applications of the same argument show that $\boldsymbol{A}^n$ is $n\boldsymbol{q}$-banded.

Since we can scale $\boldsymbol{A}$ by $\tau_m$ such that its eigenvalues are in $[a, b]$, we set $\tau_m=1$ without loss of generality. Let $p_n(\cdot)$ be a polynomial of degree $n$. Then $p_n(\boldsymbol{A})$ is $n\boldsymbol{q}$-banded. Since the set of eigenvalues for $\boldsymbol{A}$ is $\Lambda(\boldsymbol{A})\subseteq[a,b]$ by assumption, spectral theorem and Proposition 2.1 of \citep{bandedmatrix} imply that
\begin{equation*}
\|\boldsymbol{A}^{-1}-p_n(\boldsymbol{A})\|_{(2,2)}=\max_{x\in\Lambda(\boldsymbol{A})}|1/x-p_n(x)|\leq C_0[(\sqrt{b/a}-1)/(\sqrt{b/a}+1)]^{n+1}
\end{equation*}
for $C_0=(1+\sqrt{b/a})^2/(2b)$. For any $n\in\mathbb{N}$, $p_n(\boldsymbol{A})_{\boldsymbol{i},\boldsymbol{j}}=0$ if $|i_k-j_k|>nq_k$ for some $1\leq k\leq d$. Suppose $\boldsymbol{i}\neq\boldsymbol{j}$, choose $n$ to satisfy $n<\max_{1\leq k\leq d}|i_k-j_k|q_k^{-1}\leq n+1$. Therefore,
\begin{align}
|\boldsymbol{A}^{-1}(\boldsymbol{i},\boldsymbol{j})|&=|\boldsymbol{A}^{-1}(\boldsymbol{i},\boldsymbol{j})-p_n(\boldsymbol{A})_{\boldsymbol{i},\boldsymbol{j}}|\leq\|\boldsymbol{A}^{-1}-p_n(\boldsymbol{A})\|_{(2,2)}\nonumber\\
&\leq C_0[(\sqrt{b/a}-1)/(\sqrt{b/a}+1)]^{\max_{1\leq k\leq d}|i_k-j_k|/q_k}\nonumber\\
&\leq C_0\lambda^{\sum_{k=1}^d|i_k-j_k|}\label{eq:1stcase},
\end{align}
where $\lambda=[(\sqrt{b/a}-1)/(\sqrt{b/a}+1)]^{1/\sum_{k=1}^dq_k}$. When $i_k=j_k$ for all $k=1,\dotsc,d$, we have $\boldsymbol{A}^{-1}(\boldsymbol{i},\boldsymbol{i})\leq\|\boldsymbol{A}^{-1}\|_{(2,2)}=1/\lambda_\text{min}(\boldsymbol{A})\leq1/a$. Combining this case with \eqref{eq:1stcase}, we have $|\boldsymbol{A}^{-1}(\boldsymbol{i},\boldsymbol{j})|\leq C\lambda^{\sum_{k=1}^d|i_k-j_k|}$ for $C=\max\{C_0,1/a\}$. Since $0<\lambda<1$,
\begin{align*}
\|\boldsymbol{A}^{-1}\|_{(\infty,\infty)}&\leq C\max_{1\leq i_k\leq J_k,k=1,\dotsc,d}\sum_{j_1=1}^{J_1}\dotsi\sum_{j_d=1}^{J_d}\prod_{k=1}^d\lambda^{|i_k-j_k|}\\
&\lesssim\prod_{k=1}^d\left(1+2\sum_{j_k=1}^{J_k}\lambda^{j_k}\right)\lesssim
\left(1+2\sum_{j=1}^\infty\lambda^j\right)^d<\infty.\qedhere
\end{align*}
\end{proof}

\begin{lemma}\label{lem:tsplinediff} $\|\boldsymbol{b}_{\boldsymbol{J},\boldsymbol{q}-\boldsymbol{r}}(\boldsymbol{x})
-\boldsymbol{b}_{\boldsymbol{J},\boldsymbol{q}-\boldsymbol{r}}(\boldsymbol{y})\|^2
\lesssim \|\boldsymbol{J}\|^2\|\boldsymbol{x}-\boldsymbol{y}\|^2$ for $\boldsymbol{x},\boldsymbol{y}\in[0,1]^d$.
\end{lemma}
\begin{proof}
By equation (8) of Chapter X in \citep{deBoor} and the triangle inequality,
\begin{align}
|B^{'}_{j_k,q_k-r_k}(x_k)|&\lesssim\frac{|B_{j_k,q_k-r_k-1}(x_k)|}{t_{k,j_k+q_k-r_k-1}-t_{k,j_k}}+\frac{|B_{j_k+1,q_k-r_k-1}(x_k)|}{t_{k,j_k+q_k-r_k}-t_{k,j_k+1}}\nonumber\\
&\lesssim\left(\min_{1\leq l\leq N_k}\delta_{k,l}\right)^{-1}\lesssim\Delta_k^{-1}\lesssim J_k,\label{eq:meanvalue}
\end{align}
where we have used the quasi-uniformity of the knots and Lemma \ref{lem:tasymp}. Using $|\prod_{i=1}^da_i-\prod_{i=1}^db_i|\leq\sum_{i=1}^d|a_i-b_i|$ for $|a_i|\leq1$, $|b_i|\leq1$, $i=1,\dotsc,d$, the mean value theorem, \eqref{eq:meanvalue} and the Cauchy-Schwarz inequality,
\begin{align*}
\left|\prod_{k=1}^dB_{j_k,q_k-r_k}(x_k)-\prod_{k=1}^dB_{j_k,q_k-r_k}(y_k)\right|&\leq\sum_{k=1}^d|B_{j_k,q_k-r_k}(x_k)-B_{j_k,q_k-r_k}(y_k)|\\
&\lesssim\sum_{k=1}^dJ_k|x_k-y_k|\leq\|\boldsymbol{J}\|\|\boldsymbol{x}-\boldsymbol{y}\|.
\end{align*}
Since at most $2\prod_{k=1}^d(q_k-r_k)$ elements in both $\boldsymbol{b}_{\boldsymbol{J},\boldsymbol{q}-\boldsymbol{r}}(\boldsymbol{x})$ and $\boldsymbol{b}_{\boldsymbol{J},\boldsymbol{q}-\boldsymbol{r}}(\boldsymbol{y})$ will be nonzero for any $\boldsymbol{x},\boldsymbol{y}\in[0,1]^d$, $\|\boldsymbol{b}_{\boldsymbol{J},\boldsymbol{q}-\boldsymbol{r}}(\boldsymbol{x})-\boldsymbol{b}_{\boldsymbol{J},\boldsymbol{q}-\boldsymbol{r}}(\boldsymbol{y})\|^2$ is
\begin{align*}
&\sum_{j_1=1}^{J_1-r_1}\dotsi\sum_{j_d=1}^{J_d-r_d}\left|\prod_{k=1}^dB_{j_k,q_k-r_k}(x_k)-\prod_{k=1}^dB_{j_k,q_k-r_k}(y_k)\right|^2\\
&\qquad\lesssim\left[2\prod_{k=1}^d(q_k-r_k)\right]\sum_{k=1}^dJ_k^2\|\boldsymbol{x}-\boldsymbol{y}\|^2\lesssim\sum_{k=1}^dJ_k^2\|\boldsymbol{x}-\boldsymbol{y}\|^2.\qedhere
\end{align*}
\end{proof}

\begin{lemma}\label{lem:tincrement}
Let $\boldsymbol{r}\in\mathbb{N}_0^d$ be such that $\sum_{k=1}^dr_k/\alpha_k<1$. Let $Z_{n,\boldsymbol{r}}\sim\mathrm{GP}(0,\Sigma_{\boldsymbol{r}})$ and $Q_{n,\boldsymbol{r}}$ be a sub-Gaussian process with mean function $0$ and covariance function $\sigma_0^2\Psi_{\boldsymbol{r}}$. Let $J_k\asymp n^{\alpha^{*}/\{\alpha_k(2\alpha^{*}+d)\}}$ for $k=1,\dotsc,d$. Then for any $\boldsymbol{t},\boldsymbol{s}\in[0,1]^d$, we have $\mathrm{Var}[Z_{n,\boldsymbol{r}}(\boldsymbol{t})-Z_{n,\boldsymbol{r}}(\boldsymbol{s})]\leq C \|\boldsymbol{J}\|^2 \|\boldsymbol{t}-\boldsymbol{s}\|^2$ and $\mathrm{Var}[Q_{n,\boldsymbol{r}}(\boldsymbol{t})-Q_{n,\boldsymbol{r}}(\boldsymbol{s})]\leq C\|\boldsymbol{J}\|^2 \|\boldsymbol{t}-\boldsymbol{s}\|^2$ for some constant $C>0$.
\end{lemma}
\begin{proof}
Let $J_k\asymp n^{\alpha^{*}/\{\alpha_k(2\alpha^{*}+d)\}}$ for $k=1,\dotsc,d$, then $\mathrm{Var}[Z_{n,\boldsymbol{r}}(\boldsymbol{t})-Z_{n,\boldsymbol{r}}(\boldsymbol{s})]$ is bounded above by
\begin{align*}
&\|\boldsymbol{b}_{\boldsymbol{J},\boldsymbol{q}-\boldsymbol{r}}(\boldsymbol{t})-\boldsymbol{b}_{\boldsymbol{J},\boldsymbol{q}-\boldsymbol{r}}(\boldsymbol{s})\|^2\left\|\left(\boldsymbol{B}^T\boldsymbol{B}
+\boldsymbol{\Omega}^{-1}\right)^{-1}\right\|_{(2,2)}\|\boldsymbol{W}^T_{\boldsymbol{r}}\boldsymbol{W}_{\boldsymbol{r}}\|_{(2,2)}\\
&\qquad\lesssim\frac{1}{n}\left(\prod_{k=1}^dJ_k^{2r_k+1}\right)\left(\sum_{k=1}^dJ_k^2\right)\|\boldsymbol{t}-\boldsymbol{s}\|^2
\lesssim \|\boldsymbol{J}\|^2 \|\|\boldsymbol{t}-\boldsymbol{s}\|^2,
\end{align*}
where we used Lemma \ref{lem:tsplinediff}, equations \eqref{eq:teigen} and \eqref{eq:tww} to bound the three norms respectively. Similarly, $\mathrm{Var}[Q_{n,\boldsymbol{r}}(\boldsymbol{t})-Q_{n,\boldsymbol{r}}(\boldsymbol{s})]$ is bounded by
\begin{align*}
\|\boldsymbol{b}_{\boldsymbol{J},\boldsymbol{q}-\boldsymbol{r}}(\boldsymbol{t})-\boldsymbol{b}_{\boldsymbol{J},\boldsymbol{q}-\boldsymbol{r}}(\boldsymbol{s})\|^2\left\|\left(\boldsymbol{B}^T\boldsymbol{B}
+\boldsymbol{\Omega}^{-1}\right)^{-1}\right\|_{(2,2)}^2\|\boldsymbol{B}^T\boldsymbol{B}\|_{(2,2)}\|\boldsymbol{W}^T_{\boldsymbol{r}}\boldsymbol{W}_{\boldsymbol{r}}\|_{(2,2)},
\end{align*}
which is $O(\|\boldsymbol{J}\|^2\|\boldsymbol{t}-\boldsymbol{s}\|^2)$, where we used Lemma \ref{lem:tsplinediff}, \eqref{eq:teigen}, \eqref{eq:tdiagonal} and \eqref{eq:tww} to bound the four norms respectively.
\end{proof}

\begin{lemma}\label{lem:l2boundr}
Let $f(\boldsymbol{x})=\boldsymbol{b}_{\boldsymbol{J},\boldsymbol{q}}(\boldsymbol{x})^T\boldsymbol{\theta}$ and $I_{j_1,\dotsc,j_d}=\prod_{k=1}^d[t_{k,j_k-q_k},t_{k,j_k}]$. Furthermore, let $f|_{I_{j_1,\dotsc,j_d}}$ be the restriction of $f$ onto $I_{j_1,\dotsc,j_d}$. Then there exists constant $C>0$ depending on $\boldsymbol{q}=(q_1,\dotsc,q_d)^T$ such that
\begin{equation*}
\|f|_{I_{j_1,\dotsc,j_d}}\|_\infty\leq C\prod_{k=1}^d(t_{k,j_k}-t_{k,j_k-q_k})^{-1/2}\|f|_{I_{j_1,\dotsc,j_d}}\|_2.
\end{equation*}
\end{lemma}

\begin{proof}
By equation (12.8) of Theorem 12.2 from \citep{lschumaker},
\begin{align*}
f(\boldsymbol{x})|_{I_{j_1,\dotsc,j_d}}&=\sum_{m_1=1}^{J_1}\dotsi\sum_{m_d=1}^{J_d}\theta_{m_1,\dotsc,m_d}\prod_{k=1}^dB_{m_k,q_k}(x_k)|_{I_{j_1,\dotsc,j_d}}\\
&=\sum_{l_1=0}^{q_1-1}\dotsi\sum_{l_d=0}^{q_d-1}\alpha_{l_1,\dotsc,l_d}\prod_{k=1}^dx^{l_k}_k\quad\text{for $x_k\in[t_{k,j_k-q_k},t_{k,j_k}]$}.
\end{align*}
If $\boldsymbol{x}\in I_{j_1,\dotsc,j_d}$, then $\boldsymbol{x}\in\prod_{k=1}^d[t_{k,j_k-h_k-1},t_{k,j_k-h_k}]$ for some $h_k=0,1,\dotsc,q_k-1$. Therefore, this implies that only terms associated with coefficients $\boldsymbol{\gamma}=\{\theta_{m_1,\dotsc,m_d}:j_k-h_k\leq m_k\leq j_k-h_k+q_k-1,k=1,\dotsc,d\}$ will be nonzero. Furthermore, we define $\boldsymbol{\alpha}=\{\alpha_{l_1,\dotsc,l_d}:0\leq l_k\leq q_k-1,k=1,\dotsc,d\}$. The two equivalent representations of $f$ on $I_{j_i,\dotsc,j_d}$ above implies a one-to-one mapping between $\boldsymbol{\gamma}$ and $\boldsymbol{\alpha}$, i.e., each element of $\boldsymbol{\alpha}$ is a linear combination of elements in $\boldsymbol{\gamma}$ and vice-versa. Hence, there are matrices $\boldsymbol{T}$ and $\boldsymbol{V}$ of dimension $\prod_{k=1}^dq_k\times\prod_{k=1}^dq_k$ respectively, such that $\boldsymbol{T\gamma}=\boldsymbol{\alpha}$ and $\boldsymbol{V\alpha}=\boldsymbol{\gamma}$. Since these two linear transformations have entries and dimensions not depending on $n$, we have $\|\boldsymbol{T}\|_{(\infty,\infty)}=O(1)$ and $\|\boldsymbol{V}\|_{(\infty,\infty)}=O(1)$, with constants in $O(1)$ depending only on $\boldsymbol{q}$. Let $\boldsymbol{U}_{k,q_k}=(1,U_k,U^2_k,\dotsc,U_{k}^{q_k-1})^T$ where $U_k\sim\text{Uniform}(t_{k,j_k-q_k},t_{k,j_k}), k=1,\dotsc,d$. Therefore, $\|f|_{I_{j_1,\dotsc,j_d}}\|^2_2$ is
\begin{align*}
&\int_{I_{j_1,\dotsc,j_d}}\left(\sum_{l_1=0}^{q_1-1}\dotsi\sum_{l_d=0}^{q_d-1}\alpha_{l_1,\dotsc,l_d}\prod_{k=1}^dx_k^{l_k}\right)^2d\boldsymbol{x}\\
&=\sum_{l_1=0}^{q_1-1}\dotsi\sum_{l_d=0}^{q_d-1}\sum_{l^{'}_1=0}^{q_1-1}\dotsi\sum_{l^{'}_d=0}^{q_d-1}\alpha_{l_1,\dotsc,l_d}\alpha_{l^{'}_1,\dotsc,l^{'}_d}\prod_{k=1}^d\int_{[t_{k,j_k-q_k},t_{k,j_k}]}x_k^{l_k+l^{'}_k}dx_k\\
&\geq\prod_{k=1}^d(t_{k,j_k}-t_{k,j_k-q_k})\lambda_\text{min}\{\mathrm{E}(\boldsymbol{U}_{k,q_k}\boldsymbol{U}_{k,q_k}^T)\}\|\boldsymbol{\alpha}\|^2.
\end{align*}
Since $\mathrm{E}(\boldsymbol{U}_{k,q_k}\boldsymbol{U}_{k,q_k}^T)$ is nonsingular, its minimum eigenvalue is bounded below by a positive constant. Hence, $\lambda_\text{min}\{\mathrm{E}(\boldsymbol{U}_{k,q_k}\boldsymbol{U}_{k,q_k}^T)\}\|\boldsymbol{\alpha}\|^2\gtrsim\|\boldsymbol{\alpha}\|^2_\infty
\geq\|\boldsymbol{V}\|^{-2}_{(\infty,\infty)}\|\boldsymbol{\gamma}\|^2_\infty$. The lower bound is obtained by noting that $\|f|_{I_{j_1,\dotsc,j_d}}\|^2_\infty\leq\|\boldsymbol{\gamma}\|_\infty^2\|\sum_{m_1=1}^{J_1}\dotsi\sum_{m_d=1}^{J_d}\prod_{k=1}^dB_{m_k,q_k}(\cdot)|_{I_{j_1,\dotsc,j_d}}\|_\infty^2\leq\|\boldsymbol{\gamma}\|_\infty^2$.
\end{proof}

\begin{lemma}\label{lem:l2bound}
For $f(\boldsymbol{x})=\boldsymbol{b}_{\boldsymbol{J},\boldsymbol{q}}(\boldsymbol{x})^T\boldsymbol{\theta}$, we have
\begin{align*}
\|f\|^2_2\asymp\sum_{j_1=1}^{J_1}\dotsi\sum_{j_d=1}^{J_d}\theta_{j_1,\dotsc,j_d}^2\prod_{k=1}^d(t_{k,j_k}-t_{k,j_k-q_k}).
\end{align*}
\end{lemma}

\begin{proof}
Since $\boldsymbol{b}_{\boldsymbol{J},\boldsymbol{q}}(\boldsymbol{x})$ is a probability vector at any $\boldsymbol{x}$, we use Jensen's inequality to write
\begin{align*}
\int_{[0,1]^d}f(\boldsymbol{x})^2d\boldsymbol{x}&\leq\int_{[0,1]^d}\sum_{j_1=1}^{J_1}\dotsi\sum_{j_d=1}^{J_d}\theta_{j_1,\dotsc,j_d}^2\prod_{k=1}^dB_{j_k,q_k}(x_k)d\boldsymbol{x}\\
&\leq\sum_{j_1=1}^{J_1}\dotsi\sum_{j_d=1}^{J_d}\theta_{j_1,\dotsc,j_d}^2\int_{[0,1]^d}\prod_{k=1}^d\Ind_{(t_{k,j_k-q_k},t_{k,j_k})}(x_k)d\boldsymbol{x}\\
&=\sum_{j_1=1}^{J_1}\dotsi\sum_{j_d=1}^{J_d}\theta_{j_1,\dotsc,j_d}^2\prod_{k=1}^d(t_{k,j_k}-t_{k,j_k-q_k}).
\end{align*}
Using Lemma \ref{lem:l2boundr} and equation (5) of Chapter XI from \citep{deBoor}, $\|f\|^2_2$ is
\begin{align*}
\sum_{j_1=1}^{J_1}\dotsi\sum_{j_d=1}^{J_d}\|f|_{I_{j_1,\dotsc,j_d}}\|^2_2&\gtrsim\sum_{j_1=1}^{J_1}\dotsi\sum_{j_d=1}^{J_d}\prod_{k=1}^d(t_{k,j_k}-t_{k,j_k-q_k})\|f|_{I_{j_1,\dotsc,j_d}}\|^2_\infty\\
&\geq c\sum_{j_1=1}^{J_1}\dotsi\sum_{j_d=1}^{J_d}\theta^2_{j_1,\dotsc,j_d}\prod_{k=1}^d(t_{k,j_k}-t_{k,j_k-q_k}),
\end{align*}
where $c>0$ is a constant depending only on $\boldsymbol{q}=(q_1,\dotsc,q_d)^T$.
\end{proof}

The following is a multivariate generalization of Lemma 6.1 in \citep{localspline}.

\begin{lemma}\label{lem:matrix}
For quasi-uniform knots,   $\boldsymbol{\theta}^T\boldsymbol{B}^T\boldsymbol{B}\boldsymbol{\theta}\asymp n\left(\prod_{k=1}^dJ_k^{-1}\right)\|\boldsymbol{\theta}\|^2$ for any $\boldsymbol{\theta}\in\mathbb{R}^J$ if \eqref{assump:tbspline2} holds.
\end{lemma}

\begin{proof}
Let $f(\boldsymbol{x})=\boldsymbol{b}_{\boldsymbol{J},\boldsymbol{q}}(\boldsymbol{x})^T\boldsymbol{\theta}$ and
$\|f\|^2_{2,\nu}=\int_{[0,1]^d}f(\boldsymbol{x})^2d\nu$ for any sigma-finite measure $\nu$. Observe that $\|f\|^2_{2,G_n}=\boldsymbol{\theta}^T\boldsymbol{B}^T\boldsymbol{B\theta}/n$. If the density of $G$ lies between $K_{\min}$ and $K_{\max}$, then by the quasi-uniformity of the knots and Lemma \ref{lem:l2bound}, the upper bound for $\|f\|_{2,G}^2$ is
\begin{align}
\|f\|_{2,G}^2&\leq K_\text{max}\sum_{j_1=1}^{J_1}\dotsi\sum_{j_d=1}^{J_d}\theta_{j_1,\dotsc,j_d}^2\prod_{k=1}^d(t_{k,j_k}-t_{k,j_k-q_k})\lesssim
\|\boldsymbol{\theta}\|^2\prod_{k=1}^d\Delta_k,\label{eq:1stpart}
\end{align}
and for a constant $c>0$, the lower bound for $\|f\|_{2,G}^2$ is
\begin{align}
c^2K_\text{min}\sum_{j_1=1}^{J_1}\dotsi\sum_{j_d=1}^{J_d}\theta_{j_1,\dotsc,j_d}^2\prod_{k=1}^d(t_{k,j_k}-t_{k,j_k-q_k})\gtrsim \|\boldsymbol{\theta}\|^2\prod_{k=1}^d\Delta_k.\label{eq:2ndpart}
\end{align}
Noting that $(G_n-G)(\boldsymbol{1}_d)=(G_n-G)(\boldsymbol{0})=0$, we use multivariate integration by parts and \eqref{assump:tbspline2} to bound $|\int_{[0,1]^d}f(\boldsymbol{x})^2d(G_n-G)(\boldsymbol{x})|$ by
\begin{align}\label{eq:mitp}
&2\sup_{\boldsymbol{x}\in[0,1]^d}|G_n(\boldsymbol{x})-G(\boldsymbol{x})|\int_{[0,1]^d}\left|f(\boldsymbol{x})\frac{\partial^d f(\boldsymbol{x})}{\partial x_1\dotsm\partial x_d}\right|d\boldsymbol{x}\nonumber\\
&\qquad=o\left(\prod_{k=1}^dN_k^{-1}\right)\|f\|_2\left\|\frac{\partial^d f}{\partial x_1\dotsm\partial x_d}\right\|_2,
\end{align}
in view of the Cauchy-Schwarz inequality in the last line. From \eqref{eq:firstdiff}, we have that $\mathfrak{D}^{1}_{j_k}\theta_{j_1,\dotsc,j_k}=(q_k-1)\mathfrak{d}_{j_k}\theta_{j_1,\dotsc,j_d}(t_{k,j_k}-t_{k,j_k-q_k+1})^{-1}$, where $\mathfrak{d}_{j_k}\theta_{j_1,\dotsc,j_d}=\theta_{j_1,\dotsc,j_{k-1},j_k+1,j_{k+1},\dotsc,j_d}-\theta_{j_1,\dotsc,j_{k-1},j_k,j_{k+1},\dotsc,j_d}$. Let  $\mathfrak{d}\theta_{j_1,\dotsc,j_d}=\mathfrak{d}_{j_1}\dotsm\mathfrak{d}_{j_d}\theta_{j_1,\dotsc,j_d}$. By setting $\boldsymbol{r}=\boldsymbol{1}_d$ in \eqref{eq:derivspline},
\begin{align*}
\frac{\partial^d f(\boldsymbol{x})}{\partial x_1\dotsm\partial x_d}&=\sum_{j_1=1}^{J_1-1}\dotsi\sum_{j_d=1}^{J_d-1}\mathfrak{d}\theta_{j_1,\dotsc,j_d}\prod_{k=1}^d\frac{q_k-1}{t_{k,j_k}-t_{k,j_k-q_k+1}}B_{j_k,q_k-1}(x_k).
\end{align*}
Applying Lemma \ref{lem:l2bound} to $f$ and its derivatives,
\begin{align*}
\|f\|_2^2\leq\sum_{j_1=1}^{J_1}\dotsi\sum_{j_d=1}^{J_d}\theta_{j_1,\dotsc,j_d}^2\prod_{k=1}^d(t_{k,j_k}-t_{k,j_k-q_k})\lesssim\|\boldsymbol{\theta}\|^2\prod_{k=1}^d\Delta_k,
\end{align*}
and $\|\partial^d f/\partial x_1\dotsm\partial x_d\|_2^2$ is bounded by
\begin{align*}
\sum_{j_1=1}^{J_1-1}\dotsi\sum_{j_d=1}^{J_d-1}\left(\mathfrak{d}\theta_{j_1,\dotsc,j_d}\right)^2\prod_{k=1}^d\frac{(q_k-1)^2}{t_{k,j_k}-t_{k,j_k-q_k+1}}
\lesssim\|\boldsymbol{\theta}\|^2\prod_{k=1}^d\frac{1}{\min_{1\leq l\leq N_k}\delta_{k,l}},
\end{align*}
where the last inequality follows from $\sum_{j_1=1}^{J_1-1}\dotsi\sum_{j_d=1}^{J_d-1}(\mathfrak{d}\theta_{j_1,\dotsc,j_d})^2\leq2^{2d}\|\boldsymbol{\theta}\|^2$. By the quasi-uniformity of the knots, it follows that the right side of \eqref{eq:mitp} is $o(\prod_{k=1}^dN_k^{-1})\|\boldsymbol{\theta}\|^2$. Combining this result with \eqref{eq:1stpart} and using Lemma \ref{lem:tasymp},
\begin{align*}
\|f\|^2_{2,G_n}&\lesssim\|\boldsymbol{\theta}\|^2\prod_{k=1}^d\Delta_k+o\left(\prod_{k=1}^dN_k^{-1}\right)\|\boldsymbol{\theta}\|^2
\lesssim\left(\prod_{k=1}^dJ_k^{-1}\right)\|\boldsymbol{\theta}\|^2;
\end{align*}
while combining the same result with \eqref{eq:2ndpart} and in view of Lemma \ref{lem:tasymp},
\begin{align*}
\|f\|^2_{2,G_n}&\gtrsim
\|\boldsymbol{\theta}\|^2\prod_{k=1}^d\Delta_k-o\left(\prod_{k=1}^dN_k^{-1}\right)\|\boldsymbol{\theta}\|^2\gtrsim\left(\prod_{k=1}^dJ_k^{-1}\right)\|\boldsymbol{\theta}\|^2.\qedhere
\end{align*}
\end{proof}

\begin{lemma}\label{lem:tvarquad}
Let $\boldsymbol{A}$ be an $n\times n$ symmetric positive definite matrix. Assume that $\|\boldsymbol{A}\|_{(2,2)}\leq C$ for constant $C>0$. Let $\boldsymbol{\varepsilon}=(\varepsilon_1,\dotsc,\varepsilon_n)^T$ such that $\varepsilon_i$ are i.i.d. mean $\mathrm{0}$, variance $\sigma_0^2$ with finite fourth moment for $i=1,\dotsc,n$. Then $\mathrm{Var}(\boldsymbol{\varepsilon}^T\boldsymbol{A\varepsilon})=O(n)$.
\end{lemma}
\begin{proof}
By eigendecomposition, $\boldsymbol{A}=\boldsymbol{P}^T\boldsymbol{\Lambda P}$ where $\boldsymbol{\Lambda}=\mathrm{diag}(\lambda_1,\dotsc,\lambda_n)$ and $\boldsymbol{P}=(\!(p_{ij})\!)$ is an orthogonal matrix. Let  $\boldsymbol{Z}=(Z_1,\dotsc,Z_n)^T=\boldsymbol{P\varepsilon}$. Then $\mathrm{Var}(\boldsymbol{\varepsilon}^T\boldsymbol{A\varepsilon})=
\sum_{i=1}^n\lambda_i^2\mathrm{Var}(Z_i^2)
+\sum_{r\neq s}^n\lambda_r\lambda_s\mathrm{Cov}(Z_r^2,Z_s^2)$, and  $$\mathrm{E}(Z_r^2Z_s^2)=
\mathrm{E}(\varepsilon_1^4)\sum_{j=1}^np_{rj}^2p_{sj}^2+\sigma_0^4\sum_{j_1\neq j_2}^np_{rj_1}^2p_{sj_2}^2
+2\sigma_0^4\sum_{j_1\neq j_2}^np_{rj_1}p_{sj_1}p_{rj_2}p_{sj_2}.
$$
Therefore $\sum_{i=1}^n\lambda_i^2\mathrm{Var}(Z_i^2)\lesssim\sum_{i=1}^n\mathrm{E}(Z_i^4)\lesssim\sum_{i=1}^n(\sum_{j=1}^np_{ij}^2)^2\lesssim n$
since $\lambda_i\leq C,i=1,\dotsc,n$, and each row of $\boldsymbol{P}$ has unit norm. By the orthonormality of $\boldsymbol{P}$, $\mathrm{Var}(\boldsymbol{Z})=\sigma_0^2\boldsymbol{I}_n$ and $\mathrm{E}(Z_i^2)=\sigma_0^2$. Observing that for $r\neq s$,
$\sum_{j=1}^np_{rj}p_{sj}=0$, $\sum_{j_1\neq j_2}^np_{rj_1}p_{sj_1}p_{rj_2}p_{sj_2}=(\sum_{j=1}^np_{rj}p_{sj})^2-\sum_{j=1}^np_{rj}^2p_{sj}^2\leq0$ and $ \sum_{j_1,j_2=1}^n p_{rj_1}^2p_{rj_2}^2=\sum_{j=1}^np_{rj}^4+\sum_{j_1\neq j_2}^np_{rj_1}^2p_{rj_2}^2=1$. Hence using the last display $\mathrm{Cov}(Z_r^2,Z_s^2)$ is bounded by
\begin{align*}
\mathrm{E}(\varepsilon_1^4)\sum_{j=1}^np_{sj}^2p_{rj}^2
+\sigma_0^4\left(1-\sum_{j=1}^np_{rj}^4\right)-\mathrm{E}(Z_r^2)\mathrm{E}(Z_s^2)\leq\mathrm{E}(\varepsilon_1^4)\sum_{j=1}^np_{sj}^2p_{rj}^2.
\end{align*}
Therefore, $\sum_{r\neq s}^n\lambda_r\lambda_s\mathrm{Cov}(Z_r^2,Z_s^2)\lesssim\mathrm{E}(\varepsilon_1^4)\sum_{j=1}^n\sum_{r=1}^n\sum_{s=1}^np_{sj}^2p_{rj}^2\lesssim n$.
\end{proof}

\begin{lemma}\label{lem:supexpect}
Let $\{X(\boldsymbol{t}):\boldsymbol{t}\in[0,1]^d\}$ be a sub-Gaussian process with respect to the semi-metric $d(\boldsymbol{t},\boldsymbol{s})=\sqrt{\mathrm{Var}[X(\boldsymbol{t})-X(\boldsymbol{s})]}$ such that $d^2(\boldsymbol{t},\boldsymbol{s})\lesssim C(n)\|\boldsymbol{t}-\boldsymbol{s}\|^2$ for any $\boldsymbol{t},\boldsymbol{s}\in[0,1]^d$, where $C(n)$ is a polynomial in $n$. Choose points $\boldsymbol{u}=\{\boldsymbol{u}_1,\dotsc,\boldsymbol{u}_{T_n}\}$ in $[0,1]^d$ such that $\bigcup_{i=1}^{T_n}\{\boldsymbol{z}:\|\boldsymbol{z}-\boldsymbol{u}_i\|\leq\delta_n\}\supseteq[0,1]^d$, for some sequence $\delta_n\rightarrow0$ as $n\rightarrow\infty$ with $\delta_n<1$, and $T_n\le (2/\delta_n)^d$. Then for $1\leq p<\infty$,
\begin{align*}
\mathrm{E}\|X\|^p_\infty&\lesssim\{\log{(1/\delta_n)}\}^{p/2}\left\{\left(\delta_n\sqrt{C(n)}\right)^p\right.\\
&\qquad\left.+\max_{1\leq i\leq T_n}|\mathrm{E}[X(\boldsymbol{u}_i)]|^p+\max_{1\leq i\leq T_n}\{\mathrm{Var}[X(\boldsymbol{u}_i)]\}^{p/2}\right\}.
\end{align*}
\end{lemma}
\begin{proof}
It suffices to bound the $L_p$-norms of the expected process increment and the maximum of the process at $\boldsymbol{u}$. Since $X(\boldsymbol{t})$ is sub-Gaussian and $\mathrm{Var}[X(\boldsymbol{t})-X(\boldsymbol{s})]\lesssim C(n)\|\boldsymbol{t}-\boldsymbol{s}\|^2$ by assumption, we can relate the $\psi_2$-Orlicz norm of the process increment with $\|\boldsymbol{t}-\boldsymbol{s}\|$ by Section 2.2.1 of \citep{empirical}. We then bound the expected process increment by Corollary 2.2.8 of \citep{empirical} with $d(\boldsymbol{t},\boldsymbol{s})=\|\boldsymbol{t}-\boldsymbol{s}\|$. The expected maximum of the process at $\boldsymbol{u}$ is then bounded using Lemma 2.2.2 of \citep{empirical}.
\end{proof}

{\bf Acknowledgment.} The authors like to thank the referees and the associate editors for suggestions which led to significant improvements of the paper and for directing to several relevant references.

\bibliographystyle{imsart-number}
\bibliography{reference}

\begin{thebibliography}{36}

\bibitem{bickel1973}
\begin{barticle}[author]
\bauthor{\bsnm{Bickel},~\bfnm{P.}\binits{P.}} \AND
  \bauthor{\bsnm{Rosenblatt},~\bfnm{M.}\binits{M.}}
(\byear{1973}).
\btitle{On some global measures of the deviations of density function
  estimates}.
\bjournal{Ann. Statist.}
\bvolume{1}
\bpages{1071--1095}.
\bnote{Correction (1975) \textbf{3} 1370}.
\end{barticle}
\endbibitem

\bibitem{castillosupnorm}
\begin{barticle}[author]
\bauthor{\bsnm{Castillo},~\bfnm{I.}\binits{I.}}
(\byear{2014}).
\btitle{On {Bayesian} supremum norm contraction rates}.
\bjournal{Ann. Statist.}
\bvolume{42}
\bpages{2058--2091}.
\end{barticle}
\endbibitem

\bibitem{castillo2013}
\begin{barticle}[author]
\bauthor{\bsnm{Castillo},~\bfnm{I.}\binits{I.}} \AND
  \bauthor{\bsnm{Nickl},~\bfnm{R.}\binits{R.}}
(\byear{2013}).
\btitle{Nonparametric {Bernstein-von Mises} theorems in {Gaussian} white
  noise}.
\bjournal{Ann. Statist.}
\bvolume{41}
\bpages{1999--2028}.
\end{barticle}
\endbibitem

\bibitem{nickl2014}
\begin{barticle}[author]
\bauthor{\bsnm{Castillo},~\bfnm{Isma\"{e}l}\binits{I.}} \AND
  \bauthor{\bsnm{Nickl},~\bfnm{Richard}\binits{R.}}
(\byear{2014}).
\btitle{On the Bernstein-von Mises phenomenon for nonparametric Bayes
  procedures}.
\bjournal{Ann. Statist.}
\bvolume{42}
\bpages{1941--1969}.
\end{barticle}
\endbibitem

\bibitem{gaussbootstrap2013}
\begin{barticle}[author]
\bauthor{\bsnm{Chernozhukov},~\bfnm{Victor}\binits{V.}},
  \bauthor{\bsnm{Chetverikov},~\bfnm{Denis}\binits{D.}} \AND
  \bauthor{\bsnm{Kato},~\bfnm{Kengo}\binits{K.}}
(\byear{2014}).
\btitle{Anti-concentration and honest, adaptive confidence bands}.
\bjournal{Ann. Statist.}
\bvolume{42}
\bpages{1787--1818}.
\end{barticle}
\endbibitem

\bibitem{bootstrap2003}
\begin{barticle}[author]
\bauthor{\bsnm{Claeskens},~\bfnm{Gerda}\binits{G.}} \AND
  \bauthor{\bparticle{van} \bsnm{Keilegom},~\bfnm{Ingrid}\binits{I.}}
(\byear{2003}).
\btitle{Bootstrap confidence bands for regression curves and their
  derivatives}.
\bjournal{Ann. Statist.}
\bvolume{31}
\bpages{1852--1884}.
\end{barticle}
\endbibitem

\bibitem{cox1993}
\begin{barticle}[author]
\bauthor{\bsnm{Cox},~\bfnm{D.~D.}\binits{D.~D.}}
(\byear{1993}).
\btitle{An analysis of {Bayesian} inference for nonparametric regression}.
\bjournal{Ann. Statist.}
\bvolume{21}
\bpages{903--923}.
\end{barticle}
\endbibitem

\bibitem{deBoor}
\begin{bbook}[author]
\bauthor{\bparticle{de} \bsnm{Boor},~\bfnm{C.}\binits{C.}}
(\byear{2001}).
\btitle{A Practical Guide to Splines},
\bedition{Revised} ed.
\bpublisher{Springer-Verlag New York, Inc.}
\end{bbook}
\endbibitem

\bibitem{dejonge2012}
\begin{barticle}[author]
\bauthor{\bparticle{de} \bsnm{Jonge},~\bfnm{R.}\binits{R.}} \AND
  \bauthor{\bparticle{van} \bsnm{Zanten},~\bfnm{J.~H.}\binits{J.~H.}}
(\byear{2012}).
\btitle{Adaptive estimation of multivariate functions using conditionally
  {Gaussian} tensor-product spline priors}.
\bjournal{Electron. J. Stat.}
\bvolume{6}
\bpages{1984--2001}.
\end{barticle}
\endbibitem

\bibitem{dejonge2013}
\begin{barticle}[author]
\bauthor{\bparticle{de} \bsnm{Jonge},~\bfnm{Ren\'{e}}\binits{R.}} \AND
  \bauthor{\bparticle{van} \bsnm{Zanten},~\bfnm{J.~H.}\binits{J.~H.}}
(\byear{2013}).
\btitle{Semiparametric Bernstein-von Misses for the error standard devation}.
\bjournal{Electron. J. Stat.}
\bvolume{7}
\bpages{217--243}.
\end{barticle}
\endbibitem

\bibitem{bandedmatrix}
\begin{barticle}[author]
\bauthor{\bsnm{Demko},~\bfnm{S.}\binits{S.}},
  \bauthor{\bsnm{Moss},~\bfnm{W.~F.}\binits{W.~F.}} \AND
  \bauthor{\bsnm{Smith},~\bfnm{P.~W.}\binits{P.~W.}}
(\byear{1984}).
\btitle{Decay rates for inverses of band matrices}.
\bjournal{Math. Comp.}
\bvolume{43}
\bpages{491--499}.
\end{barticle}
\endbibitem

\bibitem{freedman1999}
\begin{barticle}[author]
\bauthor{\bsnm{Freedman},~\bfnm{D.}\binits{D.}}
(\byear{1999}).
\btitle{On the {Bernstein-von Mises} theorem with infinite dimensional
  parameters}.
\bjournal{Ann. Statist.}
\bvolume{27}
\bpages{1119--1140}.
\end{barticle}
\endbibitem

\bibitem{nickl2010}
\begin{barticle}[author]
\bauthor{\bsnm{Gin\'{e}},~\bfnm{Evarist}\binits{E.}} \AND
  \bauthor{\bsnm{Nickl},~\bfnm{Richard}\binits{R.}}
(\byear{2010}).
\btitle{Confidence bands in density estimation}.
\bjournal{Ann. Statist.}
\bvolume{38}
\bpages{1122--1170}.
\end{barticle}
\endbibitem

\bibitem{nickl2011}
\begin{barticle}[author]
\bauthor{\bsnm{Gin\'{e}},~\bfnm{E.}\binits{E.}} \AND
  \bauthor{\bsnm{Nickl},~\bfnm{R.}\binits{R.}}
(\byear{2011}).
\btitle{Rates of contraction for posterior distributions in {$L^r$}-metrics,
  $1\leq r\leq\infty$}.
\bjournal{Ann. Statist.}
\bvolume{39}
\bpages{2883--2911}.
\end{barticle}
\endbibitem

\bibitem{davidmatrix}
\begin{bbook}[author]
\bauthor{\bsnm{Harville},~\bfnm{D.~A.}\binits{D.~A.}}
(\byear{1997}).
\btitle{Matrix Algebra from a Statistician's Perspective}.
\bpublisher{Springer-Verlag New York, Inc}.
\end{bbook}
\endbibitem

\bibitem{hoffmann}
\begin{barticle}[author]
\bauthor{\bsnm{Hoffmann},~\bfnm{M.}\binits{M.}} \AND
  \bauthor{\bsnm{Lepski},~\bfnm{O.}\binits{O.}}
(\byear{2002}).
\btitle{Random rates in anisotropic regression}.
\bjournal{Ann. Statist.}
\bvolume{30}
\bpages{325--396}.
\end{barticle}
\endbibitem

\bibitem{adaptsupnorm}
\begin{barticle}[author]
\bauthor{\bsnm{Hoffmann},~\bfnm{M.}\binits{M.}},
  \bauthor{\bsnm{Rousseau},~\bfnm{J.}\binits{J.}} \AND
  \bauthor{\bsnm{Schmidt-Hieber},~\bfnm{J.}\binits{J.}}
(\byear{2015}).
\btitle{On adaptive posterior concentration rates}.
\bjournal{Ann. Statist.}
\bnote{To appear}.
\end{barticle}
\endbibitem

\bibitem{inverseprob}
\begin{barticle}[author]
\bauthor{\bsnm{Knapik},~\bfnm{B.~T.}\binits{B.~T.}},
  \bauthor{\bparticle{van~der} \bsnm{Vaart},~\bfnm{A.~W.}\binits{A.~W.}} \AND
  \bauthor{\bparticle{van} \bsnm{Zanten},~\bfnm{J.~H.}\binits{J.~H.}}
(\byear{2011}).
\btitle{{Bayesian} inverse problems with {Gaussian} priors}.
\bjournal{Ann. Statist.}
\bvolume{39}
\bpages{2626--2657}.
\end{barticle}
\endbibitem

\bibitem{inverseprob2013}
\begin{barticle}[author]
\bauthor{\bsnm{Knapik},~\bfnm{B.~T.}\binits{B.~T.}},
  \bauthor{\bparticle{van~der} \bsnm{Vaart},~\bfnm{A.~W.}\binits{A.~W.}} \AND
  \bauthor{\bparticle{van} \bsnm{Zanten},~\bfnm{J.~H.}\binits{J.~H.}}
(\byear{2013}).
\btitle{Bayesian recovery of the initial condition for the heat equation}.
\bjournal{Comm. Statist. Theory Methods}
\bvolume{42}
\bpages{1294--1313}.
\end{barticle}
\endbibitem

\bibitem{leahu2011}
\begin{barticle}[author]
\bauthor{\bsnm{Leahu},~\bfnm{H.}\binits{H.}}
(\byear{2011}).
\btitle{On the {Bernstein-von Mises} phenomenon in the {Gaussian} white noise
  model}.
\bjournal{Electron. J. Stat.}
\bvolume{5}
\bpages{373--404}.
\end{barticle}
\endbibitem

\bibitem{talagrand}
\begin{bbook}[author]
\bauthor{\bsnm{Ledoux},~\bfnm{M.}\binits{M.}} \AND
  \bauthor{\bsnm{Talagrand},~\bfnm{M.}\binits{M.}}
(\byear{1991}).
\btitle{Probability in Banach Spaces}.
\bpublisher{Springer-Verlag Berlin Heidelberg}.
\end{bbook}
\endbibitem

\bibitem{aniappli}
\begin{barticle}[author]
\bauthor{\bsnm{Neumann},~\bfnm{Michael~H.}\binits{M.~H.}} \AND
  \bauthor{\bparticle{von} \bsnm{Sachs},~\bfnm{Rainer}\binits{R.}}
(\byear{1997}).
\btitle{Wavelet thresholding in anisotropic function classes and application to
  adaptive estimation of evolutionary spectra}.
\bjournal{Ann. Statist.}
\bvolume{25}
\bpages{38--76}.
\end{barticle}
\endbibitem

\bibitem{kolyan2015}
\begin{barticle}[author]
\bauthor{\bsnm{Ray},~\bfnm{Kolyan}\binits{K.}}
(\byear{2015}).
\btitle{Adaptive Bernstein-von Mises theorems in Gaussian white noise}.
\bnote{Preprint arXiv:1407.3397v2 [math.ST]}.
\end{barticle}
\endbibitem

\bibitem{lschumaker}
\begin{bbook}[author]
\bauthor{\bsnm{Schumaker},~\bfnm{L.}\binits{L.}}
(\byear{2007}).
\btitle{Spline Functions: Basic Theory},
\bedition{Third} ed.
\bpublisher{Cambridge University Press, New York}.
\end{bbook}
\endbibitem

\bibitem{scricciolo2014}
\begin{barticle}[author]
\bauthor{\bsnm{Scricciolo},~\bfnm{C.}\binits{C.}}
(\byear{2014}).
\btitle{Adaptive {Bayesian} density estimation in {$L^p$}-metrics with
  {Pitman-Yor} or normalized inverse-{Gaussian} process kernel mixtures}.
\bjournal{Bayesian Analysis}
\bvolume{9}
\bpages{475--520}.
\end{barticle}
\endbibitem

\bibitem{expectquad}
\begin{bbook}[author]
\bauthor{\bsnm{Searle},~\bfnm{S.~R.}\binits{S.~R.}}
(\byear{1982}).
\btitle{Matrix Algebra Useful for Statistics}.
\bpublisher{John Wiley and Sons, Inc.}
\end{bbook}
\endbibitem

\bibitem{smooth}
\begin{barticle}[author]
\bauthor{\bsnm{Serra},~\bfnm{Paulo}\binits{P.}} \AND
  \bauthor{\bsnm{Krivobokova},~\bfnm{Tatyana}\binits{T.}}
(\byear{2014}).
\btitle{Adaptive empirical Bayesian smoothing splines}.
\bnote{Preprint arXiv:1411.6860 [math.ST]}.
\end{barticle}
\endbibitem

\bibitem{ShenandGhosal:RandomSeries}
\begin{barticle}[author]
\bauthor{\bsnm{Shen},~\bfnm{Weining}\binits{W.}} \AND
  \bauthor{\bsnm{Ghosal},~\bfnm{Subhashis}\binits{S.}}
(\byear{2014}).
\btitle{Adaptive Bayesian procedures using random series priors}.
\bjournal{Scand. J. Statist.}
\bnote{To appear}.
\end{barticle}
\endbibitem

\bibitem{ShenandGhosal:DensityRegression}
\begin{barticle}[author]
\bauthor{\bsnm{Shen},~\bfnm{Weining}\binits{W.}} \AND
  \bauthor{\bsnm{Ghosal},~\bfnm{Subhashis}\binits{S.}}
(\byear{2014}).
\btitle{Adaptive Bayesian density regression for high-dimensional data}.
\bjournal{Bernoulli}.
\bnote{To appear}.
\end{barticle}
\endbibitem

\bibitem{smirnov1950}
\begin{barticle}[author]
\bauthor{\bsnm{Smirnov},~\bfnm{N.~V.}\binits{N.~V.}}
(\byear{1950}).
\btitle{On the construction of confidence regions for the density of
  distribution of random variables}.
\bjournal{Doklady Akad. Nauk SSSR}
\bvolume{74}
\bpages{189--191}.
\end{barticle}
\endbibitem

\bibitem{credible}
\begin{barticle}[author]
\bauthor{\bsnm{Sniekers},~\bfnm{Suzanne}\binits{S.}} \AND
  \bauthor{\bparticle{van~der} \bsnm{Vaart},~\bfnm{A.~W.}\binits{A.~W.}}
(\byear{2015}).
\btitle{Credible sets in the fixed design model with Brownian motion prior}.
\bjournal{J. Statist. Plann. Inference}
\bvolume{166}
\bpages{78--86}.
\end{barticle}
\endbibitem

\bibitem{stone1980}
\begin{barticle}[author]
\bauthor{\bsnm{Stone},~\bfnm{C.~J.}\binits{C.~J.}}
(\byear{1980}).
\btitle{Optimal rates of convergence for nonparametric estimators}.
\bjournal{Ann. Statist.}
\bvolume{8}
\bpages{1348--1360}.
\end{barticle}
\endbibitem

\bibitem{stone1982}
\begin{barticle}[author]
\bauthor{\bsnm{Stone},~\bfnm{C.~J.}\binits{C.~J.}}
(\byear{1982}).
\btitle{Optimal global rates of convergence for nonparametric regression}.
\bjournal{Ann. Statist.}
\bvolume{10}
\bpages{1040--1053}.
\end{barticle}
\endbibitem

\bibitem{botond2015}
\begin{barticle}[author]
\bauthor{\bsnm{Szabo},~\bfnm{B.}\binits{B.}}, \bauthor{\bparticle{van~der}
  \bsnm{Vaart},~\bfnm{A.~W.}\binits{A.~W.}} \AND \bauthor{\bparticle{van}
  \bsnm{Zanten},~\bfnm{J.~H.}\binits{J.~H.}}
(\byear{2015}).
\btitle{Frequentist coverage of adaptive nonparametric {Bayesian} credible
  sets}.
\bjournal{Ann. Statist.}
\bvolume{43}
\bpages{1391--1428}.
\end{barticle}
\endbibitem

\bibitem{empirical}
\begin{bbook}[author]
\bauthor{\bparticle{van~der} \bsnm{Vaart},~\bfnm{A.~W.}\binits{A.~W.}} \AND
  \bauthor{\bsnm{Wellner},~\bfnm{J.~A.}\binits{J.~A.}}
(\byear{1996}).
\btitle{Weak Convergence and Empirical Process With Applications to
  Statistics}.
\bpublisher{Springer-Verlag New York, Inc.}
\end{bbook}
\endbibitem

\bibitem{localspline}
\begin{barticle}[author]
\bauthor{\bsnm{Zhou},~\bfnm{S.}\binits{S.}},
  \bauthor{\bsnm{Shen},~\bfnm{X.}\binits{X.}} \AND
  \bauthor{\bsnm{Wolfe},~\bfnm{D.~A.}\binits{D.~A.}}
(\byear{1998}).
\btitle{Local Asymptotics for Regression Splines and Confidence Regions}.
\bjournal{Ann. Statist.}
\bvolume{26}
\bpages{1760--1782}.
\end{barticle}
\endbibitem

\end{thebibliography}
\end{document}